\documentclass[noinfoline]{imsart}

\usepackage[a4paper,top=3cm, bottom=3cm, left=2.5cm, right=2.5cm]{geometry}
\usepackage{fancyhdr}

\RequirePackage{amsthm,amsmath,amsfonts,amssymb}

\allowdisplaybreaks
\RequirePackage[colorlinks,citecolor=blue,urlcolor=blue]{hyperref}
\usepackage[capitalize]{cleveref}
\usepackage{mathrsfs} 

\usepackage{graphicx}
\graphicspath{{Images/}}
\usepackage{float}
\usepackage[figuresright]{rotating} 
\usepackage{subcaption}
\usepackage{pdfpages}
\usepackage{wrapfig}

\usepackage{array}
\usepackage{booktabs} 
\usepackage{enumerate}
\usepackage{enumitem} 
\usepackage{multicol}
\usepackage{tikz}
\usepackage{algorithm}
\usepackage[noend]{algpseudocode}

\makeatletter
\def\algbackskip{\hskip-\ALG@thistlm}
\makeatother

\renewcommand{\b}[1]{\mathbb{#1}} 
\newcommand{\m}[1]{\mathbf{#1}} 
\newcommand{\s}[1]{\mathscr{#1}} 
\renewcommand{\c}[1]{\mathcal{#1}} 

\newcommand{\innpr}[2]{{\left\langle {#1},{#2}\right\rangle}} 

\newcommand{\rd}{\mathrm{d}} 

\DeclareMathOperator*{\spann}{span}
\DeclareMathOperator*{\tr}{tr}
\DeclareMathOperator*{\diag}{diag}

\DeclareMathOperator*{\Cov}{Cov}

\DeclareMathOperator*{\rk}{rank}

\DeclareMathOperator*{\argmin}{\arg\min}
\DeclareMathOperator*{\argmax}{\arg\max}

\newcommand{\vertiii}[1]{{\left\vert\kern-0.25ex\left\vert\kern-0.25ex\left\vert #1 
    \right\vert\kern-0.25ex\right\vert\kern-0.25ex\right\vert}} 

\newcommand{\vertii}[1]{{\vert\kern-0.25ex\vert\kern-0.25ex\vert #1 
    \vert\kern-0.25ex\vert\kern-0.25ex\vert}} 

\theoremstyle{plain}
\newtheorem{theorem}{Theorem}[section]
\newtheorem{lemma}[theorem]{Lemma}
\newtheorem{corollary}[theorem]{Corollary}
\newtheorem{proposition}[theorem]{Proposition}

\newtheorem{definition}{Definition}[section]
\newtheorem{example}{Example}[section]
\newtheorem{remark}{Remark}[section]

\begin{document}

\pagestyle{fancy}
\fancyhead{}
\renewcommand{\headrulewidth}{0pt} 
\fancyhead[CE]{Ho Yun \& Yoav Zemel}
\fancyhead[CO]{Gaussian Optimal Transport Beyond Brenier's Theorem}

\begin{frontmatter}
\title{{\large Gaussian Optimal Transport Beyond Brenier's Theorem}}

\begin{aug}
\author[A]{\fnms{Ho} \snm{Yun}\ead[label=e1,mark]{ho.yun@epfl.ch}}
\and
\author[A]{\fnms{Yoav} \snm{Zemel}\ead[label=e2]{yoav.zemel@epfl.ch}}


\thankstext{t1}{Research supported by a Swiss National Science Foundation grant.}
\address[A]{Ecole Polytechnique F\'ed\'erale de Lausanne, \printead{e1,e2}}
\end{aug}

\begin{abstract}
We explore the geometry of the Bures-Wasserstein space for potentially degenerate Gaussian measures on a separable Hilbert space. In this general setting, the optimal transport map is formally the subgradient of a convex function that is infinite almost everywhere, rendering conventional duality-based variational methods ineffective. We overcome this analytical barrier by exploiting a constructive operator-theoretic approach. Our central result proves that the Kantorovich problem for any pair of Gaussian measures reduces to a Monge problem; that is, an optimal transport map exists in at least one direction between two measures. This reduction allows for a complete characterization with explicit formulas for all optimal (potentially unbounded) Monge transport map and Kantorovich couplings, as well as establishing their uniqueness. Furthermore, we provide a full description of the convex set of geodesics between degenerate measures, revealing a rich geometric structure where the classical McCann interpolants arise as the extreme points. We apply these findings to construct transport maps for Gaussian processes and introduce a novel framework for Wasserstein barycenters based on random Green's operators.
\end{abstract}


\end{frontmatter}

\maketitle

\section{Introduction}
Optimal Transport (OT) has evolved into a central pillar of modern analysis, bridging probability, geometry, and partial differential equations. At the heart of this theory lies the celebrated theorem of Brenier, which asserts that for the quadratic cost, the unique optimal transport map between two probability measures on $\b{R}^{n}$ (with the source being absolutely continuous) is the gradient of a convex function \cite{villani2008optimal, villani2021topics, ambrosio2008gradient}. When restricted to the space of Gaussian measures, this geometric framework specializes to the Bures-Wasserstein geometry. Here, the transport problem between two centered Gaussian measures, $\mu=N(\m{0}, \m{A})$ and $\nu=N(\m{0}, \m{B})$, admits a symmetric positive-definite matrix $\m{T} = \m{A}^{-1/2} (\m{A}^{1/2} \m{B} \m{A}^{1/2})^{1/2} \m{A}^{-1/2}$ as the unique optimal pushforward, provided that the source covariance $\m{A}$ is non-degenerate \cite{knott1984optimal, ruschendorf1990characterization, pusz1975functional, bhatia2019bures}.
However, the classical theory encounters a fundamental barrier when the measures become degenerate or infinite-dimensional. In these regimes -- common in the study of Gaussian processes and functional data analysis \cite{dryden2016statistical, pigoli2014distances, masarotto2019procrustes, kroshnin2021statistical} -- the standard tools of convex analysis fracture \cite{rockafellar1970convex}:
\begin{enumerate}[leftmargin = *]
\item \textbf{Singularity:} When the source covariance $\m{A}$ is singular (rank-deficient), the map $\m{T}$ is no longer uniquely defined by the gradient of a convex function almost everywhere. The transport plan may split mass, and the classical duality arguments become intractable.
\item \textbf{Unboundedness:} In infinite-dimensional Hilbert spaces, optimal maps between Gaussian measures often manifest as unbounded operators (e.g., differentiation). The analysis of a pushforward becomes delicate, as the domain of the linear map may have measure zero with respect to the source Gaussian measure.
\end{enumerate}

\subsection{Contribution}
We establish a comprehensive framework for Gaussian Optimal Transport that investigates the geometry of the Bures-Wasserstein space $(\c{B}_{1}^{+}(\c{H}), \c{W}_{2})$. This space consists of trace-class covariance operators on a separable Hilbert space equipped with the 2-Wasserstein distance \cite{dowson1982frechet, olkin1982distance, cuesta1996lower, pigoli2014distances, masarotto2019procrustes}:
\begin{equation*}
    \c{W}_{2}^{2}(\m{A}, \m{B}) = \tr(\m{A}) + \tr(\m{B}) - 2 \tr (\m{A}^{1/2} \m{B} \m{A}^{1/2})^{1/2}.
\end{equation*}
Our work focuses on optimal transport between centered \emph{Gaussian} measures, with particular emphasis on the challenges posed by \emph{degeneracy} and \emph{infinite-dimensionality}. In these settings, the transport map is formally the subgradient of a convex function that is infinite almost everywhere, rendering conventional approaches rooted in Kantorovich duality or Otto's calculus intractable.

To circumvent the limitations of Brenier’s theorem, we adopt a constructive approach rooted in operator theory \cite{conway2019course, hall2013quantum, reed1980methods, riesz2012functional} and probability theory \cite{kallenberg1997foundations, billingsley2013convergence, bogachev1998gaussian}. Our analysis centers on the concept of \emph{linear reachability}. In the context of Gaussian optimal transport problem, we demonstrate that the existence of an optimal transport map is governed not by absolute continuity, rather by the dimensional inequality related to the covariance operators. By utilizing the Douglas factorization lemma and the Hilbertian Schur complement, we provide explicit solutions to the Monge and Kantorovich problems ``beyond Brenier.''
The primary tools developed to achieve these results are (i) Green's operators, (ii) the Hilbertian Schur complement, and (iii) admissible pre-pushforwards, in connection to (i) statistical shape analysis, (ii) conditional independence, and (iii) abstract Wiener spaces, respectively. Many of our results are new even in the finite-dimensional setting. Our main contributions are summarized as follows:

\paragraph{Reduction of Kantorovich to Monge} We prove that for any two centered Gaussian measures on a separable Hilbert space, the linear pre-pushforward map for the Monge problem exists if and only if the following dimensional inequality holds:
\begin{align*}
    \dim [\c{R}(\m{B}^{1/2}) \cap \overline{\c{R}(\m{A})}] \le \dim [\c{N}(\m{B}^{1/2} \m{A}^{1/2}) \cap \c{N}(\m{A})].
\end{align*}
We demonstrate that for any pair of covariances, this condition (or the same inequality with $\m{A}$ and $\m{B}$ swapped) is always satisfied. When $\c{H} = \b{R}^{n}$, the rank-nullity theorem simplifies the dimensional inequality into $\rk (\m{A}) \ge \rk (\m{B})$. Consequently, an optimal transport map always exists in at least one direction as an extension of linear pre-pushfoward map, allowing the Kantorovich problem to be reduced to a Monge problem. This result effectively completes the transport picture for degenerate Gaussians and extends to general non-centered Gaussian measures via translation \cite{panaretos2020invitation}.

\paragraph{Characterization of Transport Maps} 
To carefully address the subtleties of unbounded operators in the infinite-dimensional regime, we introduce the notion of optimal pre-pushforwards defined on dense subspaces -- specifically, the reproducing kernel Hilbert space $\c{R}(\m{A}^{1/2})$ \cite{da2006introduction, paulsen2016introduction}). While this linear map itself may not be a valid transport map on the entire space, it can always be extended to yield an optimal transport map for the Monge problem. We derive necessary and sufficient conditions for the existence of optimal pre-pushforwards and provide explicit operator formulas for them. We also emphasize that, contrary to certain oversights in the literature, proving (essential) self-adjointness is significantly more challenging than demonstrating the symmetry of an unbounded optimal pre-pushforward, particularly in $\b{R}$-Hilbert spaces \cite{conway2019course, hall2013quantum, reed1980methods}. For unbounded operators, symmetry is a \textit{strictly} weaker concept than self-adjointness.

Furthermore, we prove that the uniqueness of the optimal pre-pushforward is equivalent to the vanishing of the Schur complement, $\m{B}/\m{A} = \m{0}$. 
Standard restrictions found in the literature \cite{masarotto2019procrustes, cuesta1989notes}, such as injectivity ($\c{N}(\m{A}) = \m{0}$) or null-space inclusion ($\c{N}(\m{A}) \subseteq \c{N}(\m{B})$) are subsumed by this general uniqueness category. Notably, this condition is equivalent to the existence of a semi-positive definite (s.p.d.) pre-pushforward. In particular, any s.p.d. pre-pushforward is guaranteed to be optimal; thus, if the Schur complement does not vanish, no s.p.d. transport map exists.

\paragraph{Geometry of Geodesics} 
As a consequence of our reduction results, $(\c{B}_{1}^{+}(\c{H}), \c{W}_{2})$ is geodesically complete \cite{lee2018introduction}, in the sense that there always exists a McCann interpolant induced by deterministic Monge maps \cite{mccann1997convexity}.
In the singular setting, however, the geodesic path between measures may not be unique. We provide a full characterization of the convex set of all constant-speed geodesics connecting two Gaussian measures. We identify a geometric structure where the classical McCann interpolants are precisely the extreme points of this convex set. Additionally, we generalize the underlying Kantorovich dual formulation \cite{brenier1991polar, cuesta1989notes, knott1984optimal, ruschendorf1990characterization} to arbitrary Gaussian measures, identifying dual convex functions that generate the supremum for the dual problem via the notion of \emph{proper alignment}.

\paragraph{Application to Gaussian Processes}
Applying our results to stochastic processes, we present a constructive method for determining optimal transport maps between Gaussian processes using integral kernels \cite{kallenberg1997foundations, da2006introduction}. As a major corollary, we provide the exact value of the 2-Wasserstein distance between Integrated Brownian Motions. To the best of our knowledge, this is the first explicit result for the 2-Wasserstein distance between infinite-dimensional Gaussian measures associated with distinct differential operators.

\paragraph{Generalization to Barycenters} 
In the finite-dimensional case $\c{H} = \b{R}^{n}$, We introduce a novel framework for the Wasserstein barycenter problem \cite{agueh2011barycenters, alvarez2016fixed, zemel2019frechet} and its multi-marginal generalizations, based on the concept of random Green's matrices.

\medskip
The paper is organized as follows. \cref{sec:prelim} develops the necessary operator-theoretic preliminaries, introducing Green’s operators and the Hilbertian Schur complement. \cref{sec:fin:dim} characterizes linear reachability in finite dimensions, establishing the algebraic conditions for the existence of symmetric transport maps. \cref{sec:bary} extends these results to the Wasserstein barycenter and multicoupling problems. \cref{sec:admissible} rigorously treats the validity of unbounded pushforwards in Hilbert space. Finally, \cref{sec:hilb:reach} and \cref{sec:Gproc} present the general Hilbertian reachability theory and its application to Gaussian processes.

\section{Preliminaries}\label{sec:prelim}

\subsection{Notations}
We denote a $\b{R}$-Hilbert space by $\c{H}$, which is always assumed to be \emph{separable}. Specifically, we reserve this notation for the finite-dimensional case, and explicitly denote it as $\b{R}^{n}$ whenever possible. For any closed subspace $\c{V} \subset \c{H}$, we denote the orthogonal projection onto $\c{V}$ by $\m{\Pi}_{\c{V}}$. The symbol $^{\dagger}$ denotes the Moore-Penrose psuedoinverse. For an operator $\m{T}$ we denote its null space and its range by $\c{N}(\m{T})$ and $\c{R}(\m{T})$. 

The Banach space of bounded operators on $\c{H}$ is denoted by $\c{B}_{\infty}(\c{H})$, equipped with the operator norm $\vertiii{\cdot}_{\infty}$. For any $\m{K} \in \c{B}_{\infty}(\c{H})$, it holds that $\c{R}(\m{K}^{\dagger}) = \c{N}(\m{K})^{\perp}$, $\m{K}^{\dagger} \m{K} = \m{\Pi}_{\c{N}(\m{K})^{\perp}}$, and the closure of $\m{K} \m{K}^{\dagger}$ is given by $\m{\Pi}_{\overline{\c{R}(\m{K})}}$. 
We denote by $\c{B}_{0}(\c{H})$ the Banach space of compact operators. For $1 \le p < \infty$, the Banach space of the $p$-Schatten class operators on $\c{H}$ is defined as:
\begin{equation*}
    \c{B}_{p}(\c{H}) := \left\{ \m{K} \in \c{B}_{0}(\c{H}) : \vertiii{\m{K}}_{p} := \left( \sum_{l=1}^{\infty} \langle e_{l}, (\m{K}^{*} \m{K})^{p/2} e_{l} \rangle \right)^{1/p} < +\infty \right\}, 
\end{equation*}
where $\{e_{l}\}_{l=1}^{\infty}$ is any orthonormal basis (ONB) of $\c{H}$. We also use $\c{B}_{\infty}^{+}(\c{H})$ the Banach space of s.p.d. bounded operators, and similarly define $\c{B}_{0}^{+}(\c{H})$ or $\c{B}_{p}^{+}(\c{H})$ for $p \in [1, \infty)$.
For $\m{K} \in \c{B}_{1}(\c{H})$, its trace is defined as
\begin{equation*}
    \tr (\m{K}) := \sum_{l=1}^{\infty} \langle e_{l}, \m{K} e_{l} \rangle,
\end{equation*}
which equals to its trace norm if and only if $\m{K} \in \c{B}_{1}^{+}(\c{H})$ using the singular value decomposition, see \cref{lem:Neumann:eq}. For $p=2$, the $2$-Schatten class operators is commonly referred to as the Hilbert–Schmidt operators. Given $f,g \in \c{H}$, the elementary tensor $f \otimes g: \c{H} \rightarrow \c{H}$ is defined by $(f \otimes g) h := \langle g, h \rangle f$ \cite{hsing2015theoretical, paulsen2016introduction}. Then $\c{B}_{2}(\c{H})$ is isometric and linearly isomorphic to the tensor product space $\c{H} \otimes \c{H}$, which is the completion of the vector space spanned by elementary tensors \cite{kadison1986fundamentals}.

\subsection{Green's Operator}
The open mapping theorem states that the range of a compact operator on a Hilbert space cannot be closed unless it has a finite rank. Given this, comparing the ranges of two operators is a non-trivial task. The Douglas lemma \cite{douglas1966majorization} provides an invaluable tool for this purpose, and we rely on it extensively throughout the paper:

\begin{lemma}[Douglas]
Let $\c{H}_{1}, \c{H}_{2}, \c{H}_{3}$ be Hilbert spaces, and let $\m{A} \in \c{B}_{\infty}(\c{H}_{3}, \c{H}_{1})$ and $\m{B} \in \c{B}_{\infty}(\c{H}_{2}, \c{H}_{1})$ be bounded operators. Then the following are equivalent:
\begin{enumerate}[leftmargin = *]
    \item $\c{R}(\m{A}) \subseteq \c{R}(\m{B})$. 
    \item There exists some $c > 0$ such that $\m{A} \m{A}^{*} \preceq c^{2} \m{B} \m{B}^{*}$.
    \item There exists a bounded operator $\m{R} \in \c{B}_{\infty}(\c{H}_{3}, \c{H}_{2})$ such that $\m{A} = \m{B} \m{R}$.
\end{enumerate}
In particular, $\m{R} = \m{B}^{\dagger} \m{A}$ is a unique bounded operator satisfying $\m{A} = \m{B} \m{R}$ in part 3 with $\c{N}(\m{B}) \subset \c{N}(\m{R}^{*})$, and it holds that $\vertii{\m{B}^{\dagger} \m{A}}_{\infty} = \inf \{ c> 0 : \m{A} \m{A}^{*} \preceq c^{2} \m{B} \m{B}^{*} \}$ in part 2.
\end{lemma}

\begin{definition}[Partial Isometry]
Let $\c{H}$ be a Hilbert space. An operator $\m{U} \in \c{B}_{\infty}(\c{H})$ is called a \emph{partial isometry} if $\|\m{U} h \| = \|h \|$ for any $h \in \c{N}(\m{U})^{\perp}$. The space $\c{N}(\m{U})^{\perp}$ and $\c{R}(\m{U})$ are called the \emph{initial space} and the \emph{final space}, respectively.
\end{definition}

The final space of a partial isometry $\m{U}$ is closed. $\m{U}^{*} \m{U}$ and $\m{U} \m{U}^{*}$ are projections onto the initial and final spaces, respectively.
Also, the adjoint operator $\m{U}^{*} = \m{U}^{\dagger}$ is also a partial isometry with its initial and final spaces flipped \cite{conway2019course}.

\begin{definition}[Green's Operator]\label{def:Green's}
Let $\c{H}$ be a Hilbert space and $\m{A} \in \c{B}_{0}^{+}(\c{H})$. An operator $\m{G} \in \c{B}_{0}(\c{H})$ satisfying $\m{G} \m{G}^{*} = \m{A}$ is called a \emph{Green's operator} for $\m{A}$, denoted by $\m{G} \in \c{G}(\m{A})$. When $\c{H} = \b{R}^{n}$, we refer to $\m{G}$ as the \emph{Green's matrix} for $\m{A}$.
\end{definition}

The name ``Green's operator'' is motivated by its connection to the theory of Reproducing Kernel Hilbert Spaces (RKHSs). For a trace-class covariance operator $\m{A} \in \c{B}_{1}^{+}(\c{H})$, the space 
\begin{equation}\label{eq:Mahalanobis:RKHS}
    \b{H}(\m{A}) = (\c{R}(\m{A}^{1/2}), \innpr{\cdot}{\cdot}_{\b{H}(\m{A})}), \quad \innpr{\m{A}^{1/2} f}{\m{A}^{1/2} g}_{\b{H}(\m{A})} := \innpr{f}{g}_{\c{H}}, \quad f, g \in \c{N}(\m{A})^{\perp},
\end{equation}
forms the RKHS associated with $\m{A}$ \cite{paulsen2016introduction}. The adjoint of any Green's operator $\m{G} \in \c{G}(\m{A})$ provides a Hilbert space isomorphism between $\c{N}(\m{A})^{\perp}$ and this RKHS: $\| \m{G}^{*} f \|_{\b{H}(\m{A})} = \|f\|_{\c{N}(\m{A})^{\perp}}$. 
The Douglas lemma ensures that any Green's operator $\m{G}$ for $\m{A}$ has the same range as its principal square root; that is, $\c{R}(\m{G}) = \c{R}(\m{A}^{1/2})$. The lemma also leads to the \emph{polar decomposition}: there exists a unique partial isometry $\m{U}$ with initial space $\overline{\c{R}(\m{G}^{*})}$ (which is equivalent to $\c{N}(\m{U}) = \c{N}(\m{G})$) and final space $\overline{\c{R}(\m{G})}$ such that $\m{G} = \m{A}^{1/2} \m{U}$.

\begin{example}
When $\dim \c{H} = \infty$, the partial isometry is the correct characterization for polar decomposition, as $\dim (\c{N}(\m{G})) = \dim (\c{N}(\m{G}^{*}))$ is not necessarily true.
To illustrate this, let $\{e_{k} : k \in \b{N}\}$ be an ONB for $\c{H}$ and $\sigma_{1} \ge \sigma_{2} \ge \cdots > 0$ be a non-increasing sequence of non-zero singular values. Consider the operators:
\begin{align*}
    \m{A}^{1/2} = \sum_{k=1}^{\infty} \sigma_{k} (e_{k} \otimes e_{k}) \in \c{B}_{0}^{+}(\c{H}), \quad
    \m{G} = \sum_{k=1}^{\infty} \sigma_{k} (e_{k} \otimes e_{2k}) \in \c{G}(\m{A}).
\end{align*}
In this case, the polar decomposition $\m{G} = \m{A}^{1/2} \m{U}$ implies that $\m{U}$ is not injective. Hence, $\m{U}$ cannot be unitary or even an isometry.
\end{example}

\begin{lemma}\label{lem:Green's:part:iso}
Let $\m{A} \in \c{B}_{\infty}^{+}(\c{H})$ and $\m{G} \in \c{G}(\m{A})$. 
\begin{enumerate}[leftmargin = *]
\item Given $\tilde{\m{G}} \in \c{G}(\m{A})$, $\m{U} = \m{G}^{\dagger} \tilde{\m{G}} \in \c{B}_{\infty}(\c{H})$ is a unique partial isometry  with initial space $\c{N}(\tilde{\m{G}})^{\perp}$, and final space $\c{N}(\m{G})^{\perp}$ such that $\tilde{\m{G}} = \m{G} \m{U}$.
\item If $\m{U}$ is a partial isometry whose final space contains $\c{N}(\m{G})^{\perp}$, then $\tilde{\m{G}} = \m{G} \m{U} \in \c{G}(\m{A})$.
\end{enumerate}
Consequently, the space of Green's operators of $\m{A}$ can be characterized as
\begin{align*}
    \c{G}(\m{A}) &= \{ \m{G} \m{U} : \m{U} \text{ is a partial isometry with final space including } \c{N}(\m{G})^{\perp} \} \\
    &= \{ \m{G} \m{U} : \m{U} \text{ is a partial isometry with final space } \c{N}(\m{G})^{\perp} \}.
\end{align*}
\end{lemma}

\begin{proof}[Proof of \cref{lem:Green's:part:iso}] \quad
\begin{enumerate}[leftmargin=*]
\item Note that $\m{U} = \m{G}^{\dagger} \tilde{\m{G}} \in \c{B}_{\infty}(\c{H})$ and $\m{G} \m{U} = \m{G} \m{G}^{\dagger} \tilde{\m{G}} = \tilde{\m{G}}$, since $\c{R}(\tilde{\m{G}}) = \c{R}(\m{G})$ by the Douglas lemma. Also, for any $h = \m{G}^{*} f \in \c{R}(\m{G}^{*})$, we have
\begin{equation*}
    \|\m{U}^{*} h\|^{2} = \|\m{U}^{*} \m{G}^{*} f\|^{2} = \|\tilde{\m{G}}^{*} f\|^{2} = \|\m{G}^{*} f\|^{2} = \|h\|^{2},
\end{equation*}
hence by continuity, $\m{U}^{*}$ is a partial isometry with initial space including $\overline{\c{R}(\m{G}^{*})} = \c{N}(\m{G})^{\perp}$. On the other hand, its initial space satisfies $\c{N}(\m{U}^{*})^{\perp} = \c{R}(\m{U}) \subset \c{R}(\m{G}^{\dagger}) = \c{N}(\m{G})^{\perp}$, i.e. $\c{N}(\m{G})^{\perp}$ is the initial space of $\m{U}^{*}$.

Now, consider $\m{V}: \c{R}(\m{G}^{*}) \to \c{R}(\tilde{\m{G}}^{*}), \m{G}^{*} f \mapsto \tilde{\m{G}}^{*} f$, which is a well-defined isometry. Then, we can extend $\m{V}$ to be a partial isometry with initial space $\c{N}(\m{G})^{\perp}$ and final space $\c{N}(\tilde{\m{G}})^{\perp}$. However, this yields that
\begin{equation*}
    \m{V} \m{G}^{*} f = \tilde{\m{G}}^{*} f = \m{U}^{*} \m{G}^{*} f,
\end{equation*}
thus $\m{V}$ and $\m{U}^{*}$ agrees on a dense subset of their initial spaces, i.e. $\m{V} = \m{U}^{*}$. This shows the uniqueness of $\m{U}$ as a partial isometry  with initial space $\c{N}(\tilde{\m{G}})^{\perp}$, and final space $\c{N}(\m{G})^{\perp}$.

\item Since $\m{U} \m{U}^{*}$ is the projection onto the final space, we have $\tilde{\m{G}} \tilde{\m{G}}^{*} = \m{G} (\m{U} \m{U}^{*} \m{G}^{*}) = \m{G} \m{G}^{*} = \m{A}$.
\end{enumerate}
\end{proof}

Indeed, when $\c{H} = \b{R}^{n}$, $\m{U} \in \b{R}^{n \times n}$ could be chosen as an orthogonal matrix. The flexibility in choosing a Green's operator offers significant practical advantages. From a numerical perspective, finding the principal square root $\m{A}^{1/2}$ or its inverse often requires methods like eigenvalue decomposition or Newton-Schulz iteration, which can be computationally intensive and numerically sensitive to small spectral gaps \cite{nocedal1999numerical, golub2013matrix}. In contrast, finding a Green's operator can be far more efficient. In the finite-dimensional setting, a prime example of a Green's matrix is the Cholesky factor $\m{L}$ from the decomposition $\m{A} = \m{L}\m{L}^{\top}$, which is known to be faster and more stable to compute. However, this flexibility is also valuable for theoretical developments, particularly in stochastic processes or finding the Wasserstein barycenter. A specific choice of a Green's operator may be more physically meaningful than the abstract square root (see examples in \cref{sec:Gproc,sec:bary}). Thus, the ability to choose a convenient $\m{G}$ preserves theoretical consistency while enabling simpler proofs and more efficient, stable computations.

\subsection{Hilbertian Schur Complement}\label{ssec:gen:Schur}

The Schur complement is a fundamental concept in the study of conditional covariance, Gaussian graphical models, and residual regression, as it isolates the Gaussian structure that remains after conditioning. A main goal of this paper is to demonstrate that the Schur complement is a key component for understanding the Bures-Wasserstein geometry between two covariance operators $\m{A}, \m{B} \in \c{B}_{1}^{+}(\c{H})$.
This subsection generalizes the Schur complement to an infinite-dimensional Hilbert space, namely $\m{B} / \m{A}$. Crucially, we reveal that $\m{B} / \m{A}$ depends only on the null space $\c{N}(\m{A})$, rather than on the specific operator $\m{A}$ itself.

We begin by recalling the Schur complement for a finite-dimensional square matrix $\m{B} \in \b{R}^{n \times n}$. Decomposing the space $\b{R}^n$ into $\b{R}^{n_{1}} \oplus \b{R}^{n-n_{1}}$, we can write $\m{B}$ as a block matrix:
\begin{equation*}
    \m{B} = \begin{pmatrix}
    \m{B}_{11} & \m{B}_{12} \\
    \m{B}_{21} & \m{B}_{22}
    \end{pmatrix}
    \    :
    \begin{array}{c}
    \b{R}^{n_{1}} \\
    \oplus \\
    \b{R}^{n-n_{1}}
    \end{array}
    \rightarrow
    \begin{array}{c}
    \b{R}^{n_{1}} \\
    \oplus \\
    \b{R}^{n-n_{1}}
    \end{array}.
\end{equation*}
When the block $\m{B}_{11}$ is invertible, the Schur complement of $\m{B}_{22}$ with respect to $\m{B}_{11}$ is defined as $\m{B} / \m{B}_{11} := \m{B}_{22} - \m{B}_{21} \m{B}_{11}^{-1} \m{B}_{12}$. 
For a non-invertible $\m{B}_{11}$, we replace its inverse with $\m{B}_{11}^{\dagger}$, which leads to the following block LDU decompositions of $\m{B}$:
\begin{align}
    &\begin{pmatrix}
        \m{B}_{11} & \m{0} \\
        \m{0} & \m{B}_{22} - \m{B}_{21} \m{B}_{11}^{\dagger} \m{B}_{12}
    \end{pmatrix} = \begin{pmatrix}
        \m{I}_{11} & \m{0} \\
        - \m{B}_{21} \m{B}_{11}^{\dagger} & \m{I}_{22}
    \end{pmatrix}
    \begin{pmatrix}
        \m{B}_{11} & \m{B}_{12} \\
        \m{B}_{21} & \m{B}_{22}
    \end{pmatrix} \begin{pmatrix}
        \m{I}_{11} & - \m{B}_{11}^{\dagger} \m{B}_{12}  \\
        \m{0} & \m{I}_{22}
    \end{pmatrix}, \label{eq:LDU1} \\
    &\begin{pmatrix}
        \m{B}_{11} & \m{B}_{12} \\
        \m{B}_{21} & \m{B}_{22}
    \end{pmatrix} = \begin{pmatrix}
        \m{I}_{11} & \m{0} \\
        \m{B}_{21} \m{B}_{11}^{\dagger} & \m{I}_{22}
    \end{pmatrix}
    \begin{pmatrix}
        \m{B}_{11} & \m{0} \\
        \m{0} & \m{B}_{22} - \m{B}_{21} \m{B}_{11}^{\dagger} \m{B}_{12}
    \end{pmatrix} \begin{pmatrix}
        \m{I}_{11} & \m{B}_{11}^{\dagger} \m{B}_{12}  \\
        \m{0} & \m{I}_{22}
    \end{pmatrix}, \label{eq:LDU2}
\end{align}
These decompositions hold provided the range conditions $\c{R}(\m{B}_{12}) \subset \c{R}(\m{B}_{11})$ and $\c{R}(\m{B}_{21}^{\top}) \subset \c{R}(\m{B}_{11})$ are met, which ensures $\m{B}_{12} = \m{B}_{11} \m{B}_{11}^{\dagger} \m{B}_{12}$ and $\m{B}_{21} = \m{B}_{21} \m{B}_{11}^{\dagger} \m{B}_{11}$. Notably, if $\m{B} \succeq \m{0}$, these range conditions are always satisfied due to Douglas's lemma, and the rank of the Schur complement is given by
\begin{equation*}
    \rk (\m{B}_{22} - \m{B}_{21} \m{B}_{11}^{\dagger} \m{B}_{12}) = \rk (\m{B}) - \rk (\m{B}_{11}).
\end{equation*}

We now extend this framework to a Hilbert space $\c{H}$. Let $\m{A} \in \c{B}_{0}^{+}(\c{H})$ and $\m{B} \in \c{B}_{0}(\c{H})$ be symmetric. To define the Schur complement of $\m{B}$ with respect to $\m{A}$, we use the orthogonal decomposition of $\c{H}$ induced by $\m{A}$:
\begin{equation*}
    \c{H} = \c{H}_{1} \oplus \c{H}_{2}, \quad \c{H}_{1} := \overline{\c{R}(\m{A})}, \quad \c{H}_{2} := \c{N}(\m{A}).
\end{equation*}
Let $\m{\Pi}_{1}$ and $\m{\Pi}_{2}$ be the orthogonal projections onto $\c{H}_{1}$ and $\c{H}_{2}$, respectively. 
Relative to this decomposition, $\m{A}$ and $\m{B}$ can be expressed as block operators:
\begin{align*}
    \m{A} =
    \begin{pmatrix}
    \m{A}_{11} & \m{0} \\
    \m{0} & \m{0}
    \end{pmatrix},
    \
    \m{B} =
    \begin{pmatrix}
    \m{B}_{11} & \m{B}_{12} \\
    \m{B}_{21} & \m{B}_{22}
    \end{pmatrix}
    \    :
    \begin{array}{c}
    \c{H}_{1} \\
    \oplus \\
    \c{H}_{2}
    \end{array}
    \rightarrow
    \begin{array}{c}
    \c{H}_{1} \\
    \oplus \\
    \c{H}_{2}
    \end{array},
\end{align*}
where $\m{A}_{11} := \m{A} \mid_{\c{H}_{1}}$ is an injective operator, and the block components of $\m{B}$ are given by $\m{B}_{ij} := \m{\Pi}_{i} \m{B} \mid_{\c{H}_{j}}: \c{H}_{j} \to \c{H}_{i}$ for $i, j = 1, 2$.
Provided that the range condition $\c{R}(\m{B}_{12}) \subset \c{R}(\m{B}_{11}^{1/2})$ is satisfied, we define the $\m{A}$-\emph{Schur complement} of $\m{B}$ as the operator on $\c{H}_{2}$ as follows:
\begin{equation}\label{eq:gen:Schur}
    \m{B} / \m{A} := \m{B}_{22} - (\m{B}_{11}^{\dagger/2} \m{B}_{12})^{*} (\m{B}_{11}^{\dagger/2} \m{B}_{12}): \c{H}_{2} \to \c{H}_{2},
\end{equation}
With this definition, we obtain the $\m{A}$-block LDU decomposition for $\m{B}$:
\begin{align}\label{eq:inf:LDU}
    \m{B} = \begin{pmatrix}
        \m{B}_{11} & \m{B}_{12} \\
        \m{B}_{21} & \m{B}_{22}
    \end{pmatrix} = \begin{pmatrix}
        \m{B}_{11}^{1/2} & \m{0} \\
        (\m{B}_{11}^{\dagger/2} \m{B}_{12})^{*} & \m{\Pi}_{2}
    \end{pmatrix}
    \begin{pmatrix}
        \m{\Pi}_{1} & \m{0} \\
        \m{0} & \m{B}/\m{A}
    \end{pmatrix} \begin{pmatrix}
        \m{B}_{11}^{1/2} & \m{B}_{11}^{\dagger/2} \m{B}_{12}  \\
        \m{0} & \m{\Pi}_{2}
    \end{pmatrix}.
\end{align}
In standard notation, one would write $\m{B} / \m{B}_{11}$, but we will shortly show that this is equivalent to $\m{B} / \m{I}_{11}$ as long as $\m{B}_{11}$ is invertible.

\begin{lemma}\label{lem:Schur:spd}
Let $\m{A} \in \c{B}_{0}^{+}(\c{H})$ be a compact s.p.d. operator and $\m{B} \in \c{B}_{0}(\c{H})$ a compact symmetric operator. Then, $\m{B} \in \c{B}_{0}^{+}(\c{H})$ if and only if
\begin{equation*} 
    \c{R}(\m{B}_{12}) \subset \c{R}(\m{B}_{11}^{1/2}), \quad \m{B}_{11} \in \c{B}_{0}^{+}(\c{H}_{1}), \quad \m{B} / \m{A} \in \c{B}_{0}^{+}(\c{H}_{2}).
\end{equation*}
\end{lemma}
\begin{proof}[Proof of \cref{lem:Schur:spd}]
\begin{enumerate}
    \item [($\Rightarrow$)] Immediate from the Douglas lemma \cite{douglas1966majorization} and the Baker theorem \cite{baker1970mutual, hsing2015theoretical}.
    \item [($\Leftarrow$)] Immediate from the $\m{A}$-block LDU decomposition in \eqref{eq:inf:LDU}.
\end{enumerate}
\end{proof}

The following lemma formalizes the utility of the Green's operator by providing an alternative characterization of the $\m{A}$-Schur complement.

\begin{lemma}\label{lem:Schur:op}
Let $\c{H}$ be a Hilbert space, and $\m{A}, \m{B} \in \c{B}_{0}^{+}(\c{H})$. Let $\m{G}_{11} \in \c{G}(\m{A}_{11})$ be a Green's operator, and denote its trivial extension to $\c{H}$ by
\begin{align*}
    \m{G} =
    \begin{pmatrix}
    \m{G}_{11} & \m{0} \\
    \m{0} & \m{0}
    \end{pmatrix}
    \    :
    \begin{array}{c}
    \c{H}_{1} \\
    \oplus \\
    \c{H}_{2}
    \end{array}
    \rightarrow
    \begin{array}{c}
    \c{H}_{1} \\
    \oplus \\
    \c{H}_{2}
    \end{array}.
\end{align*}
Then the following statements hold:
\begin{enumerate}[leftmargin = *]
    \item $\m{G}_{11}^{*} \in \c{B}_{0}(\c{H}_{1})$ is injective, and its extension $\m{G} \in \c{B}_{0}(\c{H})$ is a Green's operator for $\m{A}$.
    \item On $\c{H}_{1}$, $\m{G}_{11}^{*} \m{B}_{11} \m{G}_{11} = \m{G}^{*} \m{B} \m{G}$, hence $\c{N}(\m{B}_{11}^{1/2} \m{G}_{11}) = \c{N}(\m{B}^{1/2} \m{G}) \cap \c{H}_{1}$.
    \item The Schur complement $\m{B}/\m{A}$ can also be expressed as:
    \begin{equation*}
        \m{B}/\m{A} = \m{B}_{22} - [(\m{G}_{11}^{*} \m{B}_{11} \m{G}_{11})^{\dagger/2} (\m{G}_{11}^{*} \m{B}_{12})]^{*} [(\m{G}_{11}^{*} \m{B}_{11} \m{G}_{11})^{\dagger/2} (\m{G}_{11}^{*} \m{B}_{12})],
    \end{equation*}
    for any $\m{G}_{11} \in \c{G}(\m{A}_{11})$. 
    \item The operator $\m{B}^{1/2} \m{\Pi}_{\c{N}(\m{G}^{*} \m{B}^{1/2})} \m{B}^{1/2}$ is block-diagonal, where its restriction to $\c{H}_{2}$ is the $\m{A}$-Schur complement:
    \begin{equation*}
        \m{B}^{1/2} \m{\Pi}_{\c{N}(\m{G}^{*} \m{B}^{1/2})} \m{B}^{1/2} =
        \begin{cases}
            \m{0}, &\quad \text{on } \c{H}_{1}, \\
            \m{B}/\m{A}, &\quad \text{on } \c{H}_{2},
        \end{cases}
    \end{equation*}
    hence $\c{R} [(\m{B}/\m{A})^{1/2}] = \c{R}(\m{B}^{1/2}) \cap \c{H}_{2}$.
\end{enumerate}
\end{lemma}

\begin{proof}[Proof of \cref{lem:Schur:op}] \quad
\begin{enumerate}[leftmargin=*]
\item Since $\c{N}(\m{G}_{11}^{*}) = \c{N}(\m{A}_{11}) = \{0\}$, thus $\m{G}_{11}^{*}$ is injective. Then, $\m{G} \m{G}^{*} = \m{A}$ is trivial by direct calculation.
\item Note that $\c{N}(\m{B}_{11}^{1/2} \m{G}_{11}) = \c{N}(\m{G}_{11}^{*} \m{B}_{11} \m{G}_{11})$. On $\c{H}_{1}$,
\begin{equation*}
    \m{G}_{11}^{*} \m{B}_{11} \m{G}_{11} = (\m{\Pi}_{1} \m{G}^{*} \m{\Pi}_{1}) (\m{\Pi}_{1} \m{B} \m{\Pi}_{1}) (\m{\Pi}_{1} \m{G} \m{\Pi}_{1}) = \m{G}^{*} \m{B} \m{G},
\end{equation*}
hence $h \in \c{N}(\m{B}_{11}^{1/2} \m{G}_{11}) \subset \c{H}_{1}$ if and only if $h \in \c{N}(\m{B}^{1/2} \m{G}) \cap \c{H}_{1}$.

\item Let $\m{Q}_{11} := \m{B}_{11}^{1/2} \m{G}_{11} \in \c{B}_{0}(\c{H}_{1})$. Also, since $\m{B} \in \c{B}_{0}^{+}(\c{H})$, there exists some $\m{R}_{12} \in \c{B}_{\infty}(\c{H}_{2}, \c{H}_{1})$ with $\vertii{\m{R}_{12}}_{\infty} \le 1$ such that $\m{B}_{12} = \m{B}_{11}^{1/2} \m{R}_{12} \m{B}_{22}^{1/2}$ and $\m{R}_{12} \m{B}_{22}^{1/2} = \m{B}_{11}^{\dagger/2} \m{B}_{12}$ \cite{baker1970mutual}. Because
\begin{equation*}
    \c{R}(\m{B}_{11}^{\dagger/2} \m{B}_{12}) \subset \c{R}(\m{B}_{11}^{\dagger/2}) = \c{N}(\m{B}_{11}^{1/2})^{\perp} = \c{N}(\m{G}_{11}^{*} \m{B}_{11}^{1/2})^{\perp} = \c{N}(\m{Q}_{11}^{*})^{\perp} = \overline{\c{R}(\m{Q}_{11})},
\end{equation*}
we get
\begin{align*}
    \m{B}/\m{A} &= (\m{B}/\m{A}) + (\m{B}_{11}^{\dagger/2} \m{B}_{21})^{*} (\m{I} - \m{\Pi}_{\overline{\c{R}(\m{Q}_{11})}}) \m{B}_{11}^{\dagger/2} \m{B}_{12} \\
    &= \m{B}_{22} - (\m{B}_{11}^{\dagger/2} \m{B}_{21})^{*} \m{\Pi}_{\overline{\c{R}(\m{Q}_{11})}} \m{B}_{11}^{\dagger/2} \m{B}_{12},
\end{align*}
which is a trivial extension of
\begin{align*}
    &\m{B}_{22}- \m{B}_{22}^{1/2} \m{R}_{12}^{*} \m{Q}_{11} (\m{Q}_{11}^{*} \m{Q}_{11})^{\dagger} \m{Q}_{11}^{*} \m{R}_{12} \m{B}_{22}^{1/2} \\
    &= \m{B}_{22}- (\m{B}_{11}^{1/2} \m{R}_{12} \m{B}_{22}^{1/2})^{*} \m{G}_{11} (\m{Q}_{11}^{*} \m{Q}_{11})^{\dagger} \m{G}_{11}^{*} (\m{B}_{11}^{1/2} \m{R}_{12} \m{B}_{22}^{1/2}) \\
    &= \m{B}_{22} - (\m{B}_{21} \m{G}_{11}) (\m{G}_{11}^{*} \m{B}_{11} \m{G}_{11})^{\dagger} (\m{G}_{11}^{*} \m{B}_{12}).
\end{align*}
From the last line, it becomes clear that its trivial extension $\m{B}/\m{A}$ to the entire space $\c{H}$ can also be written as
\begin{align*}
    \m{B}/\m{A} = \m{B}_{22} - [(\m{G}_{11}^{*} \m{B}_{11} \m{G}_{11})^{\dagger/2} (\m{G}_{11}^{*} \m{B}_{12})]^{*} [(\m{G}_{11}^{*} \m{B}_{11} \m{G}_{11})^{\dagger/2} (\m{G}_{11}^{*} \m{B}_{12})].
\end{align*}

\item Using $\m{B}_{ij} = \m{\Pi}_{i} \m{B} \m{\Pi}_{j}$, we get
\begin{align*}
    \m{B}_{11}^{\dagger/2} \m{B}_{12} = (\m{\Pi}_{1} \m{B} \m{\Pi}_{1})^{\dagger/2} (\m{\Pi}_{1} \m{B} \m{\Pi}_{2}) = [(\m{\Pi}_{1} \m{B} \m{\Pi}_{1})^{\dagger/2} (\m{B}^{1/2} \m{\Pi}_{1})^{*}]] (\m{B}^{1/2} \m{\Pi}_{2}).
\end{align*}
Consequently,
\begin{align*}
    \m{B}/\m{A} &= \m{B}_{22} - (\m{B}_{11}^{\dagger/2} \m{B}_{12})^{*} (\m{B}_{11}^{\dagger/2} \m{B}_{12}) \\
    &= (\m{B}^{1/2} \m{\Pi}_{2})^{*} [\m{I} - [(\m{\Pi}_{1} \m{B} \m{\Pi}_{1})^{\dagger/2} (\m{B}^{1/2} \m{\Pi}_{1})^{*}]^{*} [(\m{\Pi}_{1} \m{B} \m{\Pi}_{1})^{\dagger/2} (\m{B}^{1/2} \m{\Pi}_{1})^{*}]] (\m{B}^{1/2} \m{\Pi}_{2}) \\
    &= (\m{B}^{1/2} \m{\Pi}_{2})^{*} [\m{I} - \m{\Pi}_{\overline{\c{R}(\m{B}^{1/2} \m{\Pi}_{1})}}] (\m{B}^{1/2} \m{\Pi}_{2}) \\
    &= (\m{B}^{1/2} \m{\Pi}_{1} + \m{B}^{1/2} \m{\Pi}_{2})^{*} [\m{I} - \m{\Pi}_{\overline{\c{R}(\m{B}^{1/2} \m{\Pi}_{1})}}] (\m{B}^{1/2} \m{\Pi}_{1} + \m{B}^{1/2} \m{\Pi}_{2}) \\
    &= \m{B}^{1/2} [\m{I} - \m{\Pi}_{\overline{\c{R}(\m{B}^{1/2} \m{\Pi}_{1})}}] \m{B}^{1/2} 
    = \m{B}^{1/2} \m{\Pi}_{\c{N}(\m{\Pi}_{1} \m{B}^{1/2})} \m{B}^{1/2} 
    = \m{B}^{1/2} \m{\Pi}_{\c{N}(\m{G}^{*} \m{B}^{1/2})} \m{B}^{1/2}, 
\end{align*}
where we used $\c{N}(\m{\Pi}_{1}) = \c{N}(\m{A}) = \c{N}(\m{G}^{*})$ in the last equality. On $\c{H}_{1} = \c{H}_{2}^{\perp} = \c{N}(\m{G}^{*})^{\perp}$, one can easily show that $\m{\Pi}_{\c{N}(\m{G}^{*} \m{B}^{1/2})} \m{B}^{1/2} = \m{0}$, hence $\m{B}^{1/2} \m{\Pi}_{\c{N}(\m{G}^{*} \m{B}^{1/2})} \m{B}^{1/2} = \m{0}$.
Applying the Douglas lemma, we get
\begin{align*}
    \c{R}[(\m{B}/\m{A})^{1/2}] =
    \c{R}( \m{B}^{1/2}\m{\Pi}_{\c{N}(\m{G}^{*} \m{B}^{1/2})} ) = \m{B}^{1/2} [\c{N}(\m{G}^{*} \m{B}^{1/2})].
\end{align*}
Note that $h \in \c{R}[(\m{B}/\m{A})^{1/2}]$ if and only if there exists some $f \in \c{N}(\m{G}^{*} \m{B}^{1/2})$ such that $h = \m{B}^{1/2} f$, which is equivalent to $h \in \c{R}(\m{B}^{1/2}) \cap \c{N}(\m{G}^{*}) = \c{R}(\m{B}^{1/2}) \cap \c{H}_{2}$.
\end{enumerate}
\end{proof}

The $\m{A}$-Schur complement is a fundamentally algebraic construct, as it depends solely on the null space of $\m{A}$ rather than the geometry induced by its range. Consequently, if two operators $\m{A}, \tilde{\m{A}} \in \c{B}_{0}^{+}(\c{H})$ share the same null space, their corresponding Schur complements are identical, regardless of any other differences between them. The equivalent characterizations provided in \cref{lem:Schur:op} are therefore crucial for the development of our theory.


\begin{corollary}\label{cor:Schur:rank}
Let $\m{A} \in \b{R}^{n \times n}$ be a s.p.d. matrix and $\m{B} \in \b{R}^{n \times n}$ a symmetric matrix with $\c{R}(\m{B}_{12}) \subset \c{R}(\m{B}_{11})$. Then,
\begin{equation*}
    \rk(\m{B}/\m{A}) = \rk(\m{B}) - \rk(\m{G}^{*} \m{B} \m{G}),
\end{equation*}
independent of the choice of a Green's matrix $\m{G} \in \c{G}(\m{A})$.
\end{corollary}
\begin{proof}[Proof of \cref{cor:Schur:rank}]
Define the s.p.d. block matrix
\begin{align*}
    \m{C} :&= \begin{pmatrix}
        \m{G}_{11}^{*} & \m{0} \\
        \m{0} & \m{I}_{22}
    \end{pmatrix}
    \m{B} \begin{pmatrix}
        \m{G}_{11} & \m{0} \\
        \m{0} & \m{I}_{22}
    \end{pmatrix} = \begin{pmatrix}
        \m{G}_{11}^{*} \m{B}_{11} \m{G}_{11} & \m{G}_{11}^{*} \m{B}_{12} \\
        \m{B}_{21} \m{G}_{11} & \m{B}_{22}
    \end{pmatrix} \\
    &= \begin{pmatrix}
        \m{I} & \m{0} \\
        \m{B}_{21} \m{G}_{11} (\m{G}_{11}^{*} \m{B}_{11} \m{G}_{11})^{\dagger} & \m{I}
    \end{pmatrix}
    \begin{pmatrix}
        \m{C}_{11} & \m{0} \\
        \m{0} & \m{B}/\m{A}
    \end{pmatrix} \begin{pmatrix}
        \m{I} & (\m{G}_{11}^{*} \m{B}_{11} \m{G}_{11})^{\dagger} \m{G}_{11}^{*} \m{B}_{12} \\
        \m{0} & \m{I}
    \end{pmatrix}
\end{align*}
where we used $\m{C}/\m{I}_{11} = \m{B}/\m{I}_{11} = \m{B}/\m{A}$ from part 3 in \cref{lem:Schur:op}. Therefore,
\begin{align*}
    \rk(\m{B}) = \rk(\m{C}) &= \rk(\m{C}_{11}) + \rk (\m{B}/\m{A}) \\
    &= \rk(\m{G}_{11}^{*} \m{B}_{11} \m{G}_{11}) + \rk (\m{B}/\m{A})
    = \rk(\m{G}^{*} \m{B} \m{G}) + \rk (\m{B}/\m{A}).
\end{align*}
\end{proof}

\section{Euclidean Reachability}\label{sec:fin:dim}
This section defines and characterizes the concept of linear reachability between two covariance matrices in a finite-dimensional Euclidean space to address the Monge problem \cite{villani2008optimal, panaretos2020invitation, ambrosio2008gradient}.

\begin{definition}[Reachability in Finite Dimension]\label{def:reachable}
Let $\m{A}, \m{B} \in \b{R}^{n \times n}$ be s.p.d. matrices. We say that $\m{B}$ is linearly reachable from $\m{A}$, denoted $\m{A} \to \m{B}$, if there exists a matrix $\m{T} \in \b{R}^{n \times n}$ satisfying:
\begin{enumerate}
    \item [] \textbf{Transportability :} $\m{T} \m{A} \m{T}^{*} = \m{B}$.
    \item [] \textbf{Optimality :} $\tr (\m{A} \m{T}) = \tr [(\m{G}^{*} \m{B} \m{G})^{1/2}]$ for some (and thus any) Green's matrix $\m{G} \in \c{G}(\m{A})$. 
\end{enumerate}
Such a matrix $\m{T}$ is called an Optimal Transport (OT) matrix.
\end{definition}

Note that the set of all matrices of the form $\{\m{G}^{*} \m{B} \m{G} \mid \m{G} \in \b{R}^{n \times n}, \m{G} \m{G}^{*} = \m{A} \}$ consists of unitarily equivalent matrices due to the polar decomposition, meaning the optimality condition is not affected by the choice of $\m{G} \in \c{G}(\m{A})$. Furthermore, because the transportability condition implies that $(\m{G}^{*} \m{T} \m{G})(\m{G}^{*} \m{T} \m{G})^{*} = \m{G}^{*} \m{B} \m{G}$,
the optimality condition can be simplified to require either $\m{G}^{*} \m{T} \m{G} = (\m{G}^{*} \m{B} \m{G})^{1/2}$ or $\m{G}^{*} \m{T} \m{G} \succeq \m{0}$ due to \cref{lem:Neumann:eq}.
The following theorem shows that $\m{A} \to \m{B}$ if and only if $\rk(\m{A}) \geq \rk(\m{B})$. The theorem also provides a complete characterization of the set of all OT matrices that map $\m{A}$ to $\m{B}$ in this case.

\begin{theorem}[Reachability]\label{thm:sol:sym}
Let $\m{A}, \m{B} \in \b{R}^{n \times n}$ be s.p.d. matrices. 
\begin{enumerate}[leftmargin = *]
\item \textbf{Invertible Case:} If $\m{A} \succ \m{0}$, then $\m{A} \to \m{B}$ and the OT matrix is unique:
\begin{equation}\label{eq:woronowicz}
    \m{S}_{\m{A} \to \m{B}} := (\m{G}^{*})^{-1} (\m{G}^{*} \m{B} \m{G})^{1/2} \m{G}^{-1} \succeq \m{0}.
\end{equation}
While \eqref{eq:woronowicz} is independent of the choice of Green's matrix $\m{G} \in \c{G}(\m{A})$, setting $\m{G} = \m{A}^{1/2}$ yields the well-known Pusz–Woronowicz formula \cite{pusz1975functional}.
\item \textbf{Singular Case:} If $\m{A}$ is non-invertible with $\rk (\m{A}) < n$, we use the block decomposition based on the orthogonal subspaces $\c{H}_{1} = \c{R}(\m{A})$ and $\c{H}_{2} = \c{N}(\m{A})$:
\begin{equation*}
    \m{A} = \begin{pmatrix}
    \m{A}_{11} & \m{0} \\
    \m{0} & \m{0}
    \end{pmatrix}, \quad
    \m{B} = \begin{pmatrix}
    \m{B}_{11} & \m{B}_{12} \\
    \m{B}_{21} & \m{B}_{22}
    \end{pmatrix}, \quad
    \m{T} = \begin{pmatrix}
    \m{T}_{11} & \m{T}_{12} \\
    \m{T}_{21} & \m{T}_{22}
    \end{pmatrix}
    \    :
    \begin{array}{c}
    \c{H}_{1} \\
    \oplus \\
    \c{H}_{2}
    \end{array}
    \rightarrow
    \begin{array}{c}
    \c{H}_{1} \\
    \oplus \\
    \c{H}_{2}
    \end{array},
\end{equation*}
$\m{A} \to \m{B}$ if and only if $\rk(\m{A}) \geq \rk(\m{B})$. In this case, $\m{T} \in \b{R}^{n \times n}$ is an OT matrix if and only if
\begin{align*}
    \m{T} =
    \begin{pmatrix}
    (\m{G}_{11}^{*})^{-1} (\m{G}_{11}^{*} \m{B}_{11} \m{G}_{11})^{1/2} \m{G}_{11}^{-1} & \m{T}_{12} \\
    [\m{B}_{21} \m{G}_{11} (\m{G}_{11}^{*} \m{B}_{11} \m{G}_{11})^{\dagger/2} + (\m{B}/\m{A})^{1/2} \m{U}_{12}^{*}] \m{G}_{11}^{-1} & \m{T}_{22}
    \end{pmatrix} 
\end{align*}
where: 
\begin{itemize}[leftmargin = *]
    \item $\m{T}_{12} : \c{H}_{2} \to \c{H}_{1}$ and $\m{T}_{22} : \c{H}_{2} \to \c{H}_{2}$ are arbitrary submatrices.
    \item $\m{G}_{11} : \c{H}_{1} \to \c{H}_{1}$ is any Green's submatrix for $\m{A}_{11}$, hence invertible,
    \item $\m{U}_{12} : \c{H}_{2} \to \c{H}_{1}$ is a partial isometry mapping from $\c{R}(\m{B}/\m{A})$ to a subspace $\c{V}$ of $\c{N}(\m{B}_{11}^{1/2} \m{G}_{11})$ with the same dimension, i.e., $\m{U}_{12}^{*} \m{U}_{12} = \m{\Pi}_{\c{R}(\m{B}/\m{A})}$ and $\m{U}_{12} \m{U}_{12}^{*} = \m{\Pi}_{\c{V}}$.
\end{itemize}
This construction is independent of the choice of Green's matrix $\m{G}_{11} \in \c{G}(\m{A}_{11})$. Additionally, $\m{T} \in \b{R}^{n \times n}$ can always be made to be symmetric via $\m{T}_{12} = \m{T}_{21}^{*}$ and $\m{T}_{22} = \m{T}_{22}^{*}$.
\end{enumerate}
\end{theorem}
\begin{proof}[Proof of \cref{thm:sol:sym}]
Denote $n_{1} := \rk (\m{A}) \le n$. First, it is clear that if $\m{A} \rightarrow \m{B}$, then $\rk (\m{A}) \ge \rk (\m{B})$. Conversely, assuming $\rk (\m{A}) \ge \rk (\m{B})$, we prove that $\m{A} \rightarrow \m{B}$.

\begin{enumerate}[leftmargin=*]
\item If $n_{1} = n$, i.e. $\m{A} \succ \m{0}$, note that $\rk (\m{G}^{*}) = \rk (\m{G}) = n$ as $\c{N}(\m{G}^{*}) = \c{N}(\m{A})$, so $\m{G}^{*}$ and $\m{G}$ are both full-rank. It is straightforward that $\m{S}_{\m{A} \to \m{B}}$ is an OT matrix, and is independent of the choice of $\{\m{G} \in \b{R}^{n \times n}  \mid \m{G} \m{G}^{*} = \m{A} \}$, as $\m{G}^{*} \m{B} \m{G}$ are unitarily equivalent.

To show the uniqueness, let $\m{T}$ be any transport matrix from $\m{A}$ to $\m{B}$. Then, rewriting the transport condition gives:
\begin{align*}
    \m{T} \m{A} \m{T}^{*} = \m{B} \ \Leftrightarrow \ (\m{G}^{*} \m{T} \m{G}) (\m{G}^{*} \m{T} \m{G})^{*} = \m{G}^{*} \m{B} \m{G}.
\end{align*}
Applying \cref{lem:Neumann:eq}, a version of the von-Neumann trace inequality \cite{horn2012matrix}, we have
\begin{equation*}
    \tr (\m{A} \m{T}) = \tr (\m{G}^{*} \m{T} \m{G}) \le \tr [(\m{G}^{*} \m{B} \m{G})^{1/2}]
\end{equation*}
and equality holds if and only if $\m{G}^{*} \m{T} \m{G} = (\m{G}^{*} \m{B} \m{G})^{1/2}$, i.e. $\m{T} = \m{S}_{\m{A} \to \m{B}}$ as $\m{G}^{*}$ and $\m{G}$ are both full-rank. This establishes the uniqueness of the OT matrix.

\item Let $n_{1} < n$. Note that for any unitary matrix $\m{U} \in \b{R}^{n \times n}$, the transformation
\begin{equation*}
    \m{T} \m{A} \m{T}^{*} = \m{B} \quad \Longleftrightarrow \quad (\m{U}^{*} \m{T} \m{U}) (\m{U}^{*} \m{A} \m{U}) (\m{U}^{*} \m{T}^{*} \m{U}) = \m{U}^{*} \m{B} \m{U}
\end{equation*}
preserves the defining property for the optimality of transport matrices. In this regard, we may decompose $\m{A}$, $\m{T}$, and $\m{B}$ in block form as:
\begin{equation}\label{eq:2blocks:need}
    \m{A} = \begin{pmatrix}
    \m{A}_{11} & \m{0} \\
    \m{0} & \m{0}
    \end{pmatrix}, \quad
    \m{T} = \begin{pmatrix}
    \m{T}_{11} & \m{T}_{12} \\
    \m{T}_{21} & \m{T}_{22}
    \end{pmatrix}, \quad 
    \m{B} = \begin{pmatrix}
    \m{B}_{11} & \m{B}_{12} \\
    \m{B}_{21} & \m{B}_{22}
    \end{pmatrix},
\end{equation}
where $\m{A}_{11} \in \b{R}^{n_{1} \times n_{1}}$ is p.d.. Note that  $\tr(\m A\m T)=\tr (\m A_{11}\m T_{11})$ and $\m T\m A\m T^*$ do not depend on $\m T_{12}$ and $\m T_{22}$. Choose a Green's matrix $\m{G}_{11} \in \b{R}^{n_{1} \times n_{1}}$ for $\m{A}_{11}$.
As the optimality condition only depends on $\m{T}_{11}$ in $\m{T}$,  we can fix 
\begin{equation*}
    \m{T}_{11} = \m{S}_{\m{A}_{11} \to \m{B}_{11}} = (\m{G}_{11}^{*})^{-1} (\m{G}_{11}^{*} \m{B}_{11} \m{G}_{11})^{1/2} \m{G}_{11}^{-1},
\end{equation*}
which is the only choice to be optimal from part 1. Hence, $\m{A} \rightarrow \m{B}$ if and only if there exists a $\m{T}_{21} \in \b{R}^{(n-n_{1}) \times n_{1}}$ satisfying the transportability condition given by the following system of block equations:
\begin{align*}
    &(\m{G}_{11}^{*})^{-1} (\m{G}_{11}^{*} \m{B}_{11} \m{G}_{11})^{1/2} (\m{T}_{21} \m{G}_{11})^{*} = \m{S}_{\m{A}_{11} \to \m{B}_{11}} \m{A}_{11} \m{T}_{21}^{*} = \m{B}_{12} ,\\
    &(\m{T}_{21} \m{G}_{11}) (\m{T}_{21} \m{G}_{11})^{*} = \m{T}_{21} \m{A}_{11} \m{T}_{21}^{*} = \m{B}_{22}.
\end{align*}

To simplify this system, we introduce a change of variables.
Define the matrix $\m{M}_{21} := \m{T}_{21} \m{G}_{11} \in \b{R}^{(n-n_{1}) \times n_{1}}$. The existence of $\m{T}_{21}$ is equivalent to the existence of $\m{M}_{21}$, which must satisfy the transformed system:
\begin{align}\label{eq:reach:rk}
    (\m{G}_{11}^{*} \m{B}_{11} \m{G}_{11})^{1/2} \m{M}_{21}^{*} = \m{G}_{11}^{*} \m{B}_{12} \quad \text{and} \quad \m{M}_{21} \m{M}_{21}^{*} = \m{B}_{22}.
\end{align}
Using the Douglas lemma, observe that
\begin{align*}
    \c{R}(\m{B}_{12}) \subseteq \c{R}(\m{B}_{11}^{1/2}) \quad &\Longrightarrow \quad \c{R}(\m{G}_{11}^{*} \m{B}_{12}) \subseteq \c{R}(\m{G}_{11}^{*} \m{B}_{11}^{1/2}) = \c{R}((\m{G}_{11}^{*} \m{B}_{11} \m{G}_{11})^{1/2}),
\end{align*}
so the first equation in \eqref{eq:reach:rk} always has a solution, characterized by
\begin{equation}\label{eq:sol:decomp}
    \m{M}_{21}^{*} = (\m{G}_{11}^{*} \m{B}_{11} \m{G}_{11})^{\dagger/2} (\m{G}_{11}^{*} \m{B}_{12}) + \m{N}_{12}, \quad \c{R}(\m{N}_{12}) \subset \c{N}((\m{G}_{11}^{*} \m{B}_{11} \m{G}_{11})^{1/2}).
\end{equation}
Substituting $\m{N}_{12}$ into the second equation of \eqref{eq:reach:rk}, we obtain from part 3 in \cref{lem:Schur:op} that:
\begin{align}\label{eq:sol:resi}
    \m{N}_{12}^{*} \m{N}_{12} = \m{B}_{22} - (\m{B}_{21} \m{G}_{11}) (\m{G}_{11}^{*} \m{B}_{11} \m{G}_{11})^{\dagger} (\m{G}_{11}^{*} \m{B}_{12}) = \m{B}/\m{A} \succeq \m{0}.
\end{align}
Define the s.p.d. block matrix
\begin{align*}
    \m{C} :&= \begin{pmatrix}
        \m{G}_{11}^{*} & \m{0} \\
        \m{0} & \m{I}_{22}
    \end{pmatrix}
    \m{B} \begin{pmatrix}
        \m{G}_{11} & \m{0} \\
        \m{0} & \m{I}_{22}
    \end{pmatrix} = \begin{pmatrix}
        \m{G}_{11}^{*} \m{B}_{11} \m{G}_{11} & \m{G}_{11}^{*} \m{B}_{12} \\
        \m{B}_{21} \m{G}_{11} & \m{B}_{22}
    \end{pmatrix} \\
    &= \begin{pmatrix}
        \m{I} & \m{0} \\
        \m{B}_{21} \m{G}_{11} (\m{G}_{11}^{*} \m{B}_{11} \m{G}_{11})^{\dagger} & \m{I}
    \end{pmatrix}
    \begin{pmatrix}
        \m{C}_{11} & \m{0} \\
        \m{0} & \m{B}/\m{A}
    \end{pmatrix} \begin{pmatrix}
        \m{I} & (\m{G}_{11}^{*} \m{B}_{11} \m{G}_{11})^{\dagger} \m{G}_{11}^{*} \m{B}_{12} \\
        \m{0} & \m{I}
    \end{pmatrix}
\end{align*}
where the last equality holds as in \eqref{eq:LDU2} due to the Douglas lemma and $\m{C}/\m{I}_{11} = \m{B}/\m{I}_{11} = \m{B}/\m{A}$ from part 3 in \cref{lem:Schur:op}.
Since $\rk (\m{C}) = \rk (\m{B}) \le n_{1}$, applying \cref{lem:Schur:root} guarantees the existence of such $\m{N}_{12}$, i.e. $\m{A} \rightarrow \m{B}$. Note that $\m{N}_{12}$ is characterized by $\m{N}_{12} = \m{U}_{12} (\m{B}/\m{A})^{1/2}$,
where $\m{U}_{12}$ is a partial unitary matrix satisfying $\m{U}_{12}^{*} \m{U}_{12} = \m{\Pi}_{\c{R}(\m{B}/\m{A})}$ and $\m{U}_{12} \m{U}_{12}^{*} = \m{\Pi}_{\c{V}}$ for some $\c{V} \subset \c{N}(\m{C}_{11}) = \c{N}(\m{B}_{11}^{1/2} \m{G}_{11})$ of dimension $\rk(\m{B}/\m{A})$. Plugging these results back into $\m{T}_{21}$, we obtain 
\begin{align*}
    \m{T}_{21} &= [(\m{G}_{11}^{*} \m{B}_{11} \m{G}_{11})^{\dagger/2} \m{G}_{11}^{*} \m{B}_{12} + \m{U}_{12} (\m{B}/\m{A})^{1/2}]^{*} \m{G}_{11}^{-1} \\
    &= [\m{B}_{21} \m{G}_{11} (\m{G}_{11}^{*} \m{B}_{11} \m{G}_{11})^{\dagger/2} + (\m{B}/\m{A})^{1/2} \m{U}_{12}^{*}] \m{G}_{11}^{-1}. 
\end{align*}
Finally, we can always make an OT matrix $\m{T}$ to be symmetric by taking $\m{T}_{12} := \m{T}_{21}^{*}$ and $\m{T}_{22}$ be any symmetric submatrix.
\end{enumerate}
\end{proof}

\begin{remark}
\cref{thm:sol:sym} reveals that the \emph{reachability} relation $\to$ defines a \emph{total preorder} on the set of s.p.d. matrices, determined entirely by their rank. For any s.p.d. matrices $\m{A}, \m{B}, \m{C} \in \b{R}^{n \times n}$:
\begin{enumerate}[leftmargin = *]
    \item (Reflexivity) $\m{A} \rightarrow \m{A}$, since $\rk(\m{A}) \ge \rk(\m{A})$ is always true.
    \item (Transitivity) If $\m{A} \rightarrow \m{B}$ and $\m{B} \rightarrow \m{C}$, then $\m{A} \rightarrow \m{C}$.
    \item (Strong Connectivity) For any pair $\m{A}, \m{B}$, either $\m{A} \rightarrow \m{B}$ or $\m{B} \rightarrow \m{A}$. 
\end{enumerate} 
\end{remark}

We now examine the evolution of rank along a geodesic path that arises from an OT matrix $\m{T}$ from $\m{A} \rightarrow \m{B}$. , namely the \emph{McCann interpolant}:
\begin{equation}\label{eq:McCann:int}
    \m{\Gamma}_{t} := [(1-t) \m{I} + t \m{T}] \m{A} [(1-t) \m{I} + t \m{T}]^{*}, \quad t \in [0, 1].
\end{equation}
A direct calculation show that $\m{A} \rightarrow \m{\Gamma}_{t}$ with an OT matrix $[(1-t) \m{I} + t \m{T}]$, and 
\begin{align*}
    \c{W}_{2}(\m{A}, \m{\Gamma}_{t}) = t \c{W}_{2}(\m{A}, \m{B}), \quad \c{W}_{2}(\m{\Gamma}_{t}, \m{B}) = (1-t) \c{W}_{2}(\m{A}, \m{B}).
\end{align*}
A direct consequence of the reachability condition is that rank must be non-increasing in $t \in [0, 1]$. As we will show later, the geodesic path is generally not unique unless stringent conditions hold. While this might suggest that rank evolution could vary, the following theorem reveals a universal property: regardless of which OT matrix is chosen, the rank along its corresponding McCann interpolant remains constant, except for the endpoint $\m{B}$:

\begin{theorem}[Constant Rank along Interpolant]\label{thm:geo:Schur}
Let $\m{A}, \m{B} \in \b{R}^{n \times n}$ be s.p.d. matrices with $\rk(\m{A}) \geq \rk(\m{B})$ so that $\m{A} \to \m{B}$. Consider a McCann interpolant in \eqref{eq:McCann:int}, where $\m{T}$ is any OT matrix from $\m{A}$ to $\m{B}$. Then for any $t \in [0, 1)$, the $\m{A}$-Schur complement vanishes, i.e., $\m{\Gamma}_{t}/\m{A} = \m{0}$, which implies that the rank is constant for the entire open interval, i.e., $\rk(\m{\Gamma}_{t}) = \rk (\m{A})$.
\end{theorem}

\begin{proof}[Proof of \cref{thm:geo:Schur}]
We may decompose $\m{A}$, $\m{B}$, $\m{T}$, and $\m{\Gamma}_{t}$ in block form as in \eqref{eq:2blocks:need}. Note that
\begin{align*}
    &(\m{\Gamma}_{t})_{11} = [((1-t) \m{I}_{11} + t \m{T}_{11}) \m{G}_{11}] [((1-t) \m{I}_{11} + t \m{T}_{11}) \m{G}_{11}]^{*}, \\
    &(\m{\Gamma}_{t})_{21} = t^{2} \m{B}_{21} + t(1-t) \m{T}_{21} \m{A}_{11} = t \m{T}_{21} \m{G}_{11} \m{G}_{11}^{*} ((1-t) \m{I}_{11} + t \m{T}_{11}), \\
    &(\m{\Gamma}_{t})_{22} = t^{2} \m{B}_{22} = t^{2} \m{T}_{21} \m{G}_{11} \m{G}_{11}^{*} \m{T}_{21}^{*},
\end{align*}
where $\m{G}_{11} \in \b{R}^{n_{1} \times n_{1}}$ is a Green's submatrix for $\m{A}_{11}$. 
Note that if $t < 1$, then $\m{X}_{11} := (1-t) \m{I}_{11} + t \m{T}_{11} \succ \m{0}$ is invertible since $\m{T}_{11} \succeq \m{0}$. Hence, we obtain
\begin{align*}
    \m{\Gamma}_{t}/\m{A} &= (\m{\Gamma}_{t})_{22} - (\m{\Gamma}_{t})_{21} (\m{\Gamma}_{t})_{11}^{\dagger} (\m{\Gamma}_{t})_{21}^{*} \\
    &= t^{2} (\m{T}_{21} \m{G}_{11}) \left[ \m{I}_{11} - \m{G}_{11}^{*} \m{X}_{11} \m{X}_{11}^{-1} (\m{G}_{11} \m{G}_{11}^{*})^{\dagger} \m{X}_{11}^{-1} \m{X}_{11} \m{G}_{11} \right] (\m{T}_{21} \m{G}_{11})^{*} \\
    &= t^{2} (\m{T}_{21} \m{G}_{11}) \left[ \m{I}_{11} - \m{G}_{11}^{*} (\m{G}_{11} \m{G}_{11}^{*})^{-1} \m{G}_{11} \right] (\m{T}_{21} \m{G}_{11})^{*} 
    = t^{2} (\m{T}_{21} \m{G}_{11}) \left[ \m{I}_{11} - \m{I}_{11} \right] (\m{T}_{21} \m{G}_{11})^{*} = \m{0},
\end{align*}
since $\m{G}_{11}$ is invertible.
Therefore, using \cref{cor:Schur:rank}, we conclude 
\begin{align*}
    \rk(\m{\Gamma}_{t}) = \rk (\m{G}^{*} \m{\Gamma}_{t} \m{G}) &= \rk (\m{G}_{11}^{*} (\m{\Gamma}_{t})_{11} \m{G}_{11}) \\
    &= \rk (\m{G}_{11}^{*} ((1-t) \m{I}_{11} + t \m{T}_{11}) \m{G}_{11}) =\rk (\m{G}_{11}) = \rk (\m{A}).
\end{align*}
\end{proof}

In the classical Monge-Kantorovich problem with a quadratic cost, if the source measure is absolutely continuous, the unique optimal transport (OT) map is the gradient of a convex function \cite{brenier1991polar, cuesta1989notes, ruschendorf1990characterization}. For Gaussian measures, this implies that when the source covariance matrix $\m{A}$ is invertible, the OT matrix is necessarily s.p.d. \cite{villani2021topics, panaretos2020invitation}. However, \cref{thm:sol:sym} establishes the existence of an OT matrix even when $\m{A}$ is rank-deficient. This raises a crucial question for the singular case: \emph{When does an s.p.d. OT matrix exist on $\b{R}^{n}$?}
The following theorem answers this question by establishing a set of equivalent conditions. These equivalences except for the first condition also hold in the infinite-dimensional setting (\cref{thm:spd:inf}).

\begin{theorem}[SPD Reachability]\label{thm:spd:sol:sym}
Let $\m{A}, \m{B} \in \b{R}^{n \times n}$ be s.p.d. matrices. Define the set of s.p.d. transport matrices as:
\begin{equation*}
    \c{S}_{\m{A} \to \m{B}}^{+} := \{\m{T} \in \b{R}^{n \times n} \mid \m{T} \succeq \m{0}, \, \m{T} \m{A} \m{T}^{*} = \m{B} \}.
\end{equation*}
The following statements are equivalent:
\begin{enumerate}[leftmargin = *]
    \item $\c{S}_{\m{A} \to \m{B}}^{+} \neq \emptyset$.
    \item The OT matrix from $\m{A}$ to $\m{B}$ is $N(0, \m{A})$-a.s. unique.
    \item $\m{B}/\m{A} = \m{0}$.
    \item $\c{R} (\m{B}) = \c{R}(\m{B} \m{A})$.
    \item $ \c{R}(\m{B}) \cap \c{N}(\m{A}) = \{ \m{0} \}$.
\end{enumerate}

Moreover, when any of these conditions hold, any s.p.d. transport matrix in $\c{S}_{\m{A} \to \m{B}}^{+}$ is an OT matrix. A canonical choice for this matrix, which has minimal rank of $\rk (\m{B}^{1/2} \m{A}^{1/2})$, can be constructed as:
\begin{small}
\begin{align}\label{eq:gen:pusz:fin}
    \m{S}_{\m{A} \to \m{B}} :=
    \begin{pmatrix}
    (\m{G}_{11}^{*})^{-1} (\m{G}_{11}^{*} \m{B}_{11} \m{G}_{11})^{1/2} \m{G}_{11}^{-1} & (\m{G}_{11}^{*})^{-1} (\m{G}_{11}^{*} \m{B}_{11} \m{G}_{11})^{\dagger/2} \m{G}_{11}^{*} \m{B}_{12} \\
    \m{B}_{21} \m{G}_{11} (\m{G}_{11}^{*} \m{B}_{11} \m{G}_{11})^{\dagger/2} \m{G}_{11}^{-1} & \m{B}_{21} \m{G}_{11} (\m{G}_{11}^{*} \m{B}_{11} \m{G}_{11})^{2 \dagger} \m{G}_{11}^{*}  \m{B}_{12}
    \end{pmatrix}
    \    :
    \begin{array}{c}
    \c{H}_{1} \\
    \oplus \\
    \c{H}_{2}
    \end{array}
    \rightarrow
    \begin{array}{c}
    \c{H}_{1} \\
    \oplus \\
    \c{H}_{2}
    \end{array}.
\end{align}    
\end{small}
This construction is independent of the choice of Green's matrix $\m{G}_{11} \in \c{G}(\m{A}_{11})$.
\end{theorem}

\begin{proof}[Proof of \cref{thm:spd:sol:sym}]
First, if a transport matrix $\m{T}$ from $\m{A}$ to $\m{B}$ is s.p.d., then it is optimal due to \cref{lem:Neumann:eq} and $\m{G}^{*} \m{T} \m{G}$ being s.p.d.. Also, it is straightforward due to the polar decomposition that $\m{S}_{\m{A} \to \m{B}}$ in \eqref{eq:gen:pusz:fin} does not depend on the choice of Green's matrix $\m{G}_{11} \in \b{R}^{n_{1} \times n_{1}}$. We show the rest of the statements. Following the proof of \Cref{thm:sol:sym}, we consider $\m{A}$ in block form as:
\begin{equation*}
    \m{A} = \begin{pmatrix}
    \m{A}_{11} & \m{0} \\
    \m{0} & \m{0}
    \end{pmatrix},
\end{equation*}
where $\m{A}_{11} \succ \m{0}$ is invertible.

\begin{itemize}[leftmargin=*]
\item [($4 \Leftrightarrow 5)$] $\c{R} (\m{B}) = \c{R}(\m{B} \m{A}) \ \Leftrightarrow \ \c{N} (\m{B}) = \c{N}(\m{A} \m{B}) \ \Leftrightarrow \ \c{R}(\m{B}) \cap \c{N}(\m{A}) = \{ \m{0} \}$.
\item [($3 \Leftrightarrow 5)$] Since $\c{R}(\m{B}) = \c{R}(\m{B}^{1/2})$, this is a direct consequence of part 4 in \cref{lem:Schur:op}.
\item [($2 \Leftrightarrow 3)$] Due to \cref{thm:sol:sym}, the OT matrix $\m{T}$ is $N(0, \m{A})$-a.s. unique if and only if $\m{T}_{21}$ is uniquely determined, which in turn happens if and only if $\m{B}/\m{A} = 0$, otherwise $-\m{U}_{12} \in \b{R}^{n_{1} \times (n-n_{1})}$ is also a partial unitary submatrix satisfying $\m{U}_{12}^{*} \m{U}_{12} = \m{\Pi}_{\c{R}(\m{B}/\m{A})}$ and $\m{U}_{12} \m{U}_{12}^{*} = \m{\Pi}_{\mathcal{V}}$.

\item [($3 \Rightarrow 1)$] From the proof of \cref{thm:sol:sym}, $\m{S} := \m{S}_{\m{A} \to \m{B}}$ in \eqref{eq:gen:pusz:fin} is an OT matrix, hence it remains to show that this is a s.p.d. matrix. Fix $\m{v}_{1} \in \c{N}(\m{S}_{11})$. Then,
\begin{equation*}
    \m{G}_{11}^{-1} \m{v}_{1} \in \c{N}[(\m{G}_{11}^{*} \m{B}_{11} \m{G}_{11})^{1/2}] \quad \Rightarrow \quad (\m{G}_{11}^{*} \m{B}_{11} \m{G}_{11})^{\dagger/2} \m{G}_{11}^{-1} \m{v}_{1} = 0 \quad \Rightarrow \quad \m{v}_{1} \in \c{N}(\m{S}_{21}).
\end{equation*}
Thus, $\c{N}(\m{S}_{21}) \supset \c{N}(\m{S}_{11})$, or equivalently, $\c{R}(\m{S}_{12}) \subset \c{R}(\m{S}_{11}^{1/2})$, and thus, $\m{S}/\m{A}$ is well-defined. Now, applying part 3 in \cref{lem:Schur:op} yields
\begin{align*}
    \m{S}/\m{A} &= \m{S}_{22} - (\m{S}_{21} \m{G}_{11}) (\m{G}_{11}^{*} \m{S}_{11} \m{G}_{11})^{\dagger} (\m{G}_{11}^{*} \m{S}_{12}) \\
    &= \m{S}_{22} - (\m{B}_{21} \m{G}_{11}) (\m{G}_{11}^{*} \m{B}_{11} \m{G}_{11})^{2\dagger} (\m{G}_{11}^{*} \m{B}_{12}) = \m{0}.
\end{align*}
Hence, $\m{S} \succeq \m{0}$ from \cref{lem:Schur:spd}.

\item [($1 \Rightarrow 3)$] We may assume that $\m{A}$ is not injective, otherwise there is nothing to prove. Let $\m{S}$ be a s.p.d. OT matrix. From the proof of \cref{thm:sol:sym}, $\m{S}_{11}$ must take the form
\begin{equation*}
    \m{S}_{11} = \m{S}_{\m{A}_{11} \to \m{B}_{11}} = (\m{G}_{11}^{*})^{-1} (\m{G}_{11}^{*} \m{B}_{11} \m{G}_{11})^{1/2} \m{G}_{11}^{-1}.
\end{equation*}
Also, since $\m{S} \succeq \m{0}$, we get from \cref{lem:Schur:spd} that
\begin{equation*}
    \c{N}(\m{S}_{11}^{1/2}) = \c{R}(\m{S}_{11}^{1/2})^{\perp} \subset \c{R}(\m{S}_{12})^{\perp} = \c{N}(\m{S}_{21}),
\end{equation*}
which leads to
\begin{align*}
    \c{N}(\m{B}_{11}^{1/2} \m{G}_{11}) = \c{N}(\m{G}_{11}^{*} \m{B}_{11} \m{G}_{11}) = \c{N} (\m{G}_{11}^{*} \m{S}_{11} \m{G}_{11}) = \c{N}(\m{S}_{11}^{1/2} \m{G}_{11}) \subset \c{N}(\m{S}_{21} \m{G}_{11}),
\end{align*}
or equivalently, $\c{R}[(\m{S}_{21} \m{G}_{11})^{*}] \subset \c{N}(\m{B}_{11}^{1/2} \m{G}_{11})^{\perp}$. Due to \eqref{eq:sol:decomp}, there exists a partial unitary matrix $\m{U}_{12}$ from $\c{R}(\m{B}/\m{A})$ to $\mathcal{V} \subset \c{N}(\m{B}_{11}^{1/2} \m{G}_{11})$ such that
\begin{equation*}
    (\m{S}_{21} \m{G}_{11})^{*} - (\m{G}_{11}^{*} \m{B}_{11} \m{G}_{11})^{\dagger/2} (\m{G}_{11}^{*} \m{B}_{12}) = \m{U}_{12} (\m{B}/\m{A})^{1/2}.
\end{equation*}
However, the range of the left hand side is included in $\c{N}(\m{B}_{11}^{1/2} \m{G}_{11})^{\perp} = \c{N}(\m{G}_{11}^{*} \m{B}_{11} \m{G}_{11})^{\perp}$, whereas the range of the left hand side belongs to $\mathcal{V} \subset \c{N}(\m{B}_{11}^{1/2} \m{G}_{11})$. Hence, we have
\begin{equation*}
    \c{R}(\m{U}_{12}) = \c{R}(\m{U}_{12} (\m{B}/\m{A})^{1/2}) = \{0\},
\end{equation*}
i.e. $\m{B}/\m{A} =0$.
\end{itemize}
Finally, the canonical choice of minimal rank given in \eqref{eq:gen:pusz:fin} follows from the part ($3 \Rightarrow 1)$.
\end{proof}

The preceding theorems provide a complete picture of reachability in this setting. \cref{thm:sol:sym} shows that an OT matrix from $\m{A}$ to $\m{B}$ exists if and only if $\rk(\m{A}) \geq \rk(\m{B})$. \cref{thm:spd:sol:sym} further reveals that if any s.p.d. matrix $\m{T}$ satisfies the transport equation $\m{T}\m{A}\m{T}^*=\m{B}$, then it is automatically an optimal solution. Even more, it guarantees that these solutions are unique in an almost-sure sense, meaning that they are all equivalent on the support $\c{R}(\m{A})$ of the source measure $N(0, \m{A})$. The only requirement to verify this uniqueness is the condition $ \c{R}(\m{B}) \cap \c{N}(\m{A}) = \{ \m{0} \}$, which entails $\rk(\m{A}) \geq \rk(\m{B})$. When this condition holds for a singular $\m{A}$, the proof of the part ($3 \Rightarrow 1)$ in \cref{thm:spd:sol:sym} fully characterizes the set of all s.p.d. transport matrices:
\begin{equation*}
    \c{S}_{\m{A} \to \m{B}}^{+} = \left\{
    \m{S} = \m{S}_{\m{A} \to \m{B}} + 
    \begin{pmatrix}
    \m{0} & \m{0} \\
    \m{0} & \m{S}_{22}
    \end{pmatrix}
    \    :
    \begin{array}{c}
    \c{H}_{1} \\
    \oplus \\
    \c{H}_{2}
    \end{array}
    \rightarrow
    \begin{array}{c}
    \c{H}_{1} \\
    \oplus \\
    \c{H}_{2}
    \end{array} \vert \ \m{S}_{22} \succeq \m{0} \right\},
\end{equation*}
where $\m{S}_{\m{A} \to \m{B}}$ is defined in \eqref{eq:gen:pusz:fin}. The arbitrary component $\m{S}_{22}$ on the null space of $\m{A}$ represents the degrees of freedom that do not affect optimality, which is precisely the $\m{A}$-Schur complement of the resulting matrix $\m{S}$, i.e., $\m{S}/\m{A} = \m{S}_{22}$.

\begin{example}\label{ex:spd:ex}
Consider the following rank-deficient covariance matrices:
\begin{equation*}
    \m{A} = \begin{pmatrix}
    4 & 0 & 0 \\
    0 & 1 & 0 \\
    0 & 0 & 0
    \end{pmatrix}, \quad \m{B} = \begin{pmatrix}
    0 & 0 & 0 \\
    0 & 4 & 2 \\
    0 & 2 & 1
    \end{pmatrix}, \quad \m{C} = \begin{pmatrix}
    0 & 0 & 0 \\
    0 & 0 & 0 \\
    0 & 0 & 1
    \end{pmatrix}.
\end{equation*}

\begin{enumerate}[leftmargin = *]
\item \textbf{$\m{A} \to \m{B}$ :} Here, $\c{R}(\m{B}) = \spann \{(0, 2, 1)^{\top}\}$ intersects trivially with $\c{N}(\m{A}) = \spann \{\m{e}_{3}\}$. According to \cref{thm:spd:sol:sym}, this guarantees the existence of an s.p.d. OT matrix, which has the following form:
\begin{equation*}
    \m{T}_{\m{A} \to \m{B}} = \begin{pmatrix}
    0 & 0 & * \\
    0 & 2 & * \\
    0 & 1 & *
    \end{pmatrix},
\end{equation*}
making all such maps $N(0, \m{A})$-a.s. equivalent. The canonical s.p.d. OT matrix of minimal rank in \eqref{eq:gen:pusz:fin} is given by
\begin{equation*}
    \m{S}_{\m{A} \to \m{B}} = \begin{pmatrix}
    0 & 0 & 0 \\
    0 & 2 & 1 \\
    0 & 1 & 1/2
    \end{pmatrix}.
\end{equation*}

\item \textbf{$\m{A} \to \m{C}$ :} In this case, $\c{R}(\m{C}) = \mathrm{span}\{\m{e}_{3}\}$ has a non-trivial intersection with $\c{N}(\m{A}) = \mathrm{span}\{\m{e}_{3}\}$, and \cref{thm:spd:sol:sym} predicts that no s.p.d. OT matrix exists and the solution is not $N(0, \m{A})$-a.s. unique. Indeed, the general form of an OT matrix takes the form:
\begin{equation*}
    \m{T}_{\m{A} \to \m{C}}^{\theta} = \begin{pmatrix}
    0 & 0 & * \\
    0 & 0 & * \\
    (\cos \theta)/2 & \sin \theta & *
    \end{pmatrix}, \quad \theta \in [0, 2 \pi).
\end{equation*}
where the dependence on $\theta$ confirms the non-uniqueness, and \cref{lem:spd:complete} verifies that they do not admit an s.p.d. completion. 
The associated McCann interpolant is given by
\begin{equation*}
    \m{\Gamma}_{t}^{\theta} = \begin{pmatrix}
    4(1-t)^{2} & 0 & 2t (1-t) \cos \theta  \\
    0 & (1-t)^{2} & t (1-t) \sin \theta \\
    2t (1-t) \cos \theta & t (1-t) \sin \theta & t^{2}
    \end{pmatrix}, \quad t \in [0, 1].
\end{equation*}
For any $t \in [0, 1)$, its range is
\begin{equation*}
    \c{R}(\m{\Gamma}_{t}^{\theta}) = \mathrm{span} \left\{ 
    \begin{pmatrix}
        2(1-t) \\ 0 \\ t \cos \theta
    \end{pmatrix}, \begin{pmatrix}
        0  \\ (1-t) \\ t \sin \theta
    \end{pmatrix} \right\}.
\end{equation*}
Therefore, $\c{R}(\m{\Gamma}_{t}^{\theta}) \cap \c{N}(\m{A}) = \{\m{0}\}$, or equivalently, $\m{\Gamma}_{t}^{\theta}/\m{A} = \m{0}$ by \cref{thm:spd:sol:sym}. This confirms \cref{thm:geo:Schur}: $\rk(\m{\Gamma}_{t}^{\theta}) = \rk (\m{A}) = 2$.
\end{enumerate}
\end{example}

\begin{example}\label{ex:rk:goes:up}
Consider the following rank-deficient covariance matrices:
\begin{equation*}
    \m{A} = \begin{pmatrix}
    1 & 0 \\
    0 & 0
    \end{pmatrix}, \quad \m{B} = \begin{pmatrix}
    0 & 0 \\
    0 & 1
    \end{pmatrix}.
\end{equation*}
Since $\rk (\m{A}) = \rk (\m{B}) = 1$, both $\m{A} \to \m{B}$ and $\m{B} \to \m{A}$ are possible due to \cref{thm:sol:sym}. The OT matrices takes the form:
\begin{equation*}
    \m{T}_{\m{A} \to \m{B}}^{\pm} = \begin{pmatrix}
    0 & * \\
    \pm 1 & *
    \end{pmatrix}, \quad \m{T}_{\m{B} \to \m{A}}^{\pm} = \begin{pmatrix}
    * & \pm 1 \\
    * & 0
    \end{pmatrix}.
\end{equation*}
Neither of these matrices allows for an s.p.d. completion, as predicted by \cref{thm:spd:sol:sym}: $\c{R}(\m{B}) \cap \c{N}(\m{A}) = \spann \{\m{e}_{2}\}$ and $\c{R}(\m{A}) \cap \c{N}(\m{B}) = \mathrm{span}\{\m{e}_{1}\}$.
There are only two possible McCann interpolants from $\m{A}$ to $\m{B}$ (or \textit{vice versa}) constructed from these OT matrices:
\begin{equation*}
    \m{\Gamma}_{t}^{\pm} = \begin{pmatrix}
    (1-t)^{2} & \pm t (1-t)  \\
    \pm t (1-t) & t^{2}
    \end{pmatrix}, \quad t \in [0, 1].
\end{equation*}
Consistent with \cref{thm:geo:Schur}, the rank of these matrices is constantly $1$ along the geodesic. While these are the only possible \textit{Monge} geodesics (McCann interpolants), other \textit{Kantorovich} geodesics exist. In \cref{ex:midpoint:simple}, we show that $\m{C} = \m{I}/4$ is a midpoint between $\m{A}$ and $\m{B}$. 
Because $\rk (\m{C}) = 2 > 1 = \rk (\m{A}) = \rk (\m{B})$, we get $\m{C} \to \m{A}$ and $\m{C} \to \m{B}$. Hence, we can construct a geodesic with $\m{\Gamma}_{0} = \m{A}, \m{\Gamma}_{1/2} = \m{C}, \m{\Gamma}_{1} = \m{B}$, while there exists no single OT matrix $\m{T}_{\m{A} \to \m{B}}$ or $\m{T}_{\m{B} \to \m{A}}$ that describes this geodesic.
\end{example}

\section{Wasserstein Barycenters}\label{sec:bary}
Given a s.p.d. matrix $\m{A} \in \b{R}^{n \times n}$ and $\m{G} \in \c{G}(\m{A})$, the SVD decomposition leads to
\begin{align}\label{eq:Green's:right:inv}
    \c{G}(\m{A}) &= \{ \m{G} \m{U} : \m{U} \in O(n) \}.
\end{align}

\begin{lemma}\label{lem:opt:algn}
Let $\m{A}_{1}, \m{A}_{2}  \in \b{R}^{n \times n}$ be s.p.d. matrices. Then, for any $\m{G}_{1} \in \c{G}(\m{A}_{1})$, there exists some $\m{G}_{2} \in \c{G}(\m{A}_{2})$ such that
\begin{align*}
    \m{G}_{1}^{*} \m{G}_{2} \succeq \m{0} \quad \Leftrightarrow \quad \tr (\m{G}_{1}^{*} \m{G}_{2}) = \tr [(\m{G}_{1}^{*} \m{A}_{2} \m{G}_{1})^{1/2}] = \tr [(\m{A}_{1}^{1/2} \m{A}_{2} \m{A}_{1}^{1/2})^{1/2}].
\end{align*}
\end{lemma}
\begin{proof}[Proof of \cref{lem:opt:algn}]
Given $\m{G}_{1} \in \c{G}(\m{A}_{1})$, there exists a unitary matrix $\m{U}_{1} \in \b{R}^{n \times n}$ such that $\m{G}_{1} = \m{A}_{1}^{1/2} \m{U}_{1}$. Then,
\begin{align*}
    (\m{G}_{1}^{*} \m{A}_{2} \m{G}_{1})^{1/2} = \m{U}_{1}^{*} (\m{A}_{1}^{1/2} \m{A}_{2} \m{A}_{1}^{1/2})^{1/2} \m{U}_{1} \quad \Rightarrow \quad \tr [(\m{G}_{1}^{*} \m{A}_{2} \m{G}_{1})^{1/2}] = \tr [(\m{A}_{1}^{1/2} \m{A}_{2} \m{A}_{1}^{1/2})^{1/2}].
\end{align*}
Considering the SVD decomposition $\m{G}_{1}^{*} \m{A}_{2}^{1/2} = \m{U} \m{D} \m{V}^{*}$, $\m{G}_{2} := \m{A}_{2}^{1/2} \m{V}^{*}$ is a desired solution.
\end{proof}

In this section, we show that the Wasserstein barycenter of the random covariance matrix is equivalent to maximizing the Hilbert-Schmidt norm of the associated random Green's matrix, in favor of shape analysis. Let $\s{A}: (\Omega, \s{F}, \b{P}) \to \b{R}^{n \times n}$ be a random covariance matrix, i.e., $\s{A}(\omega) \succeq \m{0}$ for any $\omega \in \Omega$, satisfying
\begin{align*}
    \b{E} \vertiii{\s{A}}_{1} =  \int_{\Omega} \vertiii{\s{A}(\omega)}_{1} \rd \b{P}(\omega) < +\infty,
\end{align*}
so that $\b{E} \c{W}_{2}^{2}(\s{A}, \m{B})$ is well-defined for any s.p.d. matrices $\m{B} \in \b{R}^{n \times n}$: $\b{E} \c{W}_{2}^{2}(\s{A}, \m{B}) \le 2 \b{E} \vertiii{\s{A}}_{1} + 2 \vertiii{\m{B}}_{1} < \infty$. In this case, its \emph{Wasserstein barycenter} denoted by $\hat{\m{A}} \in \b{R}^{n \times n}$, is an s.p.d. matrix that is a solution to the following minimization problem:
\begin{equation*}
    \b{E} \c{W}_{2}^{2}(\hat{\m{A}}, \s{A}) = \inf_{\m{B} \in \b{R}^{n \times n}, \ \m{B} \succeq \m{0}}  \b{E} \c{W}_{2}^{2}(\m{B}, \s{A}).
\end{equation*}
We are only formulating at the population level, but this includes the weighted empirical version if we consider $\s{A}$ having only finite number of values, say $\{\m{A}_{1}, \cdots, \m{A}_{m}\}$.

\subsection{Random Green's Matrix}\label{ssec:rand:Green}

\begin{definition}\label{def:opt:aligned}
Let $\s{A}: (\Omega, \s{F}, \b{P}) \to \b{R}^{n \times n}$ be a random covariance matrix with $\b{E} \vertiii{\s{A}}_{1} < +\infty$. We then say that $\s{G}: (\Omega, \s{F}, \b{P}) \to \b{R}^{n \times n}$ is a \emph{random Green's matrix} associated to $\s{A}$ if $\s{G}(\omega) \s{G}(\omega)^{*} = \s{A}(\omega)$ for $\b{P}$-a.s.. We abuse the notation $\c{G}(\s{A})$ to denote the space of random Green's matrices with respect to $\s{A}$.  Additionally, $\s{G} \in \c{G}(\s{A})$ is said to be
\begin{enumerate}[leftmargin = *]
    \item \emph{properly aligned} if $(\b{E} \s{G})^{*} \s{G}(\omega) \succeq \m{0}$ (or equivalently, $(\b{E} \s{G})^{*} \s{G}(\omega) = (\b{E} \s{G}^{*} \s{A}(\omega) \b{E} \s{G})^{1/2}$) holds $\b{P}$-a.s..
    \item \emph{optimal} if $\vertii{\b{E} \s{G}}_{2} = \max_{\s{S} \in \c{G}(\s{A})} \vertii{\b{E} \s{S}}_{2}$.
\end{enumerate}
\end{definition}

Note that the expectation $\b{E} \s{G} \in \b{R}^{n \times n}$ of the random Green's matrix is well-defined via Bochner integral with $(\b{E} \vertii{\s{G}}_{2})^{2} \le \b{E} \vertii{\s{G}}_{2}^{2} = \b{E} \vertii{\s{A}}_{1} < \infty$. 
We also remark both the proper alignment and the optimality are right unitary invariant as in \eqref{eq:Green's:right:inv}: for any unitary operator $\m{U} \in O(n)$, $\s{G} \in \c{G}(\s{A})$ is properly aligned (resp. optimal) if and only if $\s{G} \m{U} \in \c{G}(\s{A})$ as well:
\begin{align}\label{eq:opt:align:unitary}
    (\s{G}(\omega) \m{U}) (\s{G}(\omega) \m{U})^{*} = \s{G}(\omega) \s{G}(\omega)^{*}, \quad
    \b{E}(\s{G} \m{U})^{*} \s{G}(\omega) \m{U} = \m{U}^{*} (\b{E} \s{G})^{*} \s{G}(\omega) \m{U}.
\end{align}

\begin{proposition}\label{cor:bary:eq:sol}
Let $\s{A}: (\Omega, \s{F}, \b{P}) \to \b{R}^{n \times n}$ be a random covariance matrix with $\b{E} \vertiii{\s{A}}_{1} < +\infty$.
\begin{enumerate}[leftmargin = *]
\item For any $\m{M} \in \b{R}^{n \times n}$, $\b{E} [\vertii{(\m{M}^{*} \s{A} \m{M})^{1/2}}_{1}] < + \infty$.
\item If the random Green's matrix $\s{G} \in \c{G}(\s{A})$ is properly aligned, then $\b{E} \s{G}$ is a solution to
\begin{align}\label{eq:Knott:eqn}
    \m{M}^{*} \m{M} = \b{E} [(\m{M}^{*} \s{A} \m{M})^{1/2}].
\end{align}
\item If $\m{M} \in \b{R}^{n \times n}$ satisfies $\m{M}^{*} \m{M} = \b{E} [(\m{M}^{*} \s{A} \m{M})^{1/2}]$ and $\c{N}(\m{M}^{*}) \subset \c{N}(\s{A}(\omega))$ holds $\b{P}$-a.s., then there exists a properly aligned $\s{G} \in \c{G}(\s{A})$ such that $\m{M} = \b{E} \s{G}$. 
\end{enumerate}
\end{proposition}

\begin{proof}[Proof of \cref{cor:bary:eq:sol}] \quad
\begin{enumerate}[leftmargin=*]
\item By the Cauchy-Schwarz inequality, we get
\begin{equation*}
    \b{E} [\vertii{(\m{M}^{*} \s{A} \m{M})^{1/2}}_{1}] = \b{E} [\vertii{\m{M}^{*} \s{A}^{1/2}}_{1}]
    \le \vertiii{\m{M}}_{2} \b{E} [ \vertii{\s{A}^{1/2}}_{2}] \le \vertiii{\m{M}}_{2} \sqrt{\b{E} [ \vertii{\s{A}}_{1}]} < + \infty.
\end{equation*}

\item $(\b{E} \s{G})^{*} (\b{E} \s{G}) = \b{E}[(\b{E} \s{G})^{*} \s{G}] = \b{E} [(\b{E} \s{G}^{*} \s{A} \b{E} \s{G})^{1/2}]$.

\item Define a random matrix $\s{G}: (\Omega, \s{F}, \b{P}) \to \b{R}^{n \times n}$ by $\s{G}(\omega) := (\m{M}^{*})^{\dagger} (\m{M}^{*} \s{A}(\omega) \m{M}^{*})^{1/2}$. Then, whenever $\c{N}(\m{M}^{*}) \subset \c{N}(\s{A}(\omega))$ holds, we have 
\begin{align*}
    \s{G}(\omega) \s{G}(\omega)^{*} = [(\m{M}^{*})^{\dagger} \m{M}^{*}] \s{A}(\omega) \overline{\m{M} \m{M}^{\dagger}} = \m{\Pi}_{\overline{\c{R}(\m{M})}} \s{A}(\omega) \m{\Pi}_{\overline{\c{R}(\m{M})}} = \s{A}(\omega),
\end{align*}
hence, $\s{G} \in \c{G}(\s{A})$. Note that 
\begin{equation*}
    \c{R}[\s{G}(\omega)] \subset \c{R}((\m{M}^{*})^{\dagger}) = \c{N}(\m{M}^{*})^{\perp} \, \Rightarrow \, \c{R}[\b{E}\s{G}] \subset \c{N}(\m{M}^{*})^{\perp} \, \Rightarrow \, \c{R}[\m{M} - \b{E}\s{G}] \subset \c{N}(\m{M}^{*})^{\perp},
\end{equation*}
while
\begin{align*}
    \m{M}^{*} \m{M} = \b{E} [(\m{M}^{*} \s{A} \m{M})^{1/2}] = \m{M}^{*} \b{E} \s{G} \quad \Leftrightarrow \quad \c{R}[\m{M} - \b{E}\s{G}] \subset \c{N}(\m{M}^{*}).
\end{align*}
This yields that $\c{R}[\m{M} - \b{E}\s{G}] \subset \c{N}(\m{M}^{*})^{\perp} \cap \c{N}(\m{M}^{*}) = \{0\}$, i.e., $\m{M} = \b{E} \s{G}$. Accordingly,
\begin{align*}
    (\b{E} \s{G})^{*} \s{G}(\omega) = \m{M}^{*} \s{G}(\omega) = (\m{M}^{*} \s{A}(\omega) \m{M}^{*})^{1/2} \succeq \m{0}, \quad \b{P}-\text{a.s.},
\end{align*}
which shows that $\s{G} \in \c{G}(\s{A})$ is properly aligned.
\end{enumerate}
\end{proof}

In the following theorem, we will shortly see that any Green's matrix $\m{M}$ of a Wasserstein barycenter $\hat{\m{A}}$ satisfies \eqref{eq:Knott:eqn}. In this regard, observe the peculiar order of the (LHS) in \eqref{eq:Knott:eqn} while $\hat{\m{A}} = \m{M} \m{M}^{*}$. However, the proper alignment is unitarily invariant, so by taking $\m{M} = \hat{\m{A}}^{1/2}$, we recover a well-known fixed-point equation, which is a necessity condition of the barycenter \cite{knott1984optimal,agueh2011barycenters,ruschendorf2002n,masarotto2019procrustes,kroshnin2021statistical,alvarez2016fixed}:
\begin{align}\label{eq:knott:neces}
    \hat{\m{A}} = \b{E}[(\hat{\m{A}}^{1/2} \s{A} \hat{\m{A}}^{1/2})^{1/2}].
\end{align}

To discuss the Wasserstein barycenter problem, it is convenient to introduce the following risk: for $\m{M} \in \b{R}^{n \times n}$ and $\s{G} \in \c{G}(\s{A})$, we consider
\begin{align*}
    &L(\m{M}, \s{G}) := \b{E} \vertii{\m{M} - \s{G}}_{2}^{2} = \vertii{\m{M}}_{2}^{2} - 2 \tr(\m{M}^{*} \b{E} \s{G}) + \b{E} \vertii{\s{A}}_{1}, \\
    &L(\m{M}) := \min_{\s{G} \in \c{G}(\s{A})} L(\m{M}, \s{G}) = \b{E} \c{W}_{2}^{2}(\m{M} \m{M}^{*}, \s{A}).
\end{align*}
The minimum is achieved due to \cref{lem:opt:algn}. The following statements are straightforward:
\begin{itemize}
    \item $L(\m{M}, \s{G}) = L(\b{E} \s{G}, \s{G}) + \vertii{\m{M} - \b{E} \s{G}}_{2}^{2}$.
    \item $L(\b{E} \s{G}) = L(\b{E} \s{G}, \s{G})$ if and only if $\s{G} \in \c{G}(\s{A})$ is properly aligned.
\end{itemize}
Again, observe the right unitary invariance $L(\m{M}) = L(\m{M} \m{U})$ for any $\m{U} \in O(n)$, thus the Wasserstein barycenter problem can be formulated as:
\begin{equation*}
    \hat{\m{A}} = \m{M} \m{M}^{*}, \quad \m{M} \in \argmin_{\m{G} \in \b{R}^{n \times n}} L(\m{G}),
\end{equation*}
and the right unitary evolution of $\m{M}$ does not alter the resulting barycenter. Now, we fully characterize the barycenter problem:

\begin{theorem}\label{thm:char:barycenter}
Let $\s{A}: (\Omega, \s{F}, \b{P}) \to \b{R}^{n \times n}$ be a random covariance matrix with $\b{E} \vertiii{\s{A}}_{1} < +\infty$.
\begin{enumerate}[leftmargin = *]
\item $\hat{\m{A}} \succeq \m{0}$ is a barycenter of $\s{A}$ if and only if $\hat{\m{A}} = (\b{E} \s{G}) (\b{E} \s{G})^{*}$ for some optimal random Green's matrix $\s{G} \in \c{G}(\s{A})$. Consequently, if the optimal random Green's matrix does not exist, neither does the barycenter.
\item Any optimal random Green's matrix is properly aligned.
\end{enumerate}
\end{theorem}
\begin{proof}[Proof of \cref{thm:char:barycenter}] \quad
\begin{enumerate}
\item [1. ($\Rightarrow$)] Let $\s{G} \in \c{G}(\s{A})$ be an optimal random Green's matrix. For any $\m{G} \in \b{R}^{n \times n}$ and the random Green's matrix $\s{S} \in \c{G}(\s{A})$,
\begin{align*}
    L(\m{G}, \s{S}) &= L(\b{E} \s{S}, \s{S}) + \vertii{\m{G} - \b{E} \s{S}}_{2}^{2} \ge L(\b{E} \s{S}, \s{S}) = -\vertii{\b{E} \s{S}}_{2}^{2} + \b{E} \vertii{\s{A}}_{1} \\
    &\ge - \vertii{\b{E} \s{G}}_{2}^{2} + \b{E} \vertii{\s{A}}_{1} = L(\b{E} \s{G}, \s{G}) \ge L(\b{E} \s{G}),
\end{align*}
hence $L(\m{G}) = \inf_{\s{S} \in \c{G}(\s{A})} L(\m{G}, \s{S}) \ge L(\b{E} \s{G})$. Therefore, $\hat{\m{A}} = (\b{E} \s{G}) (\b{E} \s{G})^{*}$ is a barycenter.

\item [1. ($\Leftarrow$)] Conversely, let $\hat{\m{A}}$ be a barycenter. Because of \cref{lem:opt:algn}, there exists some $\s{G} \in \c{G}(\s{A})$ such that $L(\hat{\m{A}}^{1/2}, \s{G}) = L(\hat{\m{A}}^{1/2})$. Then,
\begin{align*}
    L(\hat{\m{A}}^{1/2}) = L(\hat{\m{A}}^{1/2}, \s{G}) &= L(\b{E} \s{G}, \s{G}) + \vertii{\hat{\m{A}}^{1/2} - \b{E} \s{G}}_{2}^{2} \ge L(\b{E} \s{G}) + \vertii{\hat{\m{A}}^{1/2} - \b{E} \s{G}}_{2}^{2} \\
    &\ge L(\hat{\m{A}}^{1/2}) + \vertii{\hat{\m{A}}^{1/2} - \b{E} \s{G}}_{2}^{2},
\end{align*}
thus $\hat{\m{A}}^{1/2} = \b{E} \s{G}$ and $L(\b{E} \s{G}, \s{G}) = L(\b{E} \s{G})$, i.e., $\s{G} \in \c{G}(\s{A})$ is properly aligned. Then, the optimality of $\s{G}$ follows from [1. ($\Rightarrow$)].

\item [2.] From [1. ($\Rightarrow$)], we have
\begin{align*}
    L(\b{E} \s{G}) = \inf_{\s{S} \in \c{G}(\s{A})} L(\b{E} \s{G}, \s{S}) \ge L(\b{E} \s{G}, \s{G}) \ge L(\b{E} \s{G}),
\end{align*}
hence $L(\b{E} \s{G}, \s{G}) = L(\b{E} \s{G})$, i.e., $\s{G} \in \c{G}(\s{A})$ is properly aligned. 
\end{enumerate}
\end{proof}

\begin{remark}\label{rmk:inf:bary:loss}
As an immediate consequence, we have $\hat{\m{A}} = (\b{E} \s{G}) (\b{E} \s{G})^{*} \preceq \b{E}[\s{G} \s{G}^{*}] = \b{E} \s{A}$ due to the Cauchy-Schwarz inequality \cite{bhatia2019bures}. Also, the Mincowski determinant inequality yields the result in \cite{agueh2011barycenters}:
\begin{align*}
    &\det[\hat{\m{A}}]^{\frac{1}{2n}} = \det [\b{E} \s{G}]^{1/n} \ge \b{E}[\det (\s{G})^{1/n}] = \b{E}[\det (\s{A})^{\frac{1}{2n}}].
\end{align*}
A careful inspection of the proof of \cref{thm:char:barycenter} reveals that the Fr\'{e}chet variance is given by
\begin{equation*}
    \inf_{\m{B} \in \b{R}^{n \times n}, \ \m{B} \succeq \m{0}}  \b{E} \c{W}_{2}^{2}(\m{B}, \s{A}) 
    = - \sup_{\s{S} \in \c{G}(\s{A})} \vertii{\b{E} \s{S}}_{2}^{2} + \b{E} \vertii{\s{A}}_{1},
\end{equation*}
regardless of the existence of the barycenter. 
Additionally, it is well-known that the set of barycenters is convex due to the weak convexity \cite{panaretos2020invitation,alvarez2011uniqueness}:
\begin{align*}
    (1-t) \b{E} \c{W}_{2}^{2}(\s{A}, \hat{\m{A}}_{0}) + t \b{E} \c{W}_{2}^{2}(\s{A}, \hat{\m{A}}_{1}) &\ge \b{E} \c{W}_{2}^{2}(\s{A}, (1-t) N(0,\hat{\m{A}}_{0}) + t N(0,\hat{\m{A}}_{1})) \\
    &\ge \b{E} \c{W}_{2}^{2}(\s{A}, (1-t) \hat{\m{A}}_{0} + t \hat{\m{A}}_{1}), \quad t \in (0, 1).
\end{align*}
On the other hand, the set of optimal random Green's matrix is never closed under the convex combination, except for the trivial case $\s{G}_{0} = \s{G}_{1} \in \c{G}(\s{A})$:
\begin{align*}
    \vertii{(1-t) \s{G}_{0} + t \s{G}_{1}}_{2}^{2} = (1-t) \vertii{\s{G}_{0}}_{2}^{2} + t \vertii{\s{G}_{1}}_{2}^{2} - t(1-t) \vertii{\s{G}_{0} - \s{G}_{1}}_{2}^{2}, \quad t \in (0, 1).
\end{align*}
\end{remark}

The existence of a weighted barycenter is guaranteed when a random covariance matrix $\s{A}$ only takes finite values, say $\m{A}_{i}$ with weights $p_{i}$ for $i= 1, \cdots, m$. This is because the expected trace is finite, $\b{E} \vertiii{\s{A}}_{1} = \sum_{i=1}^{m} p_{i} \tr (\m{A}_{i}) < +\infty$, and the optimal random Gram matrix exists due to the compactness of $\b{R}^{n \times n}$. 

\begin{corollary}\label{cor:bary:injec}
Let $\m{A}_{1}, \cdots, \m{A}_{m}  \in \b{R}^{n \times n}$ be s.p.d. matrices, and $p_{1}, \cdots, p_{m} > 0$ be the weights with $p_{1} + \cdots + p_{m} = 1$. Let $\hat{\m{A}} \in \argmin_{\m{B} \succeq \m{0}}  \sum_{i=1}^{m} p_{i} \c{W}_{2}^{2}(\m{A}_{i}, \m{B})$ be a weighted barycenter. Then it holds that
\begin{align*}
    \c{R} (\m{A}_{i}) \cap \c{N}(\hat{\m{A}}) = \{0\} \quad \Leftrightarrow \quad \m{A}_{i}/\hat{\m{A}} = \m{0}, \quad i = 1, \cdots, m.
\end{align*}
In particular, $\rk(\hat{\m{A}}) \ge \max_{1 \le i \le m} \rk(\m{A}_{i})$, thus if at least one of $\m{A}_{1}, \cdots, \m{A}_{m}$ is injective, then $\hat{\m{A}}$ is also injective \cite{agueh2011barycenters}.
\end{corollary}
\begin{proof}[Proof of \cref{cor:bary:injec}]
Without loss of generality, let $i = 1$. Then,
due to \cref{thm:char:barycenter}, there exists an optimal Green matrices $\m{G}_{j} \in \c{G}(\m{A}_{j})$ for $j=1, \cdots, m$ such that $\hat{\m{G}} = \sum_{j=1}^{m} p_{j} \m{G}_{j} \in \c{G}(\m{A})$. It is trivial that $(\sum_{j=2}^{m} p_{j} \m{G}_{j} )^{*} \m{G}_{1} \succeq \m{0}$ since:
\begin{align*}
    \vertiii{\hat{\m{G}}}_{2}^{2} = \vertiii{\sum_{j=2}^{m} p_{j} \m{G}_{j} }_{2}^{2} + p_{1}^{2} \text{tr}(\m{A}_{1}) + 2 p_{1} \cdot \text{tr} [(\sum_{j=2}^{m} p_{j} \m{G}_{j} )^{*} \m{G}_{1}], \quad \m{G}_{1} \in \c{G}(\m{A}_{1}).
\end{align*}
Let $\m{c} \in \c{R} (\m{A}_{1}) \cap \c{N}(\hat{\m{A}}) = \c{R}(\m{G}_{1}) \cap \c{N}(\hat{\m{G}}^{*})$. Then, there exists some $\m{d} \in \b{R}^{n}$ such that $\m{c} = \m{G}_{1} \m{d}$ and
\begin{align*}
    0 = \m{d}^{*} \hat{\m{G}}^{*} \m{c} = \m{d}^{*} \hat{\m{G}}^{*} \m{G}_{1} \m{d} = \m{d}^{*} \left(\sum_{j=2}^{m} p_{j} \m{G}_{j} \right)^{*} \m{G}_{1} \m{d} + p_{1} \m{c}^{*} \m{c} \ge p_{1} \|\m{c}\|^{2} \quad \Rightarrow \quad \m{c} = \m{0}.
\end{align*}
Consequently, we get $\c{R} (\m{A}_{1}) \cap \c{N}(\hat{\m{A}}) = \{0\}$, i.e., $\m{A}_{1}/\hat{\m{A}} = \m{0}$ due to \cref{lem:Schur:op}.
\end{proof}

In the meantime, proper alignment is a weaker notion than optimality, hence \eqref{eq:knott:neces} cannot characterize the barycenters in general:

\begin{example}
Consider $\Omega = \{1, 2\}$, and $\b{P}({1}) = 1-t, \b{P}({2}) = t$ for some $t \in (0, 1)$, 
and let 
\begin{equation*}
    \m{A}_{1} = \begin{pmatrix}
    a_{1} & 0 & 0 \\
    0 & a_{2} & 0 \\
    0 & 0 & 0
    \end{pmatrix}, \quad \m{A}_{2} = \begin{pmatrix}
    0 & 0 & 0 \\
    0 & b_{2} & 0 \\
    0 & 0 & b_{3}
    \end{pmatrix}, \quad a_{1}, a_{2}, b_{2}, b_{3} > 0, \quad (1-t) \sqrt{a_{2}} = \sqrt{b_{2}}.
\end{equation*}
Then, one can check that $\m{G}_{1} = \mathrm{diag} (\sqrt{a_{1}}, \sqrt{a_{2}}, 0)$ and $\m{G}_{2} = \mathrm{diag} (0, - \sqrt{b_{2}}, \sqrt{b_{3}})$ are properly aligned while $\m{G}_{1}^{*} \m{G}_{2} \succeq \m{0}$ is violated, thus they are not optimal. Accordingly, the induced covariance $\hat{\m{A}} = \mathrm{diag} ((1-t)^{2} a_{1}, 0, t^{2} b_{3})$, is not a barycenter while satisfying \eqref{eq:knott:neces}.
\end{example}

\begin{corollary}\label{cor:ortho:Green's}
Let $\m{A}_{1}, \cdots, \m{A}_{m}  \in \b{R}^{n \times n}$ be s.p.d. matrices, and $p_{1}, \cdots, p_{m} > 0$ be the weights with $p_{1} + \cdots + p_{m} = 1$. If $\m{A}_{i} \m{A}_{j} = \m{0}$ for any $1 \le i < j \le m$, then the set of weighted Wasserstein barycenters are given by
\begin{equation*}
    \argmin_{\m{B} \in \b{R}^{n \times n}, \ \m{B} \succeq \m{0}}  \sum_{i=1}^{m} p_{i} \c{W}_{2}^{2}(\m{A}_{i}, \m{B}) = \left\{ \hat{\m{G}} \hat{\m{G}}^{*} : \hat{\m{G}} = \sum_{i=1}^{m} p_{i} \m{G}_{i}, \, \m{G}_{i} \in \c{G}(\m{A}_{i}), \, i =1, \cdots, m \right\}.
\end{equation*}
\end{corollary}

\begin{proof}[Proof of \cref{cor:ortho:Green's}]
Let $\{\m{G}_{i} \in \c{G}(\m{A}_{i}): i = 1, \cdots, m\}$ be an arbitrary collection of Green's matrices. Then, for any $1 \le i \neq j \le m$, we have
\begin{align*}
    \m{A}_{j} \m{A}_{i} = \m{0}
    \quad &\Leftrightarrow \quad \m{A}_{i} \m{A}_{j} = \m{0} \quad \Leftrightarrow \quad \c{R}(\m{A}_{j})\subset \c{N}(\m{A}_{i}) \quad \Leftrightarrow \quad \overline{\c{R}(\m{A}_{j})} \subset \c{N}(\m{A}_{i}) \\
    &\Leftrightarrow \quad \c{R}(\m{G}_{j}) = \c{R}(\m{A}_{j}^{1/2}) \subset \c{N}(\m{A}_{i}^{1/2}) = \c{N}(\m{G}_{i}^{*}) \quad \Leftrightarrow \quad \m{G}_{j}^{*} \m{G}_{i} = \m{0},
\end{align*}
i.e., $\{\m{G}_{i} \in \c{G}(\m{A}_{i}): i = 1, \cdots, m\}$ is optimal. Then, the results follow immediately from \cref{thm:char:barycenter}.
\end{proof}

More generally, if the covariance matrices can be grouped into several subspaces over which they act on, we can compute the barycenter by computing sub-barycenters, and we refer to \cref{cor:hier:bary} in the appendix.

\begin{example}\label{ex:midpoint:simple}
Recall \cref{ex:rk:goes:up}, where $\m{A} = \diag (1, 0)$ and $\m{B} = \diag (0, 1)$.
Since $\m{A} \m{B} = \m{0}$, we can take any Green's matrices to construct a weighted barycenter between $\m{A}$ and $\m{B}$ due to \cref{cor:ortho:Green's}. We may fix $\m{G} = \m{A}^{1/2} \in \c{G}(\m{A})$ due to right unitary invariance. Let $\m{M} = \m{B}^{1/2} \m{U} \in \c{G}(\m{B})$ where
\begin{equation*}
    \m{U} := \begin{pmatrix}
        c & \pm s \\
        s & \mp c
    \end{pmatrix} \in \mathrm{O}(2), \quad c^{2} + s^{2} = 1.
\end{equation*}
If we denote $\m{G}_{t} = (1-t) \m{G} + t \m{M}$, we get
\begin{align*}
    \m{G}_{t} \m{G}_{t}^{*} = \begin{pmatrix}
        (1-t) & 0 \\
        t s & \mp t c
    \end{pmatrix} \begin{pmatrix}
        (1-t) & t s \\
        0 & \mp t c
    \end{pmatrix} = \begin{pmatrix}
        (1-t)^{2} & t(1-t)s \\
        t(1-t)s & t^{2}
    \end{pmatrix},
\end{align*}
hence 
\begin{align*}
    \argmin_{\m{C} \in \b{R}^{n \times n}, \ \m{C} \succeq \m{0}}  (1-t) \c{W}_{2}^{2}(\m{A}, \m{C}) + t \c{W}_{2}^{2}(\m{B}, \m{C}) = \left\{ \begin{pmatrix}
        (1-t)^{2} & t(1-t)s \\
        t(1-t)s & t^{2}
    \end{pmatrix} : |s| \le 1 \right\},
\end{align*}
which is the convex hull of two McCann's interpolants $\m{\Gamma}_{t}^{+}, \m{\Gamma}_{t}^{-}$ at time $t \in (0, 1)$ in \cref{ex:rk:goes:up}. We will see in \cref{ssec:kanto} that this is not a coincidence.
\end{example}

\subsection{Multicoupling Problem}\label{ssec:multi:trans}
We generalize the multi-marginal optimal transport problem \cite{panaretos2020invitation, agueh2011barycenters} into the infinite version. 

\begin{definition}[Random Kernel]
Let $\s{A}: (\Omega, \s{F}, \b{P}) \to \b{R}^{n \times n}$ be a random covariance matrix with $\b{E} \vertiii{\s{A}}_{1} < +\infty$. A random matrix-valued kernel \cite{paulsen2016introduction} for $\s{A}$ is a measurable function $\s{A}^{\times}: \Omega \times \Omega \to \b{R}^{n \times n}$ that satisfies:
\begin{itemize}
\item (Diagonal Constraint) For all $\omega \in \Omega$, $\s{A}^{\times}(\omega, \omega) = \s{A}(\omega)$.
\item (Symmetry) For all $\omega, \omega' \in \Omega$, $\s{A}^{\times}(\omega, \omega') = \s{A}^{\times}(\omega', \omega)$.
\item (s.p.d.) For any $\omega_{1}, \cdots, \omega_{m} \in \Omega$ and $\m{v}_{1}, \cdots, \m{v}_{m} \in \b{R}^{n}$, we have $\sum_{i, j = 1}^{m} \innpr{\m{v}_{i}}{\s{A}^{\times}(\omega_{i}, \omega_{j}) \m{v}_{j}} \ge 0$. 
\end{itemize}
We denote the space of all such measurable kernels by $\c{K}^{\times}(\s{A})$.
\end{definition}

Note that for any $\s{A}^{\times} \in \c{K}^{\times}(\s{A})$, its expectation in the product measure space
\begin{align*}
    \widehat{\s{A}^{\times}} := \int_{\Omega \times \Omega} \s{A}^{\times}(\omega, \omega') \rd \b{P}(\omega) \b{P}(\omega') \succeq \m{0},
\end{align*}
is well-defined and s.p.d. through limiting arguments. Additionally, observe that
\begin{align*}
    \vertii{\widehat{\s{A}^{\times}}}_{1} = \sum_{k=1}^{n} \innpr{\m{e}_{k}}{\widehat{\s{A}^{\times}} \m{e}_{k}} &\le \sum_{k=1}^{n} \int_{\Omega \times \Omega} \left( \innpr{\m{e}_{k}}{\s{A}(\omega) \m{e}_{k}} \innpr{\m{e}_{k}}{\s{A}(\omega') \m{e}_{k}} \right)^{1/2} \rd \b{P}(\omega) \b{P}(\omega') \\
    &= \sum_{k=1}^{n} \left( \int_{\Omega} \innpr{\m{e}_{k}}{\s{A}(\omega) \m{e}_{k}}^{1/2} \right)^{2} \le \b{E} \vertii{\s{A}}_{1}.
\end{align*}
Our goal is then to maximize its trace: $\sup \{\vertii{\widehat{\s{A}^{\times}}}_{1} : \s{A}^{\times} \in \c{K}^{\times}(\s{A}) \}$. From our construction, the correspondence between the barycenter problem and the multicoupling problem becomes clear.

\begin{theorem}\label{thm:multi:trans}
Let $\s{A}: (\Omega, \s{F}, \b{P}) \to \b{R}^{n \times n}$ be a random covariance matrix with $\b{E} \vertiii{\s{A}}_{1} < +\infty$. Then, 
\begin{align*}
    \s{A}^{\times} \in \argmax_{\s{B}^{\times} \in \c{K}^{\times}(\s{A})} \vertii{\widehat{\s{B}^{\times}}}_{1},
\end{align*}
if and only if the random matrix-valued kernel $\s{A}^{\times}: \Omega \times \Omega \to \b{R}^{n \times n}$ satisfies
\begin{align*}
    \int_{\Omega} \s{A}^{\times}(\omega, \omega') \rd \b{P}(\omega') = \int_{\Omega} \s{G}(\omega)  \s{G}(\omega')^{*} \rd \b{P}(\omega') \in \b{R}^{n \times n}, \quad \b{P}-\text{a.s.},
\end{align*}
for some optimal random Green's matrix $\s{G} \in \c{G}(\s{A})$. Consequently, if the optimal random Green's matrix does not exist, neither does the barycenter. If exists, $\s{A}^{\times}(\omega, \omega') = \s{G}(\omega)  \s{G}(\omega')^{*}$ is an optimal random kernel.
\end{theorem}

\begin{proof}[Proof of \cref{thm:multi:trans}]
For $\s{A}^{\times} \in \c{K}^{\times}(\s{A})$, consider the following risk:
\begin{align*}
    L^{\times}(\s{A}^{\times}) := \frac{1}{2} \int_{\Omega \times \Omega} \tr [\s{A}(\omega) + \s{A}(\omega') - 2\s{A}^{\times}(\omega, \omega')] \rd \b{P}(\omega) \b{P}(\omega') = - \vertii{\widehat{\s{A}^{\times}}}_{1} + \b{E} \vertii{\s{A}}_{1}  \ge 0.
\end{align*}

We claim that
\begin{align*}
    \inf_{\s{B}^{\times} \in \c{K}^{\times}(\s{A})} L^{\times}(\s{B}^{\times}) = \inf_{\m{B} \in \b{R}^{n \times n}, \ \m{B} \succeq \m{0}}  \b{E} \c{W}_{2}^{2}(\m{B}, \s{A}).
\end{align*}

\begin{itemize}[leftmargin = *]
\item [($\ge$)] Given $\s{B}^{\times} \in \c{K}^{\times}(\s{A})$, denote $\tilde{\s{B}}^{\times}(\omega) := \int_{\Omega} \s{B}^{\times}(\omega, \omega') \rd \b{P}(\omega') \in \b{R}^{n \times n}$ for $\omega \in \Omega$, then the block matrix
\begin{align*}
    \begin{pmatrix}
    \s{A}(\omega) & \tilde{\s{B}}^{\times}(\omega) \\
    [\tilde{\s{B}}^{\times}(\omega)]^{*} & \widehat{\s{B}}^{\times}
    \end{pmatrix} \succeq \m{0},
\end{align*}
hence,
\begin{align*}
    L^{\times}(\s{B}^{\times}) = \int_{\Omega} \tr[\s{A}(\omega) - 2 \tilde{\s{B}}^{\times}(\omega) + \widehat{\s{B}}^{\times}] \rd \b{P}(\omega) \ge \int_{\Omega} \c{W}_{2}^{2}(\widehat{\s{B}}^{\times}, \s{A}(\omega)) \rd \b{P}(\omega) = \b{E} \c{W}_{2}^{2}(\widehat{\s{B}}^{\times}, \s{A}).
\end{align*}
Therefore,
\begin{align*}
    \inf_{\s{B}^{\times} \in \c{K}^{\times}(\s{A})} L^{\times}(\s{B}^{\times}) \ge \inf_{\s{B}^{\times} \in \c{K}^{\times}(\s{A})} \b{E} \c{W}_{2}^{2}(\widehat{\s{B}}^{\times}, \s{A}) \ge \inf_{\m{B} \in \b{R}^{n \times n}, \ \m{B} \succeq \m{0}}  \b{E} \c{W}_{2}^{2}(\m{B}, \s{A}).
\end{align*}
\item [($\le$)] Let $\{\s{S}_{k} \in \c{G}(\s{A}) : k \in \b{N}\}$ be a sequence of random Green's matrices such that
\begin{align*}
    \lim_{k \to \infty} \vertii{\b{E} \s{S}_{k}}_{2}^{2} = \sup_{\s{S} \in \c{G}(\s{A})} \vertii{\b{E} \s{S}}_{2}.
\end{align*}
Accordingly, define the sequence of random kernels $\s{B}_{k}^{\times} \in \c{K}^{\times}(\s{A})$ by $\s{B}_{k}^{\times}(\omega, \omega') = \s{S}_{k}(\omega) \s{S}_{k}(\omega')^{*}$. Then, we have $\widehat{\s{B}^{\times}} = [\b{E} \s{S}_{k}] [\b{E} \s{S}_{k}]^{*}$, hence
\begin{align*}
    \lim_{k \to \infty} L^{\times}(\s{B}^{\times}) = - \lim_{k \to \infty} \vertii{\b{E} \s{S}_{k}}_{2}^{2} + \b{E} \vertii{\s{A}}_{1} = \inf_{\m{B} \in \b{R}^{n \times n}, \ \m{B} \succeq \m{0}}  \b{E} \c{W}_{2}^{2}(\m{B}, \s{A}),
\end{align*}
due to \cref{rmk:inf:bary:loss}. Therefore,
\begin{align*}
    \inf_{\s{B}^{\times} \in \c{K}^{\times}(\s{A})} L^{\times}(\s{B}^{\times}) \le \lim_{k \to \infty} L^{\times}(\s{B}^{\times}) = \inf_{\m{B} \in \b{R}^{n \times n}, \ \m{B} \succeq \m{0}}  \b{E} \c{W}_{2}^{2}(\m{B}, \s{A}).         
\end{align*}
\end{itemize}

Consider a solution to
\begin{align*}
    \s{A}^{\times} \in \argmax_{\s{B}^{\times} \in \c{K}^{\times}(\s{A})} \vertii{\widehat{\s{B}^{\times}}}_{1} = \argmin_{\s{B}^{\times} \in \c{K}^{\times}(\s{A})} L^{\times}(\s{B}^{\times}).
\end{align*}
Because of \cref{lem:opt:algn}, there exists some $\s{G} \in \c{G}(\s{A})$ such that $\c{W}_{2}^{2}(\widehat{\s{A}}^{\times}, \s{A}(\omega)) = \vertii{(\widehat{\s{A}}^{\times})^{1/2} - \s{G}(\omega)}_{2}^{2}$ for all $\omega \in \Omega$. Then, 
\begin{align*}
    \inf_{\m{G} \in \b{R}^{n \times n}} L(\m{G}) = L^{\times}(\s{A}^{\times}) \ge \int_{\Omega} \c{W}_{2}^{2}(\widehat{\s{A}}^{\times}, \s{A}(\omega)) \rd \b{P}(\omega) = \b{E} \c{W}_{2}^{2}(\widehat{\s{A}}^{\times}, \s{A}) \ge \inf_{\m{G} \in \b{R}^{n \times n}} L(\m{G}),
\end{align*}
so we have
\begin{align}
    &\b{E} \c{W}_{2}^{2}(\widehat{\s{A}}^{\times}, \s{A}) = \inf_{\m{G} \in \b{R}^{n \times n}} L(\m{G}), \label{eq:rand:ker:opt1} \\
    &\tr[\s{A}(\omega) - 2 \tilde{\s{A}}^{\times}(\omega) + \widehat{\s{A}}^{\times}] = \vertii{(\widehat{\s{A}}^{\times})^{1/2} - \s{G}(\omega)}_{2}^{2}, \quad \b{P}-\text{a.s.}. \label{eq:rand:ker:opt2}
\end{align}
Then, from the proof of \cref{thm:char:barycenter}, \eqref{eq:rand:ker:opt1} yields that $\s{G} \in \c{G}(\s{A})$ is an optimal random Green's matrix and $(\widehat{\s{A}}^{\times})^{1/2} = \b{E} \s{G}$. Consequently, \eqref{eq:rand:ker:opt2} yields that $\tilde{\s{A}}^{\times}(\omega) = \s{G}(\omega) (\b{E} \s{G})^{*}$ holds $\b{P}$-a.s.,
or equivalently,
\begin{align*}
    \int_{\Omega} \s{A}^{\times}(\omega, \omega') \rd \b{P}(\omega') = \int_{\Omega} \s{G}(\omega)  \s{G}(\omega')^{*} \rd \b{P}(\omega'), \quad \b{P}-\text{a.s.}
\end{align*}
\end{proof}

\subsection{Kantorovich Problem}\label{ssec:kanto}
Consider $\Omega = \{0, 1\}$ with $\b{P}({0}) = 1-t, \b{P}({1}) = t$ for some $t \in  [0, 1]$, so the weighted barycenter is defined as
\begin{equation*}
    \m{\Gamma}_{t} := \argmin_{\m{C} \in \b{R}^{n \times n}, \ \m{C} \succeq \m{0}}  (1-t) \c{W}_{2}^{2}(\m{A}, \m{C}) + t \c{W}_{2}^{2}(\m{B}, \m{C}). 
\end{equation*}
Obviously, $\m{\Gamma}_{0} = \m{A}$ and $\m{\Gamma}_{1} = \m{B}$. Because $\vertii{\m{G}_{t}}_{2}^{2} = (1-t)^{2} \tr(\m{A}) + 2 t(1-t) \tr(\m{G}^{*} \m{M}) + t^{2} \tr(\m{B})$, the optimality condition for Green's matrices reduces to:
\begin{align}\label{eq:bary:char:n2}
    \m{G} \in \s{G}(\m{A}), \ \m{M} \in \s{G}(\m{B}), \ \m{G}^{*} \m{M} \succeq \m{0},
\end{align}
which does not depend on $t \in [0, 1]$.
The set of weighted barycenters is determined by
\begin{align}\label{eq:Kanto:geo}
    \c{K}_{t}(\m{A}, \m{B}) := \left\{ \m{\Gamma}_{t} = \m{G}_{t} \m{G}_{t}^{*} : \m{G}_{t} = (1-t) \m{G} + t \m{M}, \, \m{M} \in \s{G}(\m{B}), \ \m{G}^{*} \m{M} = (\m{G}^{*} \m{B} \m{G})^{1/2} \right\},
\end{align}
which is independent on the choice of $\m{G} \in \s{G}(\m{A})$ due to right unitary invariance, so one may fix it. When a Monge solution exists (i.e., $\m{A} \to \m{B}$ with an OT matrix $\m{T}$), setting $\m{M} = \m{T} \m{G}$ provides a valid Kantorovich solution, and the barycenter $\m{\Gamma}_{t}$ in \eqref{eq:Kanto:geo} recovers McCann's interpolant. From \cref{thm:multi:trans}, we note that a coupling of $N(\m{0}, \m{A})$ and $N(\m{0}, \m{B})$ is optimal for the Kantorovich problem if and only if its covariance matrix has the form
\begin{align}\label{eq:Kanto:block}
    \begin{pmatrix}
    \m{A} & \m{G} \m{M}^{*} \\
    \m{M} \m{G}^{*} & \m{B}
    \end{pmatrix} \in \b{R}^{2n \times 2n},
\end{align}
where the Green's matrix $\m{M} \in \b{R}^{n \times n}$ is a solution to \eqref{eq:bary:char:n2}.

\begin{theorem}[Kantorovich Problem]\label{thm:Kanto:char}
Let $\m{A}, \m{B}  \in \b{R}^{n \times n}$ be s.p.d. matrices with $\rk(\m{A}) \geq \rk(\m{B})$. We decompose $\m{A}, \m{B}$ in the block form as in 
\cref{thm:sol:sym}, and let
\begin{align*}
    \m{G} =
    \begin{pmatrix}
    \m{G}_{11} & \m{0} \\
    \m{0} & \m{0}
    \end{pmatrix}
    \    :
    \begin{array}{c}
    \c{H}_{1} \\
    \oplus \\
    \c{H}_{2}
    \end{array}
    \rightarrow
    \begin{array}{c}
    \c{H}_{1} \\
    \oplus \\
    \c{H}_{2}
    \end{array}, \quad \m{G}_{11} \in \c{G}(\m{A}_{11}).
\end{align*}
Then, the set of all matrices $\m{M} \in \s{G}(\m{B})$ satisfying \eqref{eq:bary:char:n2} is given by
\begin{align*}
    \m{M} =
    \begin{pmatrix}
    (\m{G}_{11}^{*})^{-1} (\m{G}_{11}^{*} \m{B}_{11} \m{G}_{11})^{1/2} & \m{0} \\
    \m{B}_{21} \m{G}_{11} (\m{G}_{11}^{*} \m{B}_{11} \m{G}_{11})^{\dagger/2} + \m{N}_{12}^{*} & \m{M}_{22}
    \end{pmatrix}
    \    :
    \begin{array}{c}
    \c{H}_{1} \\
    \oplus \\
    \c{H}_{2}
    \end{array}
    \rightarrow
    \begin{array}{c}
    \c{H}_{1} \\
    \oplus \\
    \c{H}_{2}
    \end{array} ,
\end{align*}
where: 
\begin{itemize}[leftmargin = *]
    \item $\m{N}_{12}  : \c{H}_{2} \to \c{H}_{1}$ satisfies $\c{R}(\m{N}_{12}) \subset \c{N}(\m{B}_{11}^{1/2} \m{G}_{11})$ and $\m{N}_{12}^{*} \m{N}_{12} \preceq \m{B}/\m{A}$.
    \item $\m{M}_{22}  : \c{H}_{2} \to \c{H}_{2}$ satisfies $\m{M}_{22} \m{M}_{22}^{*} = \m{B}/\m{A} - \m{N}_{12}^{*} \m{N}_{12}$.
\end{itemize}
A solution corresponds to a Monge map (i.e., $\m{M} = \m{T}\m{G}$ for some OT matrix $\m{T}$) if and only if one chooses $\m{M}_{22} = \m{0}$ (i.e., $\m{N}_{12}^{*} \m{N}_{12} = \m{B}/\m{A}$).
\end{theorem}
\begin{proof}[Proof of \cref{thm:Kanto:char}]
Decompose $\m{M}$ as:
\begin{equation*}
    \m{M} = \begin{pmatrix}
    \m{M}_{11} & \m{M}_{12} \\
    \m{M}_{21} & \m{M}_{22}
    \end{pmatrix}.
\end{equation*}
Then $\m{G}^{*} \m{M} \succeq \m{0}$ if and only if
\begin{align*}
    \begin{pmatrix}
    \m{G}_{11}^{*} \m{M}_{11} & \m{G}_{11}^{*} \m{M}_{12} \\
    \m{0} & \m{0}
    \end{pmatrix} =
    \m{G}^{*} \m{M} = (\m{G}^{*} \m{B} \m{G})^{1/2} = \begin{pmatrix}
    (\m{G}_{11}^{*} \m{B}_{11} \m{G}_{11})^{1/2} & \m{0} \\
    \m{0} & \m{0}
    \end{pmatrix},
\end{align*}
which happens if and only if $\m{M}_{11} = (\m{G}_{11}^{*})^{-1} (\m{G}_{11}^{*} \m{B}_{11} \m{G}_{11})^{1/2}$ and $\m{M}_{12} = \m{0}$. With these values, the condition $\m{M} \in \c{G}(\m{B})$ becomes
\begin{align*}
    \begin{pmatrix}
    \m{B}_{11} & \m{M}_{11} \m{M}_{21}^{*} \\
    \m{M}_{21} \m{M}_{11}^{*} & \m{M}_{22} \m{M}_{22}^{*} + \m{M}_{21} \m{M}_{21}^{*}
    \end{pmatrix} =
    \m{M} \m{M}^{*} = \m{B} = \begin{pmatrix}
    \m{B}_{11} & \m{B}_{12} \\
    \m{B}_{21} & \m{B}_{22}
    \end{pmatrix},
\end{align*}
which is equivalent to the following system of block equations:
\begin{align}
    &(\m{G}_{11}^{*})^{-1} (\m{G}_{11}^{*} \m{B}_{11} \m{G}_{11})^{1/2} \m{M}_{21}^{*} = \m{B}_{12}, \label{eq:reach:rk:1} \\
    &\m{M}_{22} \m{M}_{22}^{*} + \m{M}_{21} \m{M}_{21}^{*} = \m{B}_{22}. \label{eq:reach:rk:2}
\end{align}
As in the proof of \cref{thm:sol:sym},  \eqref{eq:reach:rk:1} always have a solution characterized by
\begin{equation*}
    \m{M}_{21}^{*} = (\m{G}_{11}^{*} \m{B}_{11} \m{G}_{11})^{\dagger/2} (\m{G}_{11}^{*} \m{B}_{12}) + \m{N}_{12}, \quad \c{R}(\m{N}_{12}) \subset \c{N}((\m{G}_{11}^{*} \m{B}_{11} \m{G}_{11})^{1/2}) = \c{N}(\m{B}_{11}^{1/2} \m{G}_{11}).
\end{equation*}
Substituting $\m{N}_{12}$ into \eqref{eq:reach:rk:2}, we obtain from part 3 in \cref{lem:Schur:op} that:
\begin{align*}
    \m{M}_{22} \m{M}_{22}^{*} + \m{N}_{12}^{*} \m{N}_{12} = \m{B}_{22} - (\m{B}_{21} \m{G}_{11}) (\m{G}_{11}^{*} \m{B}_{11} \m{G}_{11})^{\dagger} (\m{G}_{11}^{*} \m{B}_{12}) = \m{B}/\m{A} \succeq \m{0}.
\end{align*}
Finally, if $\m{M}_{22} = \m{0}$, the above equation becomes precisely \eqref{eq:sol:resi} in the proof of \cref{thm:sol:sym}. Therefore, $\m{M}_{22} = \m{0}$ if and only if there exists some OT matrix $\m{T} \in \b{R}^{n \times n}$ such that $\m{M} = \m{T} \m{G}$.
\end{proof}

\begin{remark}\label{rmk:linear:kanto}
The off-diagonal block, $\m{G} \m{M}^{*}$, of the  covariance for the optimal Kantorovich coupling in \eqref{eq:Kanto:block} has the form
\begin{equation*}
    \m{G} \m{M}^{*} = \begin{pmatrix}
    \m{G}_{11} (\m{G}_{11}^{*} \m{B}_{11} \m{G}_{11})^{1/2} \m{G}_{11}^{-1}  & \m{G}_{11} (\m{G}_{11}^{*} \m{B}_{11} \m{G}_{11})^{\dagger/2} \m{G}_{11} \m{B}_{12}  + \m{G}_{11} \m{N}_{12} \\
    \m{0} & \m{0}
    \end{pmatrix}
    \    :
    \begin{array}{c}
    \c{H}_{1} \\
    \oplus \\
    \c{H}_{2}
    \end{array}
    \rightarrow
    \c{H}_{1}.
\end{equation*}
This calculation reveals two key properties. First, $\m{G} \m{M}^{*}$ is independent of the particular choice of $\m{G}_{11} \in \c{G}(\m{A}_{11})$ and $\m{M}_{22} \in \c{G}(\m{B}/\m{A} - \m{N}_{12}^{*} \m{N}_{12})$. Second, this implies that the diagonal blocks of any barycenter $\m{\Gamma}_{t} \in \c{K}_{t}(\m{A}, \m{B})$ are fixed:
\begin{align*}
     &(\m{\Gamma}_{t})_{11} = (1-t)^{2} \m{A}_{11} + t(1-t) [\m{G}_{11} (\m{G}_{11}^{*} \m{B}_{11} \m{G}_{11})^{1/2} \m{G}_{11}^{-1} + \m{G}_{11}^{-1} (\m{G}_{11}^{*} \m{B}_{11} \m{G}_{11})^{1/2} \m{G}_{11}] + t^{2} \m{B}_{11}, \\
     &(\m{\Gamma}_{t})_{22} = t^{2} \m{B}_{22}, \\
     &(\m{\Gamma}_{t})_{12} = t(1-t) [\m{G}_{11} (\m{G}_{11}^{*} \m{B}_{11} \m{G}_{11})^{\dagger/2} \m{G}_{11} \m{B}_{12}  + \m{G}_{11} \m{N}_{12}] + t^{2} \m{B}_{12}, \quad (\m{\Gamma}_{t})_{21} = (\m{\Gamma}_{t})_{12}^{*}.
\end{align*}
The entire geometry of the set of barycenters is determined by the choice of the submatrix $\m{N}_{12}$, subject to the conditions in \cref{thm:Kanto:char}. To emphasize this dependence, we sometimes write the barycenter as $\m{\Gamma}_{t}^{\m{N}_{12}}$ in the rest of this subsection.
\end{remark}

This observation reveals the geometric structure of the set of barycenters, generalizing \cref{thm:geo:Schur}.

\begin{theorem}[Extreme Geodesics]\label{thm:ext:geo}
Let $\m{A}, \m{B}  \in \b{R}^{n \times n}$ be s.p.d. matrices with $\rk(\m{A}) \geq \rk(\m{B})$, and let $t \in (0, 1)$. The set of barycenters $\c{K}_{t}(\m{A}, \m{B})$ has the following geometric properties:
\begin{enumerate}[leftmargin = *]
\item $\m{\Gamma}_{t}$ is an extreme point of $\c{K}_{t}(\m{A}, \m{B})$ if and only if it is a McCann interpolant at time $t$.
\item $\c{K}_{t}(\m{A}, \m{B})$ is the closed convex hull of the set of all McCann interpolants at time $t$.
\item For any barycenter $\m{\Gamma}_{t}^{\m{N}_{12}} \in \c{K}_{t}(\m{A}, \m{B})$, we have
\begin{align*}
    \m{\Gamma}_{t}^{\m{N}_{12}}/\m{A} = t^{2} (\m{B}/\m{A} - \m{N}_{12}^{*} \m{N}_{12}), \quad \rk(\m{\Gamma}_{t}^{\m{N}_{12}}) = \rk(\m{A}) + \rk (\m{\Gamma}_{t}^{\m{N}_{12}}/\m{A}).
\end{align*}
\end{enumerate}
Consequently, the following statements are equivalent:
\begin{enumerate}[leftmargin = *]
\item [a.] $\m{\Gamma}_{t}$ is a McCann interpolant at time $t$.
\item [b.] Its $\m{A}$-Schur complement vanishes: $\m{\Gamma}_{t}/\m{A} = \m{0}$.
\item [c.] Its rank is preserved: $\rk(\m{\Gamma}_{t}) = \rk(\m{A})$.
\end{enumerate}
\end{theorem}

\begin{proof}[Proof of \cref{thm:ext:geo}] \quad
\begin{enumerate}[leftmargin = *]
\item We first show that the set given by
\begin{align*}
    \c{A} := \left\{\m{N}_{12}  : \c{H}_{2} \to \c{H}_{1} \mid \c{R}(\m{N}_{12}) \subset \c{N}(\m{B}_{11}^{1/2} \m{G}_{11}), \ \m{N}_{12}^{*} \m{N}_{12} \preceq \m{B}/\m{A} \right\},
\end{align*}
is a closed convex set. The closedness of $\c{A}$ is trivial. To show convexity, consider $\m{P}_{\lambda} = (1-\lambda) \m{P}_{0} + \lambda \m{P}_{1}$, where $\m{P}_{0}, \m{P}_{1} \in \c{A}$ and $\lambda \in (0, 1)$. We have $\c{R}(\m{P}_{\lambda}) \subset \c{R}(\m{P}_{0}) + \c{R}(\m{P}_{1}) \subset \c{N}(\m{B}_{11}^{1/2} \m{G}_{11})$. Also, observe $[(1-\lambda) \m{P}_{0}^{*} \m{P}_{0} + \lambda \m{P}_{1}^{*} \m{P}_{1}] - \m{P}_{\lambda}^{*} \m{P}_{\lambda} = \lambda (1 - \lambda) (\m{P}_{0} - \m{P}_{1})^{*} (\m{P}_{0} - \m{P}_{1}) \succeq \m{0}$ with equality if and only if $\m{P}_{0} = \m{P}_{1}$, which leads to
\begin{align}\label{eq:displ:cvx:Schur}
    \m{B}/\m{A} \succeq (1-\lambda) \m{P}_{0}^{*} \m{P}_{0} + \lambda \m{P}_{1}^{*} \m{P}_{1} \succeq \m{P}_{\lambda}^{*} \m{P}_{\lambda}.
\end{align}
Thus, $\m{P}_{\lambda} \in \c{A}$. 

We claim that the set of extreme points $\c{E}$ of $\c{A}$ is given by 
\begin{align*}
    \c{E} = \left\{\m{N}_{12}  : \c{H}_{2} \to \c{H}_{1} \mid \c{R}(\m{N}_{12}) \subset \c{N}(\m{B}_{11}^{1/2} \m{G}_{11}), \ \m{N}_{12}^{*} \m{N}_{12} = \m{B}/\m{A} \right\}.
\end{align*}
It is trivial that $\c{E} \supseteq \text{(RHS)}$ since the equality in \eqref{eq:displ:cvx:Schur} holds if and only if $\m{P}_{0} = \m{P}_{1} \in \c{E}$.
To show the other direction, we may assume $\m{B}/\m{A} \ne \m{0}$, otherwise trivial. Suppose $\m{N}_{12} \in \c{A}$ while $\m{N}_{12}^{*} \m{N}_{12} \ne \m{B}/\m{A}$. Then, there exists a non-zero vector $\m{w} \in \c{H}_{2}$ such that $\m{w} \m{w}^{*} \preceq \m{B}/\m{A} - \m{N}_{12}^{*} \m{N}_{12}$.

\begin{itemize}
\item [(Case 1)] \underline{$\c{R}(\m{N}_{12}) \subsetneq \c{N}(\m{B}_{11}^{1/2} \m{G}_{11})$}. 

Fix a unit vector $\m{v} \in \c{N}(\m{B}_{11}^{1/2} \m{G}_{11}) \cap \c{R}(\m{N}_{12})^{\perp}$. Define $\m{R}_{12} = \m{v} \m{w}^{*} : \c{H}_{2} \to \c{H}_{1}$. Then, we have $\m{N}_{12}^{*} \m{R}_{12} = \m{0}$, thus
\begin{align*}
    (\m{N}_{12} \pm \m{R}_{12})^{*} (\m{N}_{12} \pm \m{R}_{12}) = \m{N}_{12}^{*} \m{N}_{12} + \m{w} \m{w}^{*} \preceq \m{B}/\m{A}.
\end{align*}
Therefore, we have $(\m{N}_{12} + \m{R}_{12}), (\m{N}_{12} - \m{R}_{12}) \in \c{A}$, i.e., $\m{N}_{12}$ is not an extreme point of $\c{A}$.

\item [(Case 2)] \underline{$\c{R}(\m{N}_{12}) = \c{N}(\m{B}_{11}^{1/2} \m{G}_{11})$}. 

Note from \cref{cor:Schur:rank} that
\begin{align*}
    \rk(\m{N}_{12}^{*}) = \dim \c{N}(\m{B}_{11}^{1/2} \m{G}_{11}) &= \rk(\m{A}) - \rk[(\m{G}_{11}^{*} \m{B}_{11} \m{G}_{11})^{1/2}] \\
    &\ge \rk(\m{B}) - \rk[(\m{G}_{11}^{*} \m{B}_{11} \m{G}_{11})^{1/2}] = \rk(\m{B}/\m{A}) \ge \rk(\m{N}_{12}^{*}).
\end{align*}
Since $\c{R}(\m{N}_{12}^{*}) \subset \c{R}(\m{B}/\m{A})$, this entails that $\c{R}(\m{N}_{12}^{*}) = \c{R}(\m{B}/\m{A})$ and the restriction of the adjoint $\m{N}_{12}^{*} : \c{N}(\m{B}_{11}^{1/2} \m{G}_{11}) \to \c{R}(\m{B}/\m{A})$ is invertible. Therefore, there exists some $\m{v} \in \c{N}(\m{B}_{11}^{1/2} \m{G}_{11})$ such that $\m{N}_{12}^{*} \m{v} = \m{w}$. Choose $\varepsilon > 0$ such that $(2 \varepsilon + \varepsilon^{2} \|\m{v}\|^{2}) \le 1$. Consequently, we obtain
\begin{align*}
    (\m{N}_{12} + \varepsilon \m{v} \m{w}^{*})^{*} (\m{N}_{12} + \varepsilon \m{v} \m{w}^{*}) = \m{N}_{12}^{*} \m{N}_{12} + (2 \varepsilon + \varepsilon^{2} \|\m{v}\|^{2}) \m{w} \m{w}^{*} \preceq \m{B}/\m{A}, \\
    (\m{N}_{12} - \varepsilon \m{v} \m{w}^{*})^{*} (\m{N}_{12} - \varepsilon \m{v} \m{w}^{*}) = \m{N}_{12}^{*} \m{N}_{12} + (-2 \varepsilon + \varepsilon^{2} \|\m{v}\|^{2}) \m{w} \m{w}^{*} \preceq \m{B}/\m{A}.
\end{align*}
Finally, we have $(\m{N}_{12} + \varepsilon \m{v} \m{w}^{*}), (\m{N}_{12} - \varepsilon \m{v} \m{w}^{*}) \in \c{A}$, i.e., $\m{N}_{12}$ is not an extreme point of $\c{A}$.
\end{itemize}
Then, from \cref{rmk:linear:kanto}, it follows that $\c{K}_{t}(\m{A}, \m{B})$ is the closed convex set. It also follows that $\m{\Gamma}_{t} \in \c{K}_{t}(\m{A}, \m{B})$ is an extreme point if and only if $\m{N}_{12} \in \c{E}$, which is equivalent to the fact that $\m{\Gamma}_{t}$ is a McCann interpolant at time $t$ due to \cref{thm:Kanto:char}.
\item Direct implication of the Krein-Milman theorem \cite{conway2019course}.
\end{enumerate}

The part 3 will be shown while establishing the second set of equivalence relations:
\begin{enumerate}
\item [$(a \Rightarrow b)$] See \cref{thm:geo:Schur}.
\item [$(b \Rightarrow a)$] We prove it by contradiction. Suppose $\m{\Gamma}_{t}/\m{A} = \m{0}$ while $\m{\Gamma}_{t} \in \c{K}_{t}(\m{A}, \m{B})$ is not a McCann interpolant. From \cref{thm:Kanto:char}, we get $\m{M}_{22} \ne \m{0}$. Using the block decomposition as in the proof of \cref{thm:Kanto:char}, we obtain
\begin{align*}
    (\m{\Gamma}_{t})_{11} &= [(1-t) \m{G}_{11} + t \m{M}_{11}] [(1-t) \m{G}_{11} + t \m{M}_{11}]^{*}, \\
    (\m{\Gamma}_{t})_{12} &= t(1-t) \m{G}_{11} \m{M}_{21}^{*} + t^{2} \m{M}_{11} \m{M}_{21}^{*}
    = [(1-t) \m{G}_{11} + t \m{M}_{11}][t \m{M}_{21}]^{*} \\
    (\m{\Gamma}_{t})_{22} &= t^{2} \m{B}_{22} = t^{2}[\m{M}_{22} \m{M}_{22}^{*} + \m{M}_{21} \m{M}_{21}^{*}].
\end{align*}
To ease the notation, let us denote 
\begin{align*}
    \m{X}_{11} := [(1-t) \m{G}_{11} + t \m{M}_{11}] = [(1-t) \m{I}_{11} + t (\m{G}_{11}^{*})^{-1} (\m{G}_{11}^{*} \m{B}_{11} \m{G}_{11})^{1/2} \m{G}_{11}^{-1}] \m{G}_{11}, 
\end{align*}
which is injective. Then, the $\m{A}$-Schur complement of $\m{\Gamma}_{t}$ is given by
\begin{align*}
    \m{\Gamma}_{t}/\m{A} &= (\m{\Gamma}_{t})_{22} - (\m{\Gamma}_{t})_{21} (\m{\Gamma}_{t})_{11}^{\dagger} (\m{\Gamma}_{t})_{12} \\
    &= t^{2} \m{M}_{22} \m{M}_{22}^{*} + t^{2} \m{M}_{21} \left[ \m{I}_{11} - \m{X}_{11}^{*} (\m{X}_{11} \m{X}_{11}^{*})^{\dagger} \m{X}_{11}  \right] \m{M}_{21}^{*} \\
    &= t^{2} \m{M}_{22} \m{M}_{22}^{*} + t^{2} \m{M}_{21} \m{\Pi}_{\c{N}(\m{X}_{11})} \m{M}_{21}^{*} = t^{2} \m{M}_{22} = t^{2} [\m{B}/\m{A} - \m{N}_{12}^{*} \m{N}_{12}],
\end{align*}
hence it does not vanish, which is a contradiction.
\item [$(b \Leftrightarrow c)$] For $\m{\Gamma}_{t} \in \c{K}_{t}(\m{A}, \m{B})$, $(\m{\Gamma}_{t})_{11}$ is full-rank since $(\m{\Gamma}_{t})_{11} \succeq (1-t)^{*} \m{A}_{11}$. Hence,
\begin{align*}
    \rk (\m{G}^{*} \m{\Gamma}_{t} \m{G}) = \rk (\m{G}_{11}^{*} (\m{\Gamma}_{t})_{11} \m{G}_{11})
    = \rk ((\m{\Gamma}_{t})_{11}) = \rk (\m{G}_{11}) = \rk (\m{A}),
\end{align*}
Then, \cref{cor:Schur:rank} concludes
\begin{align*}
    \rk(\m{\Gamma}_{t}) = \rk (\m{G}^{*} \m{\Gamma}_{t} \m{G}) + \rk(\m{\Gamma}_{t}/\m{A}) = \rk (\m{A}) + \rk(\m{\Gamma}_{t}/\m{A}),
\end{align*}
and the equivalence becomes trivial.
\end{enumerate}
\end{proof}

We remark that
\begin{align*}
    \rk(\m{A}) \le \rk(\m{\Gamma}_{t}) \le \rk(\m{A}) + \rk(\m{B}/\m{A}) = \rk(\m{A}) + \rk(\m{B}) - \rk(\m{B}^{1/2} \m{A}^{1/2}).
\end{align*}
The geometry of the barycenter set also provides another perspective on s.p.d. reachability in \cref{thm:spd:sol:sym}, connecting it to the uniqueness of the Kantorovich barycenter.

\begin{corollary}[SPD Reachability Revisited]\label{thm:spd:revise}
Let $\m{A}, \m{B}  \in \b{R}^{n \times n}$ be s.p.d. matrices. For any $t \in (0, 1)$, the set of barycenters $\c{K}_{t}(\m{A}, \m{B})$ is a singleton if and only if the equivalent conditions in \cref{thm:spd:sol:sym} hold.
Furthermore, this unique barycenter is the McCann interpolant at time $t$.
\end{corollary}
\begin{proof}[Proof of \cref{thm:spd:revise}]
Trivial from \cref{thm:Kanto:char,thm:ext:geo}.
\end{proof}

The preceding discussion focused on the set of barycenters at a fixed time $t$. To form a geodesic path from $\m{A}$ to $\m{B}$, one must consistently select the free parameter $\m{N}_{12}$ for all $t \in (0,1)$. Each valid choice of $\m{N}_{12}$ in \cref{thm:Kanto:char} defines a unique geodesic. We call the collection of all such paths the \emph{Kantorovich geodesics}.

\begin{theorem}[Kantorovich Geodesics]\label{thm:kanto:geo}
Let $\m{A}, \m{B}  \in \b{R}^{n \times n}$ be s.p.d. matrices with $\rk(\m{A}) \geq \rk(\m{B})$. 
\begin{enumerate}[leftmargin = *]
\item Any path $\m{\Gamma}^{\m{N}_{12}}: [0, 1] \to \b{R}^{n \times n}, \, t \mapsto \m{\Gamma}_{t}^{\m{N}_{12}}$ constructed by fixing a choice of $\m{N}_{12}$ in \cref{thm:Kanto:char} generates a constant-speed geodesic, that is, $\c{W}_{2}(\m{\Gamma}_{s}^{\m{N}_{12}}, \m{\Gamma}_{t}^{\m{N}_{12}}) = (t-s) \c{W}_{2}(\m{A}, \m{B})$ for any $0 \le t \le s \le 1$.
\item Conversely, any constant-speed geodesic connecting $\m{A}$ and $\m{B}$ is of this form.
\item The rank of $\m{\Gamma}^{\m{N}_{12}}$ is constant for $t \in (0, 1)$ and is given by $\rk(\m{A}) + \rk (\m{B}/\m{A} - \m{N}_{12}^{*} \m{N}_{12})$.
\end{enumerate}
\end{theorem}

\begin{proof}[Proof of \cref{thm:kanto:geo}] \quad
\begin{enumerate}[leftmargin = *]
\item Since $\m{M}^{*} \m{G} = \m{G}^{*} \m{M} \succeq \m{0}$, note that, for any $0 \le t \le s \le 1$, 
\begin{align*}
    \m{G}_{t}^{*} \m{G}_{s} &= \left( [(1-t) \m{G} + t \m{M}]^{*} [(1-s) \m{G} + s \m{M}] \right) \\
    &= (1-t)(1-s) \m{G}^{*} \m{G} + ts \tr (\m{B}) + [(1-t)s + (1-s)t] (\m{G}^{*} \m{M}) = \m{G}_{s}^{*} \m{G}_{t} \succeq \m{0}.
\end{align*}
By \cref{lem:opt:algn}, we obtain
\begin{align*}
    \c{W}_{2}^{2}(\m{\Gamma}_{t}, \m{\Gamma}_{s}) &= \tr(\m{\Gamma}_{t}) + \tr(\m{\Gamma}_{t}) - 2 \tr [(\m{\Gamma}_{t}^{1/2} \m{\Gamma}_{s} \m{\Gamma}_{t}^{1/2})^{1/2}] \\
    &= \tr(\m{G}_{t}^{*} \m{G}_{t}) + \tr(\m{G}_{s}^{*} \m{G}_{s}) - 2
    \tr(\m{G}_{t}^{*} \m{G}_{s}) 
    = \tr([\m{G}_{t} - \m{G}_{s}]^{*} [\m{G}_{t} - \m{G}_{s}]) \\
    &= (t-s)^{2} \tr([\m{G} - \m{M}]^{*} [\m{G} - \m{M}]) = (t-s)^{2} \c{W}_{2}^{2}(\m{A}, \m{B}).
\end{align*}
\item Let $\tilde{\m{\Gamma}}: [0, 1] \to \b{R}^{n \times n}$ be a constant-speed geodesic connecting $\m{A}$ and $\m{B}$. Fix $t \in (0, 1)$ and we claim that $\tilde{\m{\Gamma}}_{t} \in \c{K}_{t}(\m{A}, \m{B})$. For completeness, we provide the proof of the well-known sectional curvature inequality of the 2-Wasserstein distance \cite{villani2008optimal, panaretos2020invitation,yun2023exponential}. By the Cauchy-Schwarz inequality, we get
\begin{align*}
    (1-t) \c{W}_{2}^{2}(\m{A}, \m{C}) + t \c{W}_{2}^{2}(\m{B}, \m{C}) \ge 
    \left(\frac{1}{1-t} + \frac{1}{t} \right)^{-1} (\c{W}_{2}(\m{A}, \m{C}) + \c{W}_{2}(\m{B}, \m{C}))^{2} \ge
    t (1-t) \c{W}_{2}^{2}(\m{A}, \m{B}),
\end{align*}
and the equality holds if and only if $t^{-1} \c{W}_{2}(\m{A}, \m{C}) = (1-t)^{-1} \c{W}_{2}(\m{B}, \m{C}) = \c{W}_{2}(\m{A}, \m{B})$. This shows that $\tilde{\m{\Gamma}}_{t} \in \c{K}_{t}(\m{A}, \m{B})$, i.e., $\tilde{\m{\Gamma}}_{t} = \m{\Gamma}_{t}^{\m{N}_{12}}$ for some $\m{N}_{12}$ given in \cref{thm:Kanto:char}. 

It remains to show that $\tilde{\m{\Gamma}}_{t} = \m{\Gamma}_{t}^{\m{N}_{12}}$ for all $t \in (0, 1)$. Note that $\m{M} \in \c{G}(\m{B}) $ associated with $\m{N}_{12}$ satisfies $\m{G}^{*} \m{M} \succeq \m{0}$. Then, we obtain from \cref{cor:bary:injec} that the reverse Schur complement vanishes: $\m{A}/\m{\Gamma}_{t}^{\m{N}_{12}} = \m{0}$. Due to \cref{thm:spd:revise} and part 1,  $\m{\Gamma}^{\m{N}_{12}}: [t, 0] \to \b{R}^{n \times n}$ (reversed time) is the unique constant-speed geodesic connecting $\m{\Gamma}_{t}^{\m{N}_{12}}$ and $\m{A}$. Repeating the proof, we get that $\m{\Gamma}^{\m{N}_{12}}: [t, 1] \to \b{R}^{n \times n}$ is the unique constant-speed geodesic connecting $\m{\Gamma}_{t}^{\m{N}_{12}}$ and $\m{B}$. 
\item This follows directly from \cref{thm:ext:geo}.
\end{enumerate}
\end{proof}

\begin{example}
Recall s.p.d. matrices $\m{A} = \diag(4, 1, 0)$ and $\m{C} = \diag(0, 0 , 1)$ in \cref{ex:spd:ex}. From our earlier calculation  in \cref{ex:spd:ex}, all the Kantorovich geodesics connecting $\m{A}$ and $\m{C}$ are given by
\begin{equation*}
    \m{\Gamma}: [0, 1] \to \b{R}^{3 \times 3}, \, t \mapsto \m{\Gamma}_{t} = \begin{pmatrix}
    4(1-t)^{2} & 0 & 2t (1-t) x  \\
    0 & (1-t)^{2} & t (1-t) y \\
    2t (1-t) x & t (1-t) y & t^{2}
    \end{pmatrix}, \quad x^{2} + y^{2} \le 1,
\end{equation*}
and they are McCann interpolants if and only if $x^{2} + y^{2} = 1$. Also, as predicted by \cref{thm:kanto:geo}, we have
\begin{align*}
    \rk(\m{\Gamma}_{t}) = \begin{cases}
        2 &, \quad x^{2} + y^{2} = 1, \\
        3 &, \quad x^{2} + y^{2} < 1,
    \end{cases} \quad t \in (0, 1).
\end{align*}
\end{example}

\subsection{Barycenter Algorithm}
Let $\m{A}_{1}, \cdots, \m{A}_{m}  \in \b{R}^{n \times n}$ be s.p.d. matrices, and $p_{1}, \cdots, p_{m} > 0$ be the weights with $p_{1} + \cdots + p_{m} = 1$. Recall from \cref{thm:char:barycenter} that finding the barycenter amounts to solving the maximization problem: 
\begin{equation*}
    \max \{ \vertii{\hat{\m{G}}}_{2}^{2} : \hat{\m{G}} = \sum_{i=1}^{m} p_{i} \m{G}_{i}, \, \m{G}_{i} \in \c{G}(\m{A}_{i}), \, i =1, \cdots, m \}.
\end{equation*}
We outline an efficient iterative algorithm based on block coordinate ascent to solve this problem, different to \cite{bhatia2019bures,pigoli2014distances,masarotto2019procrustes}. A detailed analysis of its performance is beyond the scope of this paper, and be a future topic. 

In each main iteration, we cycle through $i = 1, 2, \cdots, m$ and update each Green's operator $\m{G}_{i} \in \c{G}(\m{A}_{i})$ to maximize the objective function while holding all other operators $\m{G}_{j}$ (for $j \ne i$) fixed. To derive the update rule for a specific $\m{G}_{i}$, let $\hat{\m{G}}_{-i} := \sum_{j \ne i} p_{j} \m{G}_{j}$. The objective can then be written as:
\begin{align*}
    \vertii{p_{i} \m{G}_{i} + \hat{\m{G}}_{-i} }_{2}^{2} = 2 p_{i} \tr [ (\hat{\m{G}}_{-i})^{*} \m{G}_{i} ] + p_{i}^{2} \tr(\m{A}_{i}) + \vertii{\hat{\m{G}}_{-i}}_{2}^{2},
\end{align*}
which is maximized if and only if $\m{G}_{i}^{\text{new}}$ is the solution to 
\begin{align*}
    \m{G}_{i} \in \c{G}(\m{A}_{i}), \quad (\hat{\m{G}}_{-i})^{*} \m{G}_{i} = [ (\hat{\m{G}}_{-i})^{*} \m{A}_{i} \hat{\m{G}}_{-i} ]^{1/2},
\end{align*}
which can easily be solved either (i) using the (possibly rank-deficient) SVD decomposition of $(\hat{\m{G}}_{-i})^{*} \m{A}_{i}^{1/2}$ outlined in \cref{alg:wba}, or (ii) \eqref{eq:bary:char:n2} using \cref{thm:Kanto:char}. By construction, each update increases (or leaves unchanged) the value of the objective function, ensuring that the algorithm is monotonically convergent, particularly to the global solution provided uniqueness. Also, one may incorporate a natural stopping criterion: the process can be terminated when the increase in the objective function is negligible:
\begin{equation*}
    \vertii{p_{i} \m{G}_{i}^{\text{new}} + \hat{\m{G}}_{-i}}_{2}^{2} - \vertii{p_{i} \m{G}_{i}^{\text{old}} + \hat{\m{G}}_{-i} }_{2}^{2} = 2 p_{i} \tr \left[ (\hat{\m{G}}_{-i})^{*} (\m{G}_{i}^{\text{new}} - \m{G}_{i}^{\text{old}})  \right] \ge 0
\end{equation*}

\begin{algorithm}[h!]
\caption{Weighted Wasserstein Barycenter Algorithm}
\label{alg:wba}
\begin{algorithmic}[1]
\Require
    \Statex A set of s.p.d. matrices $\m{A}_{1}, \dots, \m{A}_{m} \in \b{R}^{n \times n}$.
    \Statex A set of weights $p_{1}, \dots, p_{m} > 0$ with $\sum_{i=1}^{m} p_{i} = 1$.
    \Statex Maximum iteration counts $\text{max\_iter} \in \b{N}$, and tolerance $\text{tol} > 0$.
\Ensure
    \Statex The weighted Wasserstein barycenter $\hat{\m{A}} \in \b{R}^{n \times n}$.

\Statex
\Statex \Comment{\textit{--- Initialization ---}}
\For{$i = 1 \to m$}
    \State Solve for $\m{G}_{i}$ in $\m{G}_{i} \m{G}_{i}^{*} = \m{A}_{i}$ \Comment{e.g., via Cholesky decomposition}
\EndFor
\State $\hat{\m{G}} \gets \sum_{i=1}^{m} p_i \m{G}_i$ \Comment{Initialize the weighted average of Green's matrices}

\Statex
\Statex \Comment{\textit{--- Main Iteration ---}}
\For{$k = 1 \to \text{max\_iter}$}
    \For{$i = 1 \to m$}
        \State $\m{G}_i^{\text{old}} \gets \m{G}_i$
        \State $\hat{\m{G}} \gets \hat{\m{G}} - p_i \m{G}_i^{\text{old}}$ \Comment{Temporarily remove component $i$}
        
        \State Solve for $\m{G}_i$ in: $\m{G}_i \m{G}_i^{*} = \m{A}_i$,  $\hat{\m{G}}^* \m{G}_i = (\hat{\m{G}}^* \m{A}_i \hat{\m{G}})^{1/2}$ \Comment{e.g., Golub-Kahan Bidiagonalization \cite{golub2013matrix}}
        
        \State $\hat{\m{G}} \gets \hat{\m{G}} + p_i \m{G}_i$ \Comment{Re-incorporate the updated component}

        \If{$\tr[\hat{\m{G}}^* (\m{G}_i - \m{G}_i^{\text{old}})] < \text{tol}$}
            \State \textbf{return} $\hat{\m{G}} \hat{\m{G}}^*$ \Comment{Early stopping on convergence}
        \EndIf
    \EndFor
\EndFor

\State \textbf{return} $\hat{\m{G}} \hat{\m{G}}^*$ \Comment{Return result after max iterations}
\end{algorithmic}
\end{algorithm}

\section{Validity of Pushforward}\label{sec:admissible}
Henceforth, we assume $\c{H}$ is an infinite dimensional separable $\b{R}$-Hilbert space. 
Extending the concept of reachability between Gaussian measures to this setting introduces two fundamental challenges: defining suitable operator domains and accommodating unbounded transport maps.

Consider a Gaussian random element $X$ in $\c{H}$ with a covariance operator $\m{A}: \c{H} \rightarrow \c{H}$. It is well established that the covariance operator of a Gaussian element must be of trace class. Conversely, for any trace-class operator $\m{A} \in \c{B}_{1}^{+}(\c{H})$, there exists a corresponding Gaussian measure \cite{da2006introduction}. We can therefore associate a centered Gaussian measure with its covariance operator $\m{A} \in \c{B}_{1}^{+}(\c{H})$. Since $\m{A}$ is compact, the open mapping theorem prevents it from being bijective \cite{conway2019course}. This theorem also establishes the equivalence of the following statements:
\begin{itemize}
    \item $\rk (\m{A}) < \infty$, i.e. $\m{A}$ has a finite-dimensional range.
    \item $\c{R}(\m{A})$ is a closed subspace in $\c{H}$.
    \item Equality holds in $\c{R}(\m{A}) \subseteq \c{R}(\m{A}^{1/2})$.
\end{itemize}

The compactness of $\m{A}$ introduces a subtlety concerning the support of the associated Gaussian measure.
A key property is that $\b{P}(X \in \overline{\c{R}(\m{A})})=1$ \cite{hsing2015theoretical}. However, if $\m{A}$ has an infinite rank, the Karhunen-Lo\`{e}ve expansion and the Picard criterion show that $\b{P}(X \in \c{R}(\m{A}^{1/2}))=0$ \cite{bogachev1998gaussian, hairer2009introduction}. A classical example is Brownian motion (BM) in $\c{H} = L_{2} [0, 1]$, where $\c{R}(\m{A}^{1/2})$ is the Cameron-Martin space (see \cref{ex:BM2BB}) \cite{nualart2006malliavin, paulsen2016introduction, da2006introduction}. 
This disparity raises a dilemma: \emph{if we wish to apply an unbounded transport operator $\m{T}$, its domain may not be the entire space. What, then, should its essential domain be?}

A second challenge arises from the counterintuitive nature of unbounded operators in this context. Let $\{e_{i} : i \in \b{N} \}$ be an ONB for $\c{H}$, and consider the covariance operators
\begin{equation*}
    \m{A} = \sum_{n = 1}^{\infty} n^{-a} e_{n} \otimes e_{n}, \quad \m{B} = \sum_{n = 1}^{\infty} n^{-b} e_{n} \otimes e_{n},
\end{equation*}
where $a, b > 1$ ensures both operators are trace-class. A natural candidate for the operator that pushes forward the measure $N_{\c{H}}(\m{0}, \m{A})$ to $N_{\c{H}}(\m{0}, \m{B})$ would be
\begin{equation*}
    \m{T} = \sum_{n = 1}^{\infty} n^{(a-b)/2} e_{n} \otimes e_{n}.
\end{equation*}
Note that this operator is unbounded when $a > b > 1$, and in fact, the series defining $\m{T}$ fails to converge even in the weak operator topology, meaning it is not a well-defined operator on $\overline{\c{R}(\m{A})}$. On the other hand, this map cannot be arbitrarily unbounded. For instance, if we attempted to pushforward to a measure where $b \le 1$ (e.g., white noise with $b=0$), the resulting measure could not exist in $\c{H}$ as $\m{B}$ would no longer be trace-class. The theory of abstract Wiener spaces demonstrates that the Hilbert space must be expanded to a larger Banach space to accommodate such measures \cite{gross1967abstract, da2006introduction}.

These challenges show that in infinite dimensions, one \emph{cannot} simply assert that $\m{T}$ is optimal merely because it satisfies certain linear equations as in \cref{def:reachable}. Before tackling the Monge-Kantorovich problem, we must first address a more fundamental question: \emph{If a possibly unbounded operator $\m{T}$ pushes forward one Gaussian measure to another, what essential conditions must it satisfy?}

\subsection{Unbounded Operators}
This subsection provides a brief overview of unbounded operators, which are central to our analysis. Although the theory is often developed for $\b{C}$-Hilbert space to leverage their richer spectral theory \cite{conway2019course, hall2013quantum, reed1980methods, riesz2012functional, kato2013perturbation}, our work is set in $\b{R}$-Hilbert spaces, which introduces certain technical nuances that we address here.

\begin{definition}
Let $\c{H}_{1}, \c{H}$ be Hilbert spaces. 
\begin{enumerate}[leftmargin = *]
\item An operator $\m{T}: \c{D} (\m{T}) \subseteq \c{H}_{1} \rightarrow \c{H}$ is called \emph{unbounded} if it is
\begin{itemize}[leftmargin = *]
    \item [] \textbf{Densely defined :} Its domain $\c{D} (\m{T})$ is a dense subspace of $\c{H}_{1}$.
    \item [] \textbf{Linearity :} For any $f, g \in \c{D} (\m{T})$ and $\alpha, \beta \in \b{R}$, $\m{T}(\alpha f + \beta g) = \alpha \m{T}(f) +\beta \m{T} (g)$.
\end{itemize}
We denote the space of unbounded operators from $\c{H}_{1}$ to $\c{H}$ by $\c{D}(\c{H}_{1}, \c{H})$, or simply $\c{D}(\c{H})$ when $\c{H}_{1} = \c{H}$. An operator $\m{T} \in \c{D}(\c{H}_{1}, \c{H})$ is called \emph{bounded} if there exists a constant $c > 0$ such that $\|\m{T}h\| \le c \|h\|$ for all $h \in \c{D}(\m{T})$.

\item The \emph{graph} of $\m{T}$ is the set $\c{G}(\m{T}) := \{(h, \m{T}h) \in \c{H}_{1} \times \c{H} \mid h \in \c{D}(\m{T}) \}$.
The operator $\m{T} \in \c{D}(\c{H}_{1}, \c{H})$ is called \emph{closed} if its graph $\c{G}(\m{T})$ is closed in $\c{H}_{1} \times \c{H}$. We denote the space of closed operators by $\c{C}(\c{H}_{1}, \c{H})$, or simply $\c{C}(\c{H})$ when $\c{H}_{1} = \c{H}$.  

\item For $\m{T}, \tilde{\m{T}} \in \c{D}(\c{H}_{1}, \c{H})$, we say that $\tilde{\m{T}}$ is an \emph{extension} of $\m{T}$, denoted by $\m{T} \subset \tilde{\m{T}}$, if $\c{G}(\m{T}) \subset \c{G}(\tilde{\m{T}})$.

\item An operator $\m{T} \in \c{D}(\c{H}_{1}, \c{H})$ is called \emph{closable} if it admits a closed extension $\tilde{\m{T}} \in \c{C}(\c{H}_{1}, \c{H})$. Its extension $\bar{\m{T}}$ is called the \emph{closure} if $\c{G}(\bar{\m{T}}) = \overline{\c{G}(\m{T})}$. 
\end{enumerate}
\end{definition}

Note that \textit{unbounded} does not necessarily mean \textit{not bounded}. However, a bounded operator has only a trivial bounded extension, making this case uninteresting. Additionally, a closed operator is defined everywhere if and only if it is bounded by the closed graph theorem \cite{conway2019course}.

\begin{definition}[Adjoint]
Let $\m{T} \in \c{D}(\c{H}_{1}, \c{H})$. The domain of its adjoint is defined as
\begin{equation*}
    \c{D}(\m{T}^{*}) := \{h \in \c{H}: g \in \c{D}(\m{T}) \mapsto \innpr{h}{\m{T}g} \text{ is a bounded linear functional} \}.
\end{equation*}
For each $h \in \c{D}(\m{T}^{*})$, the Riesz representation theorem guarantees the existence of a unique vector $\m{T}^{*} h \in \c{H}_{1}$ such that $\innpr{\m{T}^{*} h}{g} = \innpr{h}{\m{T} g}$ for all $g \in \c{D}(\m{T})$. This defines the \emph{adjoint} of $\m{T}$ as the linear operator $\m{T}^{*}: \c{D}(\m{T}^{*}) \to \c{H}_{1}$.
\end{definition}

The adjoint has several key properties. It is always a closed operator, thus $\m{T}^{*}$ is defined everywhere if and only if $\m{T}$ is bounded. Furthermore, if $\m{A} \subset \m{B}$, then $\m{B}^{*} \subset \m{A}^{*}$. A fundamental property is the \emph{duality}:
\begin{equation}\label{eq:adj:dual}
    \innpr{\m{T}^{*} h}{g} = \innpr{h}{\m{T} g}, \quad h \in \c{D}(\m{T}^{*}), \ g \in \c{D}(\m{T}).
\end{equation}
Additionally, $\c{G} \in \c{H}_{1} \times \c{H}$ is the graph of an operator if and only if it is a linear subspace with the single valued property: $(0, g) \in \c{G} \, \Rightarrow \, g =0$, from which the following equivalence follows:
\begin{itemize}
    \item $\m{T}$ is closable.
    \item If $f_{n} \in \c{D}(\m{T}) \to 0$ and $\m{T} f_{n} \to g$, then $g = 0$.
    \item There exists a closed extension of $\m{T}$.
    \item $\m{T}^{*} \in \c{D}(\c{H}, \c{H}_{1})$ is densely defined.
\end{itemize}
In this case, we have $\bar{\m{T}} = \m{T}^{**}$, which is the minimal closed extension. Consequently, $\m{T} \in \c{C}(\c{H}_{1}, \c{H})$ if and only if $\m{T} = \m{T}^{**}$.

\subsection{Admissible Operators}\label{ssec:admiss}
Let $X \sim N_{\c{H}}(\m{0}, \m{A})$ be a Gaussian random element  with covariance operator $\m{A} \in \c{B}_{1}^{+}(\c{H})$. For the expression $\m{T} X$ to be well-defined, the domain of $\m{T}$ (or a suitable extension) must contain a set of full measure for the law of $X$. Since $\b{P}(X \in \overline{\c{R}(\m{A}})) = 1$, we naturally restrict our operators to have domains within this subspace. The following lemma establishes that any such subspace of full measure must be dense.

\begin{lemma}\label{lem:a.s.:dense}
Let $X \sim N_{\c{H}}(\m{0}, \m{A})$. Then for any linear subspace $\c{D} \subset \overline{\c{R}(\m{A}})$, we have $\b{P}(X \in \c{D})=1$ only if $\c{D}$ is dense in $\overline{\c{R}(\m{A}})$. 
\end{lemma}
\begin{proof}[Proof of \cref{lem:a.s.:dense}]
Assume that $\c{D}$ is not dense in $\c{H}$. Then, there exists a nonzero element $h \in \overline{\c{R}(\m{A}}) \cap \c{D}^{\perp}$. By the definition of Gaussianity \cite{hsing2015theoretical, da2006introduction}, we have $\innpr{h}{X} \sim N(0, \innpr{h}{\m{A}h})$
with non-zero variance due to the injectivity of $\m A$ on $\c{H}_1$. However, we also have
\begin{align*}
    1 = \b{P}(\innpr{h}{X} \neq 0) = \b{P}(\innpr{h}{X} \neq 0, X \in \c{D}) = 0,
\end{align*}
which is a contradiction. Hence, it must be that $\overline{\c{D}} = \c{H}_{1}$.
\end{proof}

In light of this lemma, we henceforth consider operators $\m{T}$ whose domains are dense subspaces of $\c{H}_{1} = \overline{\c{R}(\m{A})}$. 
A key challenge is that the measurability of 
$\m{T}X$ is not guaranteed, as an operator must typically be continuous to be Borel measurable \cite[Theorem. 9.10]{kechris2012classical}. The next theorem provides sufficient conditions for a closable operator
$\m{T}$ to admit an extension  
$\tilde{\m{T}}$ for which 
$\tilde{\m{T}} X$ is a well-defined random element.

\begin{theorem}[Pre-pushforward]\label{thm:prepush}
Let $\m{T} \in \c{D}(\overline{\c{R}(\m{A}}), \c{H})$ be an unbounded operator, and let $X \sim N_{\c{H}}(\m{0}, \m{A})$ where $\m{A} \in \c{B}_{1}^{+}(\c{H})$. Assume that $\m{T}$ is closable and $\m{T} \m{A}^{1/2} \in \c{B}_{2}(\c{H})$. If $\tilde{\m{T}}$ is any extension of the closure $\overline{\m{T}}$ of $\m{T}$, then $\b{P}(X \in \c{D}(\tilde{\m{T}}))=1$ and
\begin{equation*}
    \tilde{\m{T}} X \sim N_{\c{H}}(\m{0}, (\m{T} \m{A}^{1/2})(\m{T} \m{A}^{1/2})^{*}).
\end{equation*}
\end{theorem}

\begin{proof}[Proof of \cref{thm:prepush}]
Consider the spectral decomposition of the covariance $\m{A} \in \c{B}_{1}^{+}(\c{H})$:
\begin{equation*}
    \m{A} = \sum_{k=1}^{\infty} \lambda_{k} e_{k} \otimes e_{k},
\end{equation*}
where $\lambda_{1} \ge \lambda_{2} \ge \dots >0$ are positively-valued real sequence that is non-increasing and $\{e_{k} : k \in \b{N}\}$ is an ONB of $\overline{\c{R}(\m{A})}$. For each $k \in \b{N}$, define $\xi_{k} := \innpr{X}{e_{k}} / \sqrt{\lambda_{k}}$ so that $\xi_{k} \stackrel{i.i.d.}{\sim} N(0, 1)$. For each $n \in \b{N}$, define
\begin{align*}
    X_{n} := \sum_{k=1}^{n} \sqrt{\lambda_{k}} \xi_{k} e_{k} \in \c{R}(\m{A}^{1/2}) \subset \c{D}(\m{T}), \quad
    S_{n} := \m{T} X_{n} = \sum_{k=1}^{n} \xi_{k} (\m{T} \m{A}^{1/2}) e_{k}.
\end{align*}

\begin{enumerate}
\item [(Step 1)] We claim that $S_{n}$ is a $\b{P}$-a.s. Cauchy sequence. As in \cite[Lemma 3.15]{kallenberg1997foundations}, we get for any $\varepsilon > 0$ and $n \in \b{N}$ that
\begin{align*}
    \b{P} \left[ \sup_{k \ge n} \|S_{k} - S_{n} \| > \varepsilon \right] &\le \frac{1}{\varepsilon^{2}} \sum_{k \ge n} \b{E} \|\xi_{k} (\m{T} \m{A}^{1/2}) e_{k}\|^{2} = \frac{1}{\varepsilon^{2}} \sum_{k \ge n} \|(\m{T} \m{A}^{1/2}) e_{k}\|^{2} \\
    &\le \frac{1}{\varepsilon^{2}} \sum_{k=1}^{\infty} \|(\m{T} \m{A}^{1/2}) e_{k}\|^{2} = \frac{1}{\varepsilon^{2}} \vertii{\m{T} \m{A}^{1/2}}_{2}^{2} < \infty,
\end{align*}
and thus $\sup_{k \ge n} \|S_{k} - S_{n} \| \stackrel{\b{P}}{\to} 0$ as $n \to \infty$. Hence, there is a subsequence $\{n'\} \subset \{n\}$ such that $\sup_{k \ge n'} \|S_{k} - S_{n'} \| \stackrel{a.s.}{\to} 0$ as $n' \to \infty$. Because $\sup_{k \ge n} \|S_{k} - S_{n} \|$ is a non-increasing sequence, this entails that $\sup_{k \ge n} \|S_{k} - S_{n} \| \stackrel{a.s.}{\to} 0$, i.e. $S_{n}$ is a.s. Cauchy convergent, hence there is a random element $S$ in $\c{H}$ such that $S_{n} \stackrel{a.s.}{\to} S$.
\item [(Step 2)] Plugging in $\m{T} = \m{I}$ in (Step 1), or using the white noise expansion theorem \cite{da2006introduction}, we also obtain that $X_{n}$ is a.s. Cauchy convergent. Then, Bessel's equality concludes that $X_{n} \stackrel{a.s.}{\to} X$.
\item [(Step 3)] Consider a Borel measurable set $\Omega_{X}$ on which both $X_{n} \to X$ and $S_{n} \to S$ hold. Note that $\b{P}(\Omega_{X}) = 1$. Consider any $\tilde{\m{T}} \supset \m{T}^{**} = \overline{\m{T}}$. Then, on $\Omega_{X}$, we get
\begin{align*}
    \tilde{\m{T}} X_{n} = \overline{\m{T}} X_{n} = \m{T} X_{n} = S_{n} \to S, \quad X_{n} \in \c{D}(\m{T}) \subset \c{D}(\overline{\m{T}}) \to X.
\end{align*}
Since $\overline{\m{T}}$ is a closed operator, this yields that $X \in \c{D}(\overline{\m{T}})$ and $S = \overline{\m{T}} X = \tilde{\m{T}} X$ on $\Omega_{X}$, i.e. $S_{n} \stackrel{a.s.}{\to} \tilde{\m{T}} X$. Finally, because $S_{n}$'s are Gaussian and $\Cov[S_{n}] \to (\m{T} \m{A}^{1/2})(\m{T} \m{A}^{1/2})^{*}$, we conclude that $\tilde{\m{T}} X \sim N_{\c{H}}(\m{0}, (\m{T} \m{A}^{1/2})(\m{T} \m{A}^{1/2})^{*})$.
\end{enumerate}
\end{proof}

\cref{thm:prepush} shows that if a closable operator $\m{T}$ is sufficiently regular -- in the sense that $\m{T} \m{A}^{1/2} \in \c{B}_{2}(\c{H})$ -- then it can be extended to define a valid pushforward measure, e.g., $\overline{\m{T}}$. Furthermore, the theorem reveals a notable infinite-dimensional phenomenon: the resulting covariance operator is determined entirely by how the \emph{pre-pushforward} $\m{T}$
acts over $\c{R}(\m{A}^{1/2})$, which is a $N_{\c{H}}(0, \m{A})$-null set. 

\begin{remark}\label{rmk:close:break} 
The Hilbert-Schmidt condition does not imply the closability condition. To see this, consider an ONB $\{e_{n}\}$ of $\c{H}$, and define
\begin{equation*}
    \m{A}^{1/2} = \sum_{n=1}^{\infty} 2^{-2n} e_{n} \otimes e_{n}, \quad \m{T} = \sum_{n = 1}^{\infty} 2^{n} e_{1} \otimes e_{n}.
\end{equation*}
Then, $\m{T} \m{A}^{1/2} \in \c{B}_{2}(\c{H})$ but $\m{T}$ fails to be closable: $2^{-n} e_{n} \in \c{D}(\m{T}) \to 0$ but $\m{T} (2^{-n} e_{n}) \rightarrow e_{1} \neq 0$.
\end{remark}

We also note that if $\m{T}$ is closable, then its restriction $\m{T} \mid_{\c{R}(\m{A}^{1/2})}: \c{R}(\m{A}^{1/2}) \subset \c{H}_{1} \rightarrow \c{H}$ as well. We may therefore assume the domain of pre-pushforward to be $\c{R}(\m{A}^{1/2})$. The next theorem provides a converse to \cref{thm:prepush}: for a closable pushforward operator, the Hilbert-Schmidt condition is not only sufficient but also necessary:

\begin{theorem}\label{thm:trans:closed}
Let $\m{T} : \c{R}(\m{A}^{1/2}) \subset \c{H}_{1} \rightarrow \c{H}$ be closable, and let $X \sim N_{\c{H}}(\m{0}, \m{A})$. If $\m{T} X$ is a random element in $\c{H}$, then it must be Gaussian and $\m{T} \m{A}^{1/2} \in \c{B}_{2}(\c{H})$. 
Specifically,
\begin{equation*}
    \m{T} X \sim N_{\c{H}}(\m{0}, (\m{T} \m{A}^{1/2})(\m{T} \m{A}^{1/2})^{*}).
\end{equation*}
\end{theorem}

\begin{proof}[Proof of \cref{thm:trans:closed}]
Choose any $f \in \c{D}(\m{A}^{1/2} \m{T}^{*}) = \c{D}(\m{T}^{*})$. Then for any $g \in \c{D}(\m{T} \m{A}^{1/2})$, or equivalently, $\m{A}^{1/2} g \in \c{D}(\m{T})$, we have from \eqref{eq:adj:dual} that
\begin{align*}
    \innpr{f}{\m{T} \m{A}^{1/2} g} = \innpr{\m{T}^{*} f}{\m{A}^{1/2} g} = \innpr{\m{A}^{1/2} \m{T}^{*} f}{g},
\end{align*}
which shows that $\m{A}^{1/2} \m{T}^{*} \subset (\m{T} \m{A}^{1/2})^{*}$. 
Since $\b{P}(X \in \c{D}(\m{T}))=1$, note that, for any $f \in \c{D}(\m{T}^{*})$,
\begin{align*}
    \innpr{f}{\m{T} X} \stackrel{\b{P}-a.s.}{\equiv} \innpr{\m{T}^{*} f}{X} \sim N(0, \innpr{\m{T}^{*} f}{\m{A} \m{T}^{*}f}),
\end{align*}
with $\innpr{\m{T}^{*} f}{\m{A} \m{T}^{*}f} = \| \m{A}^{1/2} \m{T}^{*}f \|^{2} = \|(\m{T} \m{A}^{1/2})^{*} f \|^{2}$.

Now, let $f \in \c{H}$. Since $\c{D}(\m{T}^{*})$ is dense in $\c{H}$, there exists a sequence $f_{n} \in \c{D}(\m{T}^{*})$ such that $f_{n} \to f$ in $\c{H}$. For any fixed $\varepsilon > 0$, by the Cauchy-Schwarz inequality,
\begin{align*}
    \b{P} \left( |\innpr{f}{\m{T} X} - \innpr{f_{n}}{\m{T} X} |>\varepsilon \right) \le \b{P} \left( \| \m{T} X \| > \frac{\varepsilon}{\|f_{n}-f\|} \right).
\end{align*}
Because any probability measure on $\c{H}$ is tight \cite{billingsley2013convergence}, the right-hand side vanishes as $n \to \infty$, implying convergence in probability:
\begin{equation}\label{eq:conv:prob}
    \innpr{f_{n}}{\m{T} X} \sim N(0, \|(\m{T} \m{A}^{1/2})^{*} f_{n} \|^{2}) \stackrel{\b{P}}{\rightarrow} \innpr{f}{\m{T} X}.
\end{equation}
This results in the fact that $\|(\m{T} \m{A}^{1/2})^{*} f_{n} \|$ converges to some constant $\sigma_{f} \ge 0$, and $\innpr{f}{\m{T} X} \sim N(0, \sigma_{f}^{2})$.

Applying the Banach-Alaoglu theorem \cite{conway2019course} to the bounded sequence $\|(\m{T} \m{A}^{1/2})^{*} f_{n} \|$, there is a subsequence $(n') \subset (n)$ such that $(\m{T} \m{A}^{1/2})^{*} f_{n'} \stackrel{wk}{\rightarrow} h$ with some $h \in \c{H}$. For any $g \in \c{D}(\m{T} \m{A}^{1/2})$,
\begin{align*}
    \innpr{g}{h} = \lim_{n' \rightarrow \infty} \innpr{g}{(\m{T} \m{A}^{1/2})^{*} f_{n'}} = \lim_{n' \rightarrow \infty} \innpr{(\m{T} \m{A}^{1/2}) g}{f_{n'}} = \innpr{(\m{T} \m{A}^{1/2}) g}{f},
\end{align*}
which demonstrates that $f \in \c{D}[(\m{T} \m{A}^{1/2})^{*}]$ and $h = (\m{T} \m{A}^{1/2})^{*} f$. However, since $f \in \c{H}$ was chosen arbitrary, this implies that $\c{D}[(\m{T} \m{A}^{1/2})^{*}] = \c{H}$, i.e. $(\m{T} \m{A}^{1/2})^{*} \in \c{B}_{\infty}(\c{H})$ due to the Closed Graph Theorem, thus $(\m{T} \m{A}^{1/2}) \in \c{B}_{\infty}(\c{H})$. Therefore, $\m{T} X \sim N_{\c{H}}(\m{0}, (\m{T} \m{A}^{1/2})(\m{T} \m{A}^{1/2})^{*})$, and since its covariance should be of trace-class \cite{da2006introduction}, we obtain $(\m{T} \m{A}^{1/2}) \in \c{B}_{2}(\c{H})$.
\end{proof}

As a direct application, \cref{thm:prepush,thm:trans:closed} yield the Driscoll 0-1 law  \cite{driscoll1973reproducing, lukic2001stochastic}:

\begin{corollary}
Let $X \sim N_{\c{H}}(\m{0}, \m{A})$ where $\m{A} \in \c{B}_{1}^{+}(\c{H})$. For $\alpha \in (0, 1/2)$, $\b{P}(X \in \c{R}(\m{A}^{\alpha}))=1$ if and only if $\m{A}^{1/2-\alpha} \in \c{B}_{2}(\c{H})$.
\end{corollary}
\begin{proof}
Let $\c{H}_{1} = \overline{\c{R}(\m{A})}$ and define $\m{T} := (\m{A}^{\alpha})^{\dagger} \vert_{\c{H}_{1}} : \c{R}(\m{A}^{\alpha}) \subset \c{H}_{1} \to \c{H}$. Then $\m{T} \in \c{C}(\c{H}_{1}, \c{H})$ is closed, so the result follows immediately from \cref{thm:prepush,thm:trans:closed} since $\m{T} \m{A}^{1/2} = \m{A}^{1/2-\alpha}$. 
\end{proof}

Finally, we consider the case where the closability condition on the pre-pushforward $\m{T}$ is dropped. The following theorem shows that an extension is still possible, although it is less explicit than the closure:

\begin{theorem}\label{thm:wo:closable}
Let $\m{T}: \c{R}(\m{A}^{1/2}) \subset \c{H}_{1} \rightarrow \c{H}$ be an unbounded operator, and let $X \sim N_{\c{H}}(\m{0}, \m{A})$ where $\m{A} \in \c{B}_{1}^{+}(\c{H})$. If $\m{T} \m{A}^{1/2} \in \c{B}_{2}(\c{H})$, then there exists an extension $\tilde{\m{T}} : \c{V} \to \c{H}$ of $\m{T}$ such that $\c{V}$ has full $N_{\c{H}}(0, \m{A})$-measure. In this case,
\begin{equation*}
    (\tilde{\m{T}} X, X) \sim N_{\c{H} \times \c{H}}(\m{0}, (\m{T} \m{A}^{1/2}, \m{A}^{1/2})(\m{T} \m{A}^{1/2}, \m{A}^{1/2})^{*}).
\end{equation*}
\end{theorem}

\begin{proof}[Proof of \cref{thm:wo:closable}]
We follow the proof similar to that of \cite[Theorem 4.47]{hairer2009introduction}. Noting that the RKHS associated with $N_{\c{H}}(0, \m{A})$ is given by \eqref{eq:Mahalanobis:RKHS},
it suffices to show that $\m{T}: \b{H}(\m{A}) \rightarrow \c{H}$ is a Hilbert-Schmidt operator, . Consider the spectral decomposition of $\m{A} \in \c{B}_{1}^{+}(\c{H})$ as in the proof of \cref{thm:prepush}. Then $\{\sqrt{\lambda_{k}} e_{k} : k \in \b{N}\}$ is an ONB of $\b{H}(\m{A})$, and
\begin{align*}
    \sum_{k=1}^{\infty} \innpr{\m{T} \sqrt{\lambda_{k}} e_{k}}{\m{T} \sqrt{\lambda_{k}} e_{k}}_{\c{H}} = \sum_{k=1}^{\infty} \innpr{\m{T} \m{A}^{1/2} e_{k}}{\m{T} \m{A}^{1/2} e_{k}}_{\c{H}} = \vertiii{\m{T} \m{A}^{1/2}}_{2}^{2} < \infty,
\end{align*}
which shows that $\m{T} \in \c{B}_{2}(\b{H}(\m{A}), \c{H}): \c{R}(\m{A}^{1/2}) \subset \c{H}_{1} \rightarrow \c{H} \times \c{H}$. For each $n \in \b{N}$, define the linear operator of finite rank:
\begin{equation*}
    \m{S}_{n} : \c{H} \to \c{H} \times \c{H}, \, h \mapsto \sum_{k=1}^{n} \innpr{h}{e_{k}}_{\c{H}} (\m{T} e_{k}, e_{k}) =  \sum_{k=1}^{n} \innpr{h}{\frac{e_{k}}{\sqrt{\lambda_{k}}}}_{\c{H}} (\m{T} \m{A}^{1/2} e_{k}, \m{A}^{1/2} e_{k}).
\end{equation*}
Define the linear subspace
\begin{equation*}
    \c{V} := \{ h \in \c{H} : \lim_{n \to \infty} \m{S}_{n}(h) \text{ exists in } \c{H} \times \c{H} \},
\end{equation*}
and let $\m{S}: \c{V} \to \c{H} \times \c{H}$ by $\m{S}(h) = \lim_{n \to \infty} \m{S}_{n}(h)$ for $h \in \c{V}$. To see that $\m{S}$ is an extension of $(\m{T}, \m{I})$, observe that for any $h = \m{A}^{1/2} g \in $
\begin{align*}
    \m{S}_{n} h &= \sum_{k=1}^{n} \innpr{\m{A}^{1/2} g}{\frac{e_{k}}{\sqrt{\lambda_{k}}}}_{\c{H}} (\m{T} \m{A}^{1/2} e_{k}, \m{A}^{1/2} e_{k}) \\
    &= \sum_{k=1}^{n} \innpr{g}{e_{k}}_{\c{H}} (\m{T} \m{A}^{1/2} e_{k}, \m{A}^{1/2} e_{k}) = (\m{T}\m{A}^{1/2}, \m{A}^{1/2})  \left(\sum_{k=1}^{n} \innpr{g}{e_{k}}_{\c{H}} e_{k} \right) \\
    &\to (\m{T} \m{A}^{1/2} g, \m{A}^{1/2} g) = (\m{T}, \m{I}) h
\end{align*}
as $n \to \infty$ by the Plancherel equality.
Note that $\m{S}_{n}(X)$ is an $(\c{H} \times \c{H})$-valued martingale with
\begin{align*}
    \sup_{n \in \b{N}} \b{E} \|\m{S}_{n}(X)\|_{\c{H} \times \c{H}}^{2} &= \sum_{k=1}^{\infty} \|(\m{T} \m{A}^{1/2} e_{k}, \m{A}^{1/2} e_{k})\|_{\c{H} \times \c{H}}^{2} \\
    &= \vertiii{(\m{T} \m{A}^{1/2}, \m{A}^{1/2})}_{2}^{2} = \vertiii{\m{T} \m{A}^{1/2}}_{2}^{2} + \vertiii{\m{A}}_{1} < \infty,
\end{align*}
hence $\c{V}$ has a full $N_{\c{H}}(0, \m{A})$-measure by the Doob martingale convergence theorem \cite{kallenberg1997foundations}. Using characteristic function as in \cite[Proposition 4.40]{hairer2009introduction}, $\m{S} X$ is a Gaussian random element, with the covariance given by $(\m{T} \m{A}^{1/2}, \m{A}^{1/2})(\m{T} \m{A}^{1/2}, \m{A}^{1/2})^{*} \in \c{B}_{1}^{+}(\c{H} \times \c{H})$, and letting $\tilde{\m{T}}$ to be the projection of $\m{S}$ onto $\c{V}$ completes the proof.
\end{proof}

\begin{remark}
The proof of \cref{thm:wo:closable} shows that $\m{T} \m{A}^{1/2} \in \c{B}_{2}(\c{H})$ if and only if $\m{T} \in \c{B}_{2}(\b{H}(\m{A}), \c{H})$, where $\b{H}(\m{A})$ is the RKHS associated with $N_{\c{H}}(0, \m{A})$. This perspective reinforces a key insight from the previous theorems: the resulting pushforward measure is determined solely by the action of the pre-pushforward operator $\m{T}$ on the RKHS.
\end{remark}

\section{Hilbertian Reachability}\label{sec:hilb:reach}

\begin{definition}[Reachability in Infinite Dimension]\label{def:inf:reach}
Let $\c{H}$ be a separable Hilbert space, and $\m{A}, \m{B} \in \c{B}_{1}^{+}(\c{H})$ be covariance operators, and let $\c{H}_{1} = \overline{\c{R}(\m{A})}$. We say that $\m{B}$ is \emph{unboundedly reachable} from $\m{A}$, denoted $\m{A} \leadsto \m{B}$, if there exists an unbounded operator $\m{T}: \c{R}(\m{A}^{1/2}) \subset \c{H}_{1} \rightarrow \c{H}$ such that
\begin{align}
    &(\m{T} \m{A}^{1/2}) (\m{T} \m{A}^{1/2})^{*} = \m{B}, \label{eq:cond:ricatti} \\ 
    &\tr (\m{A}^{1/2} \m{T} \m{A}^{1/2}) = \tr (\m{T} \m{A}) = \tr [(\m{A}^{1/2} \m{B} \m{A}^{1/2})^{1/2}]. \label{eq:cond2:opt}
\end{align}
Such an operator $\m{T}$ called an \emph{optimal pre-pushforward}.
If $\m{T}$ can be chosen to be bounded, then we say that $\m{B}$ is \emph{boundedly reachable} from $\m{A}$, denoted $\m{A} \rightarrow \m{B}$.
\end{definition}

\eqref{eq:cond:ricatti} automatically guarantees $(\m{T} \m{A}^{1/2}) \in \c{B}_{2}(\c{H})$. As established in \cref{sec:admissible}, $\m{T}$ itself is a pre-pushforward that admits a valid pushforward extension, and the resulting Gaussian measure depends only on $\m{T}$. If $\m{T}$ is not bounded, then $\m{T} \m{A} \m{T}^{*} \subsetneq (\m{T} \m{A}^{1/2}) (\m{T} \m{A}^{1/2})^{*} = \m{B}$, since (LHS) is not defined everywhere. Indeed, (RHS) is a trivial extension of (LHS).

\begin{remark}
As in the Euclidean case, we may replace $\m{T} \m{A}^{1/2}$ by $\m{T} \m{G}$ for any Green's operator $\m{G} \in \c{G}(\m{A})$ in \eqref{eq:cond:ricatti} and \eqref{eq:cond2:opt}:
\begin{enumerate}
    \item [] \textbf{Transportability :} $(\m{T} \m{G}) (\m{T} \m{G})^{*} = \m{B}$,
    \item [] \textbf{Optimality :} $\tr (\m{G}^{*} \m{T} \m{G}) = \tr (\m{T} \m{A}) = \tr [(\m{G}^{*} \m{B} \m{G})^{1/2}]$, 
\end{enumerate}
due to the polar decomposition $\m{G} = \m{A}^{1/2} \m{U}$ with $\m{U} \m{U}^{*} = \m{\Pi}_{1}$, leading to $(\m{T} \m{G}) = (\m{T} \m{A}^{1/2}) \m{U}$ and $(\m{G}^{*} \m{B} \m{G})^{1/2} = \m{U}^{*} (\m{A}^{1/2} \m{B} \m{A}^{1/2})^{1/2} \m{U}$. 
As shown in \cref{lem:Neumann:eq}, the optimality condition is equivalent to requiring either $\m{G}^{*} \m{T} \m{G} = (\m{G}^{*} \m{B} \m{G})^{1/2}$ or $\m{G}^{*} \m{T} \m{G} \in \c{B}_{\infty}^{+}(\c{H})$.
\end{remark}

\begin{corollary}
Let $\m{A}, \m{B} \in \c{B}_{1}^{+}(\c{H})$ and $\m{G} \in \c{G}(\m{A})$. Then $\m{A} \leadsto \m{B}$ if and only if there is some $\m{M} \in \c{B}_{\infty}(\c{H})$ satisfying any of the following equivalent sets of conditions:
\begin{itemize}
    \item $\c{N}(\m{G}) \subset \c{N}(\m{M})$, $\m{M} \m{M}^{*} = \m{B}$, and $\m{G}^{*} \m{M} = (\m{G}^{*} \m{B} \m{G})^{1/2}$.
    \item $\c{N}(\m{G}) \subset \c{N}(\m{M})$, $\m{M} \m{M}^{*} = \m{B}$, and $\m{G}^{*} \m{M} \succeq \m{0}$.
\end{itemize}
These conditions are independent of the choice of Green's operator $\m{G} \in \s{G}(\m{A})$. 
\end{corollary}
\begin{proof}
Trivial by considering $\m{T}: \c{R}(\m{A}^{1/2}) \subset \c{H}_{1} \rightarrow \c{H}, \m{G} h \mapsto \m{M} h$.
\end{proof}

As in the finite dimensional case, we investigate the following question: \emph{How can reachability be characterized, and can an optimal pre-pushforward be extended to be symmetric or even s.p.d.?}

\begin{definition}
Let $\m{T} \in \c{D}(\c{H})$ be an unbounded operator. Then,
\begin{enumerate}[leftmargin = *]
    \item $\m{T}$ is called \emph{symmetric} if $\innpr{\m{T} f}{g} = \innpr{f}{\m{T} g}$ for any $f, g \in \c{D}(\m{T})$, i.e. $\m{T} \subset \m{T}^{*}$.
    \item $\m{T}$ is called \emph{self-adjoint} if $\m{T} = \m{T}^{*}$.
    \item $\m{T}$ is called \emph{semi-positive definite} (s.p.d.) if it is symmetric and $\innpr{\m{T} f}{f} \ge 0$ for any $f \in \c{D}(\m{T})$.
\end{enumerate}
\end{definition}

By definition, any symmetric operator is closable, and a self-adjoint operator admits no symmetric extension other than itself.
Any self-adjoint operator is symmetric and closed. However, in contrast to the bounded case, a closed symmetric operator is not necessarily self-adjoint \cite[Example X.1.12]{conway2019course} or \cite[Section 9.6]{hall2013quantum}. Additionally, given $\m{A} \in \c{B}_{0}^{+}(\c{H})$, we remark that $\m{A}^{\dagger} \in \c{C}(\c{H})$ is self-adjoint. If $\m{A}$ is injective, then $\m{A}^{\dagger}$ is boundedly invertible on $\c{D}(\m{A}^{\dagger}) = \c{R}(\m{A})$, with its unique right inverse $\m{A}$.


\begin{proposition}\label{cor:worowicz:self:adj}
Let $\m{A}, \m{C} \in \c{B}_{0}^{+}(\c{H})$ with $\c{R}(\m{C}) \subset \c{R}(\m{A}^{1/2})$. Define the operator $\m{S} := \m{A}^{\dagger/2} \m{C} \m{A}^{\dagger/2}: \c{D}(\m{A}^{\dagger/2}) \rightarrow \c{H}$. Then the following holds:
\begin{enumerate}[leftmargin = *]
\item $\m{S} \in \c{D}(\c{H})$ is an s.p.d. operator.
\item If we let $\m{R} := \m{A}^{\dagger/2} \m{C} \in \c{B}_{\infty}(\c{H})$, then $\m{R}^{*} = \overline{\m{C} \m{A}^{\dagger/2}}$ and $\m{S}^{*} = \m{A}^{\dagger/2} \m{R}^{*} \in \c{D}(\c{H})$.
\item $\m{A}^{\dagger/2} \m{C}^{1/2} \in \c{C}(\c{H})$ is closed and $\tilde{\m{S}} := \m{A}^{\dagger/2} \m{C}^{1/2} \overline{\m{C}^{1/2} \m{A}^{\dagger/2}}$ is a self-adjoint extension of $\m{S}$.
\end{enumerate}
\end{proposition}

\begin{proof}[Proof of \cref{cor:worowicz:self:adj}]
Note that $\m{R} := \m{A}^{\dagger/2} \m{C} \in \c{B}_{\infty}(\c{H})$ is a bounded operator by the Douglas lemma.

\begin{enumerate}[leftmargin = *]
\item We remark $\c{D}(\m{S}) = \c{D}(\m{A}^{\dagger/2})$ is dense in $\c{H}$. 
Also, for any $f, g \in \c{D}(\m{A}^{\dagger/2})$, let $k := \m{A}^{\dagger/2} f, h := \m{A}^{\dagger/2} g \in \overline{\c{R}(\m{A}^{1/2})} \subset \c{D}(\m{A}^{\dagger/2})$. Then, it holds that $f - \m{A}^{1/2} k \in \c{N}(\m{A}^{1/2})$, thus
\begin{equation*}
    \m{S} f = \m{S} \m{A}^{1/2} k = \m{A}^{\dagger/2} \m{C} \m{A}^{\dagger/2} \m{A}^{1/2} k = \m{A}^{\dagger/2} \m{C} k,
\end{equation*}
and similarly, $\m{S} g = \m{A}^{\dagger/2} \m{C} h$. Therefore, we obtain
\begin{align*}
    \innpr{\m{S}f}{g} = \innpr{\m{A}^{\dagger/2} \m{C} k}{g} = \innpr{\m{C} k}{\m{A}^{\dagger/2} g} &= \innpr{\m{C} k}{h} \\
    &= \innpr{k}{\m{C} h} = \innpr{\m{A}^{\dagger/2} f}{\m{C} h} = \innpr{f}{\m{A}^{\dagger/2} \m{C} h} = \innpr{f}{\m{S} g},
\end{align*}
which establishes that $\m{S}$ is is a s.p.d. operator.

\item By \cref{lem:closable:half}, $\m{R} = \m{A}^{\dagger/2} \m{C} = (\m{C} \m{A}^{\dagger/2})^{*} \in \c{B}_{\infty}(\c{H})$ is defined everywhere, hence $\m{C} \m{A}^{\dagger/2}$ is closable. Therefore, $\m{R}^{*} = (\m{C} \m{A}^{\dagger/2})^{**} = \overline{\m{C} \m{A}^{\dagger/2}}$. Applying \cref{lem:closable:half} again, we also get
\begin{equation*}
    \m{S}^{*} = (\m{R} \m{A}^{\dagger/2})^{*} = \m{A}^{\dagger/2} \m{R}^{*}.
\end{equation*}
Also, since $\m{S}$ is symmetric, $\m{S}^{*}$ is densely defined in $\c{H}$, i.e., $\m{S}^{*} \in \c{D}(\c{H})$.

\item Note that $\c{R}(\m{C}^{1/2}) \oplus \c{N}(\m{C}^{1/2}) \subset \c{D}(\m{A}^{\dagger/2} \m{C}^{1/2})$ are dense in $\c{H}$, thus it is immediate from \cref{lem:closable:half} that $\m{A}^{\dagger/2} \m{C}^{1/2} = (\m{C}^{1/2} \m{A}^{\dagger/2})^{*} \in \c{C}(\c{H})$. This yields that $\m{Q}:= \m{C}^{1/2} \m{A}^{\dagger/2}$ is closable. Define $\tilde{\m{S}} := \m{Q}^{*} \overline{\m{Q}}$ so that $\m{S} = \m{Q}^{*} \m{Q} \subset \tilde{\m{S}}$. The closure of symmetric operator is symmetric \cite{conway2019course}, and one can easily verify that $\tilde{\m{S}}$ is also s.p.d.. 

We claim that $\m{I} + \tilde{\m{S}}$ is boundedly invertible.
Observe that $\m{I} + \tilde{\m{S}}$ is injective since $\m{Q}^{*} = \m{Q}^{***} = \overline{\m{Q}}^{*}$ and
\begin{equation*}
    \innpr{(\m{I} + \tilde{\m{S}})h}{h} = \|h\|^{2} + \| \overline{\m{Q}} h\|^{2} \ge \|h\|^{2}, \quad h \in \c{D}(\tilde{\m{S}}).
\end{equation*}

Note that $\c{H} \times \c{H} = \c{G}(\overline{\m{Q}}) \oplus \c{J} [\c{G}(\overline{\m{Q}}^{*})]$ \cite{conway2019course}. Then, for any $f \in \c{H}$, there exists $x \in \c{D}(\overline{\m{Q}})$ and $y \in \c{D}(\overline{\m{Q}}^{*})$ such that
\begin{equation*}
    (f, 0) = (x, \overline{\m{Q}}x) \oplus (-\overline{\m{Q}}^{*}y, y),
\end{equation*}
which leads to
\begin{equation*}
    f = x - \overline{\m{Q}}^{*}y = x + \overline{\m{Q}}^{*} \overline{\m{Q}} x = (\m{I} + \tilde{\m{S}})x.
\end{equation*}
This demonstrates that $\m{I} + \tilde{\m{S}}$ is bijective, and 
\begin{equation*}
    \|(\m{I} + \tilde{\m{S}})h\|^{2} = \|h\|^{2} + 2 \| \overline{\m{Q}} h\|^{2} + \|\tilde{\m{S}} h\|^{2} \ge \|h\|^{2}, \quad h \in \c{D}(\tilde{\m{S}}),
\end{equation*}
shows that $\vertii{(\m{I} + \tilde{\m{S}})^{-1}}_{\infty} \le 1$. 
Since $\m{I} + \tilde{\m{S}}$ is boundedly invertible, it is closed \cite{conway2019course}. If $(h_{n}, \tilde{\m{S}} h_{n}) \in \c{G}(\tilde{\m{S}}) \to (f, g) \in \c{H} \times \c{H}$, then
\begin{equation*}
    f = \lim_{n \to \infty} f_{n}, \quad f + g = \lim_{n \to \infty} (\m{I} + \tilde{\m{S}}) h_{n},
\end{equation*}
hence $f \in \c{D}(\m{I} + \tilde{\m{S}}) = \c{D}(\tilde{\m{S}})$ and $f + g = (\m{I} + \tilde{\m{S}}) f$, i.e. $g = \tilde{\m{S}} f$. This shows that $\tilde{\m{S}}$ is closed.

Now, suppose $\tilde{\m{S}}$ is not self-adjoint, i.e. $\tilde{\m{S}} \subsetneq \tilde{\m{S}}^{*}$. Note that $(\m{I}+\tilde{\m{S}})^{*} = \m{I}+\tilde{\m{S}}^{*}$ since $\m{I} \in \c{B}_{\infty}(\c{H})$, so we get $(\m{I}+\tilde{\m{S}}) \subsetneq (\m{I}+\tilde{\m{S}})^{*}$. Meanwhile, $\m{I} + \tilde{\m{S}}$ is surjective from the claim, meaning that $\c{N}[(\m{I}+\tilde{\m{S}})^{*}] \neq \{0\}$. However,
\begin{equation*}
    \c{N}[(\m{I}+\tilde{\m{S}})^{*}] = \c{R}[(\m{I}+\tilde{\m{S}})]^{\perp} = \{0\},
\end{equation*}
which is a contradiction. Therefore, $\tilde{\m{S}}$ is self-adjoint.
\end{enumerate}
\end{proof}

We remark that $\tilde{\m{S}}$ is not necessarily the unique self-adjoint extension of $\m{S}$. For further details, see discussions on essential self-adjointness and maximal symmetric extensions \cite{hall2013quantum, conway2019course}.

\subsection{Injective Case}\label{ssec:reach:inf}

\begin{proposition}\label{thm:sol:injec}
Let $\m{A}, \m{B} \in \c{B}_{1}^{+}(\c{H})$ be covariance operators with $\m{A} \succ \m{0}$, and let $\m{G} \in \c{G}(\m{A})$ be any Green's operator for $\m{A}$. Define
\begin{equation*}
    \m{S}_{\m{A} \leadsto \m{B}} := (\m{G}^{*})^{\dagger} (\m{G}^{*} \m{B} \m{G})^{1/2} \m{G}^{\dagger}: \c{R}(\m{A}^{1/2}) \subset \c{H} \rightarrow \c{H}.
\end{equation*}
Then $\m{A} \leadsto \m{B}$, and the following holds:
\begin{enumerate}[leftmargin = *]
    \item $\m{S}_{\m{A} \leadsto \m{B}}$ is independent of the choice of $\m{G} \in \c{G}(\m{A})$.
    \item $\m{S}_{\m{A} \leadsto \m{B}}$ is s.p.d. and is the unique optimal pre-pushforward.
    \item $(\m{G}^{*})^{\dagger} (\m{G}^{*} \m{B} \m{G})^{1/4} \overline{(\m{G}^{*} \m{B} \m{G})^{1/4} \m{G}^{\dagger}}$ is a self-adjoint extension of $\m{S}_{\m{A} \leadsto \m{B}}$ and thus a valid pushforward.
\end{enumerate}
\end{proposition}

\begin{proof}[Proof of \cref{thm:sol:injec}]
In the proof, we denote $\m{S}$ instead of $\m{S}_{\m{A} \leadsto \m{B}}$ to ease notation. By the Douglas lemma, $\m{S} \m{A}^{1/2} = \m{A}^{\dagger/2} (\m{A}^{1/2} \m{B} \m{A}^{1/2})^{1/2}$ is bounded, hence $\c{R}(\m{A}^{1/2}) = \c{D}(\m{S})$.

\begin{enumerate}[leftmargin = *]
\item Consider the polar decomposition $\m{G} = \m{A}^{1/2} \m{U}$, where $\m{U}$ is a unique partial isometry with initial space $\overline{\c{R}(\m{G}^{*})}$ and final space $\overline{\c{R}(\m{G})} = \c{H}$ since $\c{N}(\m{A})=\{0\}$. Then, $\m{U} \m{U}^{*} = \m{I}$, and we get 
\begin{equation*}
    (\m{G}^{*})^{\dagger} (\m{G}^{*} \m{B} \m{G})^{1/2} \m{G}^{\dagger} = \m{A}^{\dagger/2} \m{U} \m{U}^{*} (\m{A}^{1/2} \m{B} \m{A}^{1/2})^{1/2} \m{U} \m{U}^{*} \m{A}^{\dagger/2} = \m{A}^{\dagger/2} (\m{A}^{1/2} \m{B} \m{A}^{1/2})^{1/2} \m{A}^{\dagger/2},
\end{equation*}
which does not depend on the choice of $\m{G}$.

\item Let $\m{G} = \m{A}^{1/2}$. From \cref{cor:worowicz:self:adj}, $\m{S}$ is s.p.d.. Note that when $\m{A} \succ \m{0}$, we have $\m{A}^{\dagger/2} \m{A}^{1/2} = \m{I}$. Also, by the polar decomposition, there exists a partial isometry $\m{R} \in \c{B}_{\infty}(\c{H})$ such that $(\m{A}^{1/2} \m{B} \m{A}^{1/2})^{1/2} = \m{A}^{1/2} \m{B}^{1/2} \m{R}$. This implies that $\m{S} \m{A}^{1/2} = \m{A}^{\dagger/2} (\m{A}^{1/2} \m{B} \m{A}^{1/2})^{1/2} = [\m{A}^{\dagger/2} \m{A}^{1/2}] \m{B}^{1/2} \m{R} = \m{B}^{1/2} \m{R} \in \c{B}_{2}(\c{H})$. Then $\m{S}$ satisfies \eqref{eq:cond:ricatti} and \eqref{eq:cond2:opt}, so it is optimal:
\begin{align*}
    &\m{A}^{1/2} \m{B}^{1/2} \m{R} \m{R}^{*} \m{B}^{1/2} \m{A}^{1/2} = \m{A}^{1/2} \m{B} \m{A}^{1/2} \quad \Longleftrightarrow \quad (\m{S} \m{A}^{1/2}) (\m{S} \m{A}^{1/2})^{*} = \m{B}^{1/2} \m{R} \m{R}^{*} \m{B}^{1/2} = \m{B}, \\
    &\m{A}^{1/2} \m{S} \m{A}^{1/2} = [\m{A}^{1/2} \m{A}^{\dagger/2}] (\m{A}^{1/2} \m{B} \m{A}^{1/2})^{1/2} = (\m{A}^{1/2} \m{B} \m{A}^{1/2})^{1/2}.
\end{align*}

To show the uniqueness, let $\m{T}: \c{R}(\m{A}^{1/2}) \subset \c{H} \rightarrow \c{H}$ be an optimal pre-pushforward. Then, \eqref{eq:cond:ricatti} implies
\begin{align*}
    \m{A}^{1/2} \m{B} \m{A}^{1/2} = \m{A}^{1/2} (\m{T} \m{A}^{1/2}) (\m{T} \m{A}^{1/2})^{*} (\m{A}^{1/2})^{*} = (\m{A}^{1/2} \m{T} \m{A}^{1/2}) (\m{A}^{1/2} \m{T} \m{A}^{1/2})^{*},
\end{align*}
hence \eqref{eq:cond2:opt} and \cref{lem:Neumann:eq} yield $\m{A}^{1/2} \m{T} \m{A}^{1/2} = (\m{A}^{1/2} \m{B} \m{A}^{1/2})^{1/2} = \m{A}^{1/2} \m{S} \m{A}^{1/2}$. Since $\m{A}^{1/2} \succ \m{0}$, we get $\m{T} \m{A}^{1/2} = \m{S} \m{A}^{1/2}$, i.e., $\m{S} = \m{T}$.

\item Straightforward from \cref{cor:worowicz:self:adj,thm:prepush}.
\end{enumerate}
\end{proof}

\begin{remark}\label{rmk:spd:then:opt}
Even with a suitable choice of $\m{G} \in \c{G}(\m{A})$, computing the square root $(\m{G}^{*} \m{B} \m{G})^{1/2}$ could still be demanding. However, if one can show that a candidate operator $\m{T} \in \c{D}(\c{H})$ satisfies \eqref{eq:cond:ricatti} and $(\m{G}^{*} \m{T} \m{G}) \succeq \m{0}$, i.e.
\begin{equation*}
    \innpr{\m{G} h}{(\m{T} \m{G}) h} \ge 0, \quad h \in \c{H},
\end{equation*}
then this is sufficient to establish the optimality of $\m{T}$ without explicitly computing the square root. See \cref{sec:Gproc} for examples.
\end{remark}

Assuming $\c{N}(\m{A}) \subset \c{N}(\m{B})$ is effectively equivalent to assuming $\m{A}$ is injective ($\m{A} \succ \m{0}$), as we can restrict the analysis to the subspace $\overline{\c{R}(\m{A})}$ without loss of generality. In this setting, bounded reachability can be characterized as follows:

\begin{theorem}[Bounded Reachability]\label{thm:bdd:reach}
Let $\m{A}, \m{B} \in \c{B}_{1}^{+}(\c{H})$ be covariance operators with $\c{N}(\m{A}) \subset \c{N}(\m{B})$. Then the following are equivalent:
\begin{enumerate}[leftmargin = *]
    \item $\m{A} \to \m{B}$.
    \item There exists a constant $c \ge 0$ such that $(\m{A}^{1/2} \m{B} \m{A}^{1/2})^{1/2} \preceq c \m{A}$.
    \item There exists some $\m{T} \in \c{B}_{\infty}^{+}(\c{H})$ such that $(\m{A}^{1/2} \m{B} \m{A}^{1/2})^{1/2} = \m{A}^{1/2} \m{T} \m{A}^{1/2}$.
\end{enumerate}
In particular, if $\c{R}(\m{B}^{1/2}) \subset \c{R}(\m{A}^{1/2})$, these conditions are met.
\end{theorem}

\begin{proof}[Proof of \cref{thm:bdd:reach}] \quad
\begin{itemize}
\item [($1 \Rightarrow 3)$] Trivial.
\item [($3 \Rightarrow 1)$] Let $\tilde{\m{T}} = \m{\Pi}_{1} \m{T}$, which trivially satisfies \eqref{eq:cond2:opt} as $\m{A}^{1/2} \m{T} \m{A}^{1/2} = \m{A}^{1/2} \tilde{\m{T}} \m{A}^{1/2}$. On the other hand, we have $\m{A}^{1/2}[\m{B} - (\tilde{\m{T}} \m{A}^{1/2}) (\tilde{\m{T}} \m{A}^{1/2})^{*}] \m{A}^{1/2} = \m{0}$,
hence \eqref{eq:cond:ricatti} is also satisfied:
\begin{align*}
    \m{B} = \m{\Pi}_{1} \m{B} \m{\Pi}_{1} =
    \m{\Pi}_{1} (\tilde{\m{T}} \m{A}^{1/2}) (\tilde{\m{T}} \m{A}^{1/2})^{*} \m{\Pi}_{1}  = (\tilde{\m{T}} \m{A}^{1/2}) (\tilde{\m{T}} \m{A}^{1/2})^{*}.
\end{align*}

\item [($1 \Rightarrow 2)$] Since any bounded operator is closable with a trivial extension, note that $\m{A} \to \m{B}$ if and only if there exists a $\m{T} \in \c{B}_{\infty} (\c{H})$ such that \eqref{eq:cond:ricatti} and \eqref{eq:cond2:opt} hold. Let $\vertiii{\m{T}}_{\infty} = c \ge 0$. Then, for any $h \in \c{H}$,
\begin{equation*}
    \innpr{h}{(\m{A}^{1/2} \m{B} \m{A}^{1/2})^{1/2} h} = \innpr{h}{(\m{A}^{1/2} \m{T} \m{A}^{1/2}) h} = \innpr{\m{A}^{1/2} h}{\m{T} \m{A}^{1/2} h} \le c \|\m{A}^{1/2} h\|^{2} = c \innpr{h}{\m{A} h}.
\end{equation*}

\item [($2 \Rightarrow 1)$] By the Douglas lemma, $\m{A}^{\dagger/2} (\m{A}^{1/2} \m{T} \m{A}^{1/2})^{1/4} \in \c{B}_{\infty}(\c{H})$. Also, since $\c{N}(\m{A}^{1/2}) \subset \c{N}[(\m{A}^{1/2} \m{B} \m{A}^{1/2})^{1/2}]$,
\begin{equation*}
    \m{T}: \c{R}(\m{A}^{1/2}) \rightarrow \c{H}, \m{A}^{1/2} h \rightarrow \m{A}^{\dagger/2} (\m{A}^{1/2} \m{B} \m{A}^{1/2})^{1/2} h,
\end{equation*}
or equivalently, $\m{T} = \m{A}^{\dagger/2} (\m{A}^{1/2} \m{B} \m{A}^{1/2})^{1/2} \m{A}^{\dagger/2}$ is a well-defined operator. Note that $\m{T}$ is bounded with its closure:
\begin{equation*}
    \overline{\m{T}} = [\m{A}^{\dagger/2} (\m{A}^{1/2} \m{T} \m{A}^{1/2})^{1/4}] [\m{A}^{\dagger/2} (\m{A}^{1/2} \m{T} \m{A}^{1/2})^{1/4}]^{*}.
\end{equation*}
Also, since $\m{A}^{\dagger/2} \m{A}^{1/2} = \overline{\m{A}^{1/2} \m{A}^{\dagger/2}} = \m{\Pi}_{1}$, we $\m{T}$
satisfies \eqref{eq:cond:ricatti}:
\begin{align*}
    &(\m{T} \m{A}^{1/2}) (\m{T} \m{A}^{1/2})^{*} = [\m{A}^{\dagger/2} \m{A}^{1/2}] \m{B} [\overline{\m{A}^{1/2} \m{A}^{\dagger/2}}] = \m{\Pi}_{1} \m{B} \m{\Pi}_{1} = \m{B}, \\
    &\m{A}^{1/2} (\m{T} \m{A}^{1/2}) = \m{A}^{\dagger/2} \m{A}^{1/2} (\m{A}^{1/2} \m{B} \m{A}^{1/2})^{1/2} = (\m{A}^{1/2} \m{B} \m{A}^{1/2})^{1/2}.
\end{align*}
\end{itemize}

Finally, if $\c{R}(\m{B}^{1/2}) \subset \c{R}(\m{A}^{1/2})$, there exists some $c \ge 0$ such that $\m{B} \preceq c \m{A}$. Then, we get from the L\"{o}wner-Heinz inequality \cite{lowner1934monotone, kato1952notes} that 
\begin{equation*}
    \m{A}^{1/2} \m{B} \m{A}^{1/2} \preceq c^{2} \m{A}^{2} \quad \Longrightarrow \quad (\m{A}^{1/2} \m{B} \m{A}^{1/2})^{1/2} \preceq (c^{2} \m{A}^{2})^{1/2} = c \m{A}.
\end{equation*}
\end{proof}

\subsection{Non-Injective Case}
We now return to the general case and represent the pre-pushforward $\m{T}$ in block operator form:
\begin{align*}
    \m{T} =
    \begin{pmatrix}
    \m{T}_{11} \\
    \m{T}_{21}
    \end{pmatrix}
    \    :
    \c{R}(\m{A}^{1/2}) \subset \c{H}_{1}
    \rightarrow
    \begin{array}{c}
    \c{H}_{1} \\
    \oplus \\
    \c{H}_{2}
    \end{array}, \text{ where }
    \m{T}_{i1} := \m{\Pi}_{i} \m{T}: \c{R}(\m{A}^{1/2}) \subset \c{H}_{1} \to \c{H}_{i}, \quad i = 1, 2.
\end{align*}

To analyze reachability in this setting, we must account for a known subtlety of compact operators in infinite dimensions.
For compact operators $\m{T} \in \c{B}_{0}(\c{H}_{1}, \c{H}_{2})$, while the Riesz-Schauder theorem ensures that $\rk(\m{T}) = \rk(\m{T}^{*})$, it is possible that $\dim \c{N}(\m{T}) \ne \dim \c{N}(\m{T}^{*})$, a phenomenon studied via the Fredholm index \cite{conway2019course}. For instance, consider an ONB $\{e_{n}\}$ of $\c{H}$, and define $\m{T} = \sum_{n=1}^{\infty} n^{-2025} e_{n} \otimes e_{2n}$. Then $\dim \c{N}(\m{T}) = 0$, whereas $\dim \c{N}(\m{T}^{*}) = \infty$. 


\begin{theorem}\label{thm:sol:sym:inf}
Let $\m{A}, \m{B} \in \c{B}_{1}^{+}(\c{H})$ be covariance operators.
Then $\m{A} \leadsto \m{B}$ if and only if the dimensional inequality holds:
\begin{align*}
    \dim [\c{R}(\m{B}^{1/2}) \cap \c{H}_{2}] \le \dim [\c{N}(\m{B}^{1/2} \m{A}^{1/2}) \cap \c{H}_{1}].
\end{align*}
If it does, any optimal pre-pushforward has the form
\begin{align}\label{eq:sol:set}
    \m{T} = 
    \begin{pmatrix}
    \m{A}_{11}^{\dagger/2} (\m{A}_{11}^{1/2} \m{B}_{11} \m{A}_{11}^{1/2})^{1/2} \m{A}_{11}^{\dagger/2} \\
    [(\m{A}_{11}^{1/2} \m{B}_{11} \m{A}_{11}^{1/2})^{\dagger/2} \m{A}_{11}^{1/2} \m{B}_{12} + \m{U}_{12} (\m{B}/\m{A})^{1/2}]^{*} \m{A}_{11}^{\dagger/2}
    \end{pmatrix}
    \    :
    \c{R}(\m{A}^{1/2}) \subset \c{H}_{1}
    \rightarrow
    \begin{array}{c}
    \c{H}_{1} \\
    \oplus \\
    \c{H}_{2}
    \end{array},
\end{align}
where $\m{U}_{12}$ is a partial isometry with initial space $\overline{\c{R}(\m{B}/\m{A})}$, and final space contained in $\c{N}(\m{B}_{11}^{1/2} \m{A}_{11}^{1/2})$.
\end{theorem}

\begin{proof}[Proof of \cref{thm:sol:sym:inf}]
\quad
\begin{itemize}[leftmargin = *]
\item [($\Rightarrow$)] Choose an optimal pre-pushforward $\m{T}: \c{R}(\m{A}^{1/2}) \subset \c{H}_{1} \rightarrow \c{H}$. Note that \eqref{eq:cond:ricatti} amounts to the fact that $\m{P} := (\m{T} \m{A}^{1/2})^{\dagger} \m{B}^{1/2}$ is a bounded operator satisfying $\m{B}^{1/2} = (\m{T} \m{A}^{1/2}) \m{P}$. Also, from \eqref{eq:cond2:opt}, we have $\m{A}^{1/2} \m{T} \m{A}^{1/2} = (\m{A}^{1/2} \m{B} \m{A}^{1/2})^{1/2}$, which is a symmetric bounded operator.

To prove the claim, it suffices to show that
\begin{equation*}
    \m{P} \m{B}^{\dagger/2} \mid_{\c{H}_{2}} : \c{R}(\m{B}^{1/2}) \cap \c{H}_{2} \rightarrow \c{N}(\m{B}^{1/2} \m{A}^{1/2}) \cap \c{H}_{1}.
\end{equation*}
is a well-defined injective operator. For well-definedness, we claim that $\m{P} \m{B}^{\dagger/2} y \in \c{N}(\m{B}^{1/2} \m{A}^{1/2}) \cap \c{H}_{1}$ for any $y \in \c{R}(\m{B}^{1/2}) \cap \c{H}_{2}$. By letting $z = \m{B}^{\dagger/2} y$ so that
\begin{align}\label{eq:injec:contra}
    y = \m{B}^{1/2} z = (\m{T} \m{A}^{1/2}) \m{P} z \in \c{H}_{2},
\end{align}
it is equivalent to showing that $\m{P} z = \m{P} \m{B}^{\dagger/2} y \in \c{N}(\m{B}^{1/2} \m{A}^{1/2}) \cap \c{H}_{1}$. For any $x \in \c{H}_{1}$,
\begin{align*}
    0 = \innpr{\m{A}^{1/2} x}{y} = \innpr{\m{A}^{1/2} x}{(\m{T} \m{A}^{1/2}) \m{P} z}
    = \innpr{x}{(\m{A}^{1/2} \m{T} \m{A}^{1/2}) \m{P} z},
\end{align*}
which establishes $(\m{A}^{1/2} \m{T} \m{A}^{1/2}) \m{P} z  = 0$, thus 
\begin{equation*}
    \m{P} z \in \c{N} (\m{A}^{1/2} \m{T} \m{A}^{1/2}) = \c{N}[(\m{A}^{1/2} \m{B} \m{A}^{1/2})^{1/2}] = \c{N}(\m{B}^{1/2} \m{A}^{1/2}).
\end{equation*}
On the other hand,
\begin{align*}
    \m{P} z \in \c{R}[(\m{T} \m{A}^{1/2})^{\dagger}] = \c{N}(\m{T} \m{A}^{1/2})^{\perp} \subset \c{N}(\m{A}^{1/2})^{\perp} = \c{H}_{1}.
\end{align*}

Finally, the injectivity of $\m{P} \m{B}^{\dagger/2}: y \mapsto \m{P} z$ is immediate from \eqref{eq:injec:contra}: $\m{P} z = 0$ if and only if $y = 0$.

\item [($\Leftarrow$)] 
Let $\m{T}: \c{R}(\m{A}^{1/2}) \subset \c{H}_{1} \rightarrow \c{H}$, and define $\m{M}_{i1} := \m{T}_{i1} \m{A}_{11}^{1/2} \in \c{B}_{\infty}(\c{H}_{1}, \c{H}_{i})$ for $i=1, 2$. Then \eqref{eq:cond:ricatti} becomes: 
\begin{equation}\label{eq:cond:ricatti:pf}
    \m{M}_{i1} \m{M}_{j1}^{*} = \m{B}_{ij}, \quad i, j = 1, 2,
\end{equation}
and \eqref{eq:cond2:opt} becomes:
\begin{equation*}
    \tr[ \m{A}_{11} \m{M}_{11}] = \tr [(\m{A}_{11}^{1/2} \m{B}_{11} \m{A}^{1/2})^{1/2}].
\end{equation*}
These two conditions determine $\m{T}_{11}$ and $\m{M}_{11}$ uniquely:
\begin{align*}
    &\m{T}_{11} = \m{S}_{\m{A}_{11} \leadsto \m{B}_{11}} = \m{A}_{11}^{\dagger/2} (\m{A}_{11}^{1/2} \m{B}_{11} \m{A}_{11}^{1/2})^{1/2} \m{A}_{11}^{\dagger/2}: \c{R}(\m{A}_{11}^{1/2}) \subset \c{H}_{1} \rightarrow \c{H}_{1}, \\
    &\m{M}_{11} = \m{T}_{11} \m{A}_{11}^{1/2} = \m{A}_{11}^{\dagger/2} (\m{A}_{11}^{1/2} \m{B}_{11} \m{A}_{11}^{1/2})^{1/2} : \c{H}_{1} \rightarrow \c{H}_{1}.
\end{align*}

Hence, it suffices to show that there exists $\m{M}_{21}$ satisfying \eqref{eq:cond:ricatti:pf} for $(i, j) =(1, 2), (2, 2)$ (as the case $(2, 1)$ is redundant):
\begin{align}
    \m{A}_{11}^{\dagger/2} (\m{A}_{11}^{1/2} \m{B}_{11} \m{A}_{11}^{1/2})^{1/2} \m{M}_{21}^{*} &= \m{B}_{12}: \c{H}_{2} \rightarrow \c{H}_{1}, \label{eq:cond:ricatti:pf1} \\
    \m{M}_{21} \m{M}_{21}^{*} &= \m{B}_{22}: \c{H}_{2} \rightarrow \c{H}_{2}. \label{eq:cond:ricatti:pf2}
\end{align}

Consider a s.p.d. block operator defined by:
\begin{align}\label{eq:comp:block}
    \m{C} := \begin{pmatrix}
        \m{C}_{11} & \m{C}_{12} \\
        \m{C}_{21} & \m{C}_{22}
    \end{pmatrix} = \begin{pmatrix}
        \m{A}_{11}^{1/2} \m{B}_{11} \m{A}_{11}^{1/2} & \m{A}_{11}^{1/2} \m{B}_{12} \\
        \m{B}_{21} \m{A}_{11}^{1/2} & \m{B}_{22}
    \end{pmatrix}
    \    :
    \begin{array}{c}
    \c{H}_{1} \\
    \oplus \\
    \c{H}_{2}
    \end{array}
    \rightarrow
    \begin{array}{c}
    \c{H}_{1} \\
    \oplus \\
    \c{H}_{2}
    \end{array}.
\end{align}
Since $\c{H}_{1} = \c{N}(\m{A}_{11}^{1/2})^{\perp}$, \eqref{eq:cond:ricatti:pf1} is equivalent to $\m{C}_{11}^{1/2} \m{M}_{21}^{*} = \m{C}_{12}$, which always have a solution characterized by
\begin{equation*}
    \m{M}_{21}^{*} = \m{C}_{11}^{\dagger/2} \m{C}_{12} \oplus \m{N}_{12}, \quad \m{N}_{12} \in \c{B}(\c{H}_{2}, \c{H}_{1}), \quad \c{R}(\m{N}_{12}) \subset \c{N}(\m{C}_{11}).
\end{equation*}
Applying this form into \eqref{eq:cond:ricatti:pf2}, we obtain from \cref{lem:Schur:op} that $\m{N}_{12}^{*} \m{N}_{12} = \m{C}/\m{A} = \m{B}/\m{A}$.
Therefore, it suffices to prove that there is some $\m{N}_{12}: \c{H}_{2} \rightarrow \c{H}_{1}$ such that $\c{R}(\m{N}_{12}) \subset \c{N} (\m{B}_{11}^{1/2} \m{A}_{11}^{1/2})$ and $\m{N}_{12}^{*} \m{N}_{12} = \m{B}/\m{A}$. Given that
\begin{align*}
    \rk (\m{B} / \m{A}) = \rk [(\m{B} / \m{A})^{1/2}] &= \dim [\c{R}(\m{B}^{1/2}) \cap \c{H}_{2}] \\
    &\le \dim [\c{N}(\m{B}^{1/2} \m{A}^{1/2}) \cap \c{H}_{1}] = \dim \c{N} (\m{B}_{11}^{1/2} \m{A}_{11}^{1/2}),
\end{align*}
due to \cref{lem:Schur:op}, there is a partial isometry $\m{U}_{12}$ with initial space $\overline{\c{R}(\m{B}/\m{A})}$, and final space included in $\c{N}(\m{B}_{11}^{1/2} \m{A}_{11}^{1/2})$. Then, $\m{N}_{12} = \m{U}_{12} (\m{B}/\m{A})^{1/2}$ is a desired solution, and any such solution takes this form.
\end{itemize}
\end{proof}

\begin{remark}
When $\c{H} = \b{R}^{n}$, the dimensional inequality reduces to:
\begin{align*}
    \rk(\m{B}) - \rk(\m{A}^{1/2} \m{B} \m{A}^{1/2}) &= \rk(\m{B}/\m{A}) = \dim [\c{R}(\m{B}^{1/2}) \cap \c{H}_{2}] \\
    &\le \dim [\c{N}(\m{B}^{1/2} \m{A}^{1/2}) \cap \c{H}_{1}] = \rk(\m{A}) - \rk(\m{A}^{1/2} \m{B} \m{A}^{1/2}), 
\end{align*}
so we recover the condition in \cref{thm:sol:sym} : $\rk(\m{B}) \le \rk(\m{A})$.
\end{remark}

The following corollary reveals that the Kantorovich problem can always be reduced to the Monge problem, either from $\m{A}$ to $\m{B}$ or vice versa.

\begin{corollary}\label{cor:reach:connect}
Let $\m{A}, \m{B} \in \c{B}_{1}^{+}(\c{H})$ be covariance operators.
Then, either $\m{A} \leadsto \m{B}$ or $\m{B} \leadsto \m{A}$ holds.
\end{corollary}

\begin{proof}[Proof of \cref{cor:reach:connect}]
Suppose neither $\m{A} \leadsto \m{B}$ or $\m{B} \leadsto \m{A}$ holds. From \cref{thm:sol:sym:inf}, this is equivalent to
\begin{align*}
    &\dim [\c{R}(\m{B}^{1/2}) \cap \c{N}(\m{A}^{1/2})] > \dim [\c{N}(\m{B}^{1/2} \m{A}^{1/2}) \cap \overline{\c{R}(\m{A}^{1/2})}], \\
    &\dim [\c{R}(\m{A}^{1/2}) \cap \c{N}(\m{B}^{1/2})] > \dim [\c{N}(\m{A}^{1/2} \m{B}^{1/2}) \cap \overline{\c{R}(\m{B}^{1/2})}].
\end{align*}

We claim that
\begin{align*}
    f \in \c{R}(\m{A}^{1/2}) \cap \c{N}(\m{B}^{1/2}) \mapsto \m{A}^{\dagger/2} f \in \c{N}(\m{B}^{1/2} \m{A}^{1/2}) \cap \overline{\c{R}(\m{A}^{1/2})},
\end{align*}
is a well-defined injective operator. For any $f \in \c{R}(\m{A}^{1/2}) \cap \c{N}(\m{B}^{1/2})$, there exists a unique $g = \m{A}^{\dagger/2} f \in \overline{\c{R}(\m{A}^{1/2})}$ such that $f = \m{A}^{1/2} g$. This yields
\begin{align*}
    \m{B}^{1/2} \m{A}^{1/2} g = \m{B}^{1/2} f = 0 \quad \Rightarrow \quad g \in \c{N}(\m{B}^{1/2} \m{A}^{1/2}).
\end{align*}
The injectivity is straightforward since $g = 0$ implies $f = \m{A}^{1/2} g = 0$. Therefore, we get
\begin{align*}
    \dim [\c{N}(\m{B}^{1/2} \m{A}^{1/2}) \cap \overline{\c{R}(\m{A}^{1/2})}] \ge \dim [\c{R}(\m{A}^{1/2}) \cap \c{N}(\m{B}^{1/2})].
\end{align*}
and similarly,
\begin{align*}
    \dim [\c{N}(\m{A}^{1/2} \m{B}^{1/2}) \cap \overline{\c{R}(\m{B}^{1/2})}] \ge \dim [\c{R}(\m{B}^{1/2}) \cap \c{N}(\m{A}^{1/2})].
\end{align*}
However, this results in
\begin{align*}
    \dim [\c{R}(\m{B}^{1/2}) \cap \c{N}(\m{A}^{1/2})] &> \dim [\c{N}(\m{B}^{1/2} \m{A}^{1/2}) \cap \overline{\c{R}(\m{A}^{1/2})}] \ge \dim [\c{R}(\m{A}^{1/2}) \cap \c{N}(\m{B}^{1/2})] \\
    &> \dim [\c{N}(\m{A}^{1/2} \m{B}^{1/2}) \cap \overline{\c{R}(\m{B}^{1/2})}] \ge \dim [\c{R}(\m{B}^{1/2}) \cap \c{N}(\m{A}^{1/2})],
\end{align*}
which is a contradiction.
\end{proof}

Unlike the finite-dimensional case, constructing a symmetric extension of an optimal pre-pushforward is challenging. For such an extension to exist, $\m{T}$ must be symmetric, which requires the operator block $\m{T}_{21}$ in \eqref{eq:sol:set} to be closable. This closability condition is equivalent to
\begin{align*}
    \m{A}_{11}^{1/2} h_{n} \to 0, \quad 
    [(\m{A}_{11}^{1/2} \m{B}_{11} \m{A}_{11}^{1/2})^{\dagger/2} \m{A}_{11}^{1/2} \m{B}_{12} + \m{U}_{12} (\m{B}/\m{A})^{1/2}]^{*} h_{n} \to f \quad \Rightarrow \quad f =0.
\end{align*}

It is trivial that any s.p.d. pushforward, if one exists, is optimal. The following theorem reveals that a necessary condition for the existence of an s.p.d. pushforward is the uniqueness of the optimal pre-pushforward:

\begin{theorem}\label{thm:spd:inf}
Let $\m{A}, \m{B} \in \c{B}_{1}^{+}(\c{H})$ be covariance operators. The following statements are equivalent:
\begin{enumerate}[leftmargin = *]
    \item There exists a unique optimal pre-pushforward from $\m{A}$ to $\m{B}$, given by
\begin{align}\label{eq:spd:pre}
    \m{T} = \begin{pmatrix}
    \m{T}_{11} \\
    \m{T}_{21}
    \end{pmatrix} :=
    \begin{pmatrix}
    \m{A}_{11}^{\dagger/2} (\m{A}_{11}^{1/2} \m{B}_{11} \m{A}_{11}^{1/2})^{1/2} \m{A}_{11}^{\dagger/2} \\
    ((\m{A}_{11}^{1/2} \m{B}_{11} \m{A}_{11}^{1/2})^{\dagger/2} \m{A}_{11}^{1/2} \m{B}_{12})^{*} \m{A}_{11}^{\dagger/2}
    \end{pmatrix}
    \    :
    \c{R}(\m{A}^{1/2}) \subset \c{H}_{1}
    \rightarrow
    \begin{array}{c}
    \c{H}_{1} \\
    \oplus \\
    \c{H}_{2}
    \end{array}.
\end{align}
    \item $\m{B}/\m{A} = \m{0}$.
    \item $\c{N} (\m{B}^{1/2}) = \c{N}(\m{A}^{1/2} \m{B}^{1/2})$.
    \item $ \c{R}(\m{B}^{1/2}) \cap \c{N}(\m{A}) = \{ \m{0} \}$.    
\end{enumerate}
If an s.p.d. extension of an optimal pre-pushforward exists, then any of these equivalent conditions hold and $\m{T}_{21}: \c{R}(\m{A}^{1/2}) \to \c{H}_{2}$ from \eqref{eq:spd:pre} is closable.
\end{theorem}

\begin{proof}[Proof of \cref{thm:spd:inf}] \quad
\begin{itemize}[leftmargin = *]
\item [($1 \Leftrightarrow 2)$] Trivial since $\dim[\c{R} [(\m{B}/\m{A})^{1/2}]] = \dim [\c{R}(\m{B}^{1/2}) \cap \c{H}_{2}]$ due to \cref{lem:Schur:op}, and $- \m{U}_{12}$ in \cref{thm:sol:sym:inf} is also a partial isometry with initial space $\overline{\c{R}(\m{B}/\m{A})}$, and final space included in $\c{N}(\m{B}_{11}^{1/2} \m{A}_{11}^{1/2})$.
\item [($3 \Leftrightarrow 4)$] Trivial by definition.
\item [($2 \Leftrightarrow 4)$] Direct consequence of part 4 in \cref{lem:Schur:op}.
\end{itemize}

We now claim that an s.p.d. extension of an optimal pre-pushforward is possible only if $\m{B}/\m{A} = \m{0}$ and $\m{T}_{21}: \c{R}(\m{A}^{1/2}) \to \c{H}_{2}$ in \eqref{eq:spd:pre} is closable.
Let $\m{S} : \c{D}(\m{S}) \subset \c{H} \to \c{H}$ an s.p.d. pushforward. Assume that $\m{B}/\m{A} \succeq \m{0}$ does not vanish. We may restrict $\m{S}$ so that $\c{D}(\m{S}) \cap \c{H}_{1} = \c{R}(\m{A}^{1/2})$, which is still an s.p.d. operator. Note that $\m{S}_{12} \subset \m{S}_{21}^{*}$ due to symmetry, and
\begin{equation*}
    \m{A}_{11}^{1/2} \m{S}_{11} \m{A}_{11}^{1/2} \succeq \m{0} \quad \Rightarrow \quad \m{S}_{11}= \m{A}_{11}^{\dagger/2} (\m{A}_{11}^{1/2} \m{B}_{11} \m{A}_{11}^{1/2})^{1/2} \m{A}_{11}^{\dagger/2},
\end{equation*}
which shows that $\m{S}$ is an extension of an optimal pre-pushforward. Thus, due to \cref{thm:sol:sym:inf}, there exists a partial isometry $\m{U}_{12}$ with initial space $\overline{\c{R}(\m{B}/\m{A})}$, and final space included in $\c{N}(\m{B}_{11}^{1/2} \m{A}_{11}^{1/2})$ such that
\begin{equation}\label{eq:spd:prdx}
    (\m{S}_{21} \m{A}_{11}^{1/2})^{*} - (\m{A}_{11}^{1/2} \m{B}_{11} \m{A}_{11}^{1/2})^{\dagger/2} \m{A}_{11}^{1/2} \m{B}_{12} = \m{U}_{12} (\m{B}/\m{A})^{1/2}.
\end{equation}

We now claim that $\c{N}(\m{S}_{11}) \subset \c{N}(\m{S}_{12}^{*})$. Since $\m{S} \in \c{D}(\m{H})$ is densely defined, we have $\c{R}(\m{S}_{12})^{\perp} = \c{N}(\m{S}_{12}^{*})$, thus it suffices to show that $\innpr{h_{1}}{\m{S}_{12} g_{2}} = 0$ for any $h_{1} \in \c{N}(\m{S}_{11})$ and $g_{2} \in \c{D}(\m{S}_{12})$. For any $\alpha \in \b{R}$, we have $h_{1} \oplus \alpha g_{2} \in \c{D}(\m{S})$, thus
\begin{align*}
    0 \le \innpr{(h_{1} \oplus \alpha g_{2})}{\alpha \m{S} (h_{1} \oplus \alpha g_{2})} 
    &= \innpr{h_{1} \oplus \alpha g_{2}}{\alpha \m{S}_{12} g_{2} \oplus (\m{S}_{21} h_{1} + \alpha \m{S}_{22} g_{2})} \\
    &= \alpha \innpr{h_{1}}{\m{S}_{12} g_{2}} + \alpha \innpr{g_{2}}{\m{S}_{21} h_{1}} + \alpha^{2} \innpr{g_{2}}{\m{S}_{22} g_{2}} \\
    &= 2\alpha \innpr{h_{1}}{\m{S}_{12} g_{2}} + \alpha^{2} \innpr{g_{2}}{\m{S}_{22} g_{2}},
\end{align*}
where we used
\begin{align*}
    \innpr{g_{2}}{\m{S}_{21} h_{1}} = \innpr{g_{2}}{\m{\Pi}_{2} \m{S} \m{\Pi}_{1} h_{1}} = \innpr{\m{\Pi}_{2}  g_{2}}{\m{S} \m{\Pi}_{1} h_{1}} = \innpr{\m{S} \m{\Pi}_{2}  g_{2}}{ \m{\Pi}_{1} h_{1}} = \innpr{\m{\Pi}_{1}  \m{S} \m{\Pi}_{2}  g_{2}}{ h_{1}} = \innpr{\m{S}_{12} g_{2}}{h_{1}}.
\end{align*}
In order that the above inequality holds for any $\alpha \in \b{R}$, we get $\innpr{h_{1}}{\m{S}_{12} g_{2}} = 0$, which proves the claim. 

From the claim above, we have
\begin{align*}
    \c{N}(\m{B}_{11}^{1/2} \m{A}_{11}^{1/2}) = \c{N}(\m{A}_{11}^{1/2} \m{B}_{11} \m{A}_{11}^{1/2}) = \c{N} (\m{A}_{11}^{1/2} \m{S}_{11} \m{A}_{11}^{1/2}) = \c{N}(\m{S}_{11}^{1/2} \m{A}_{11}^{1/2}) 
    \subset \c{N}(\m{S}_{11} \m{A}_{11}^{1/2}) \subset \c{N}(\m{S}_{12}^{*} \m{A}_{11}^{1/2}).
\end{align*}
Additionally, since $\m{S}_{21} \m{A}_{11}^{1/2} \in \c{B}_{\infty}(\c{H}_{1}, \c{H}_{2})$ and $\m{S}_{21} \subset \m{S}_{12}^{*}$, we have $\m{S}_{12}^{*} \m{A}_{11}^{1/2} = \m{S}_{21} \m{A}_{11}^{1/2}$. Hence, we get
\begin{align*}
    \c{N}(\m{B}_{11}^{1/2} \m{A}_{11}^{1/2}) \subset \c{N}(\m{S}_{21} \m{A}_{11}^{1/2}) \quad \Rightarrow \quad \c{R}[(\m{S}_{21} \m{A}_{11}^{1/2})^{*}] \subset \c{N}(\m{B}_{11}^{1/2} \m{A}_{11}^{1/2})^{\perp}.
\end{align*}
However, this entails that the range of (LHS) in \eqref{eq:spd:prdx} belongs to $\c{N}(\m{B}_{11}^{1/2} \m{A}_{11}^{1/2})^{\perp} = \c{N}(\m{A}_{11}^{1/2} \m{B}_{11} \m{A}_{11}^{1/2})^{\perp}$, whereas the range of the (RHS) is included in $\c{N}(\m{B}_{11}^{1/2} \m{A}_{11}^{1/2})$. Therefore, we conclude $\m{B}/\m{A} =0$, which is a contradiction. This also incorporates that $\m{S}_{21} = \m{T}_{21}$, so $\m{T}_{21}$ is closable.
\end{proof}

\begin{remark}
Constructing a concrete s.p.d. extension is a delicate issue. To illustrate the difficulty, recall the block components of the optimal pre-pushforward $\m{T}$ from \eqref{eq:spd:pre}:
\begin{align*}
    \m{T}_{11} = \m{A}_{11}^{\dagger/2} \m{C}_{11}^{1/2} \m{A}_{11}^{\dagger/2}, \quad \m{T}_{21} = (\m{C}_{11}^{\dagger/2} \m{C}_{12})^{*} \m{A}_{11}^{\dagger/2}.
\end{align*}
where $\m{C} \in \c{B}_{1}^{+}(\c{H})$ is the block s.p.d. operator in \eqref{eq:comp:block}, 
A natural candidate for the s.p.d. extension is the operator $\m{S}$ formed by analogy with the finite-dimensional solution \eqref{eq:gen:pusz:fin}:
\begin{align*}
    \m{S} := \begin{pmatrix}
        \m{T}_{11} & \m{T}_{21}^{*} \\
        \m{T}_{21} & (\m{C}_{11}^{\dagger/2} \m{C}_{12})^{*} \m{C}_{11}^{\dagger} (\m{C}_{11}^{\dagger/2} \m{C}_{12})
    \end{pmatrix}.
\end{align*}
However, this formal construction is problematic -- verifying whether this is even densely defined is intractable.

\end{remark}

\subsection{Kantorovich Geodesics}
Given covariance operators $\m{A}, \m{B} \in \c{B}_{1}^{+}(\c{H})$, we may assume $\m{A} \leadsto \m{B}$ without loss of generality due to \cref{cor:reach:connect}. As in \cref{ssec:kanto}, we consider
\begin{equation*}
    \c{K}_{t}(\m{A}, \m{B}) := \argmin_{\m{C} \in \c{B}_{1}^{+}(\c{H})}  (1-t) \c{W}_{2}^{2}(\m{A}, \m{C}) + t \c{W}_{2}^{2}(\m{B}, \m{C}), \quad t \in [0, 1].
\end{equation*}
Obviously, $\m{\Gamma}_{0} = \m{A}$ and $\m{\Gamma}_{1} = \m{B}$. 

\begin{proposition}\label{prop:Kanto:geo:inf}
Let $\m{A}, \m{B} \in \c{B}_{1}^{+}(\c{H})$ be covariance operators with $\m{A} \leadsto \m{B}$ and $t \in [0, 1]$. Then, 
\begin{align*}
    \c{K}_{t}(\m{A}, \m{B}) = \left\{ \m{\Gamma}_{t} = \m{G}_{t} \m{G}_{t}^{*} : \m{G}_{t} = (1-t) \m{G} + t \m{M}, \, \m{M} \in \s{G}(\m{B}), \ \m{G}^{*} \m{M} \succeq \m{0} \right\} \ne \emptyset,
\end{align*}
which does not depend on the choice of $\m{G} \in \c{G}(\m{A})$.
\end{proposition}

\begin{proof}[Proof of \cref{prop:Kanto:geo:inf}]
We first show that (RHS) is not an emptyset. Assume $\m{A} \leadsto \m{B}$ from \cref{cor:reach:connect}. Let $\m{M} = \m{T} \m{G}$, where $\m{T}: \c{R}(\m{G}) = \c{R}(\m{A}^{1/2}) \subset \c{H}_{1} \rightarrow \c{H}$ is an optimal pre-pushforward. Then, we get $\m{M} \in \s{G}(\m{B})$ and $\m{G}^{*} \m{M} \succeq \m{0}$. Note from \cref{lem:Neumann:eq,lem:opt:algn:inf-dim} that $\m{G}^{*} \m{M} \succeq \m{0}$ is equivalent to
\begin{align*}
    \tr[\m{G}^{*} \m{M}] = \tr [(\m{G}^{*} \m{B} \m{G})^{1/2}] = \tr [(\m{A}^{1/2} \m{B} \m{A}^{1/2})^{1/2}].
\end{align*}

Once we establish the equivalence, it becomes clear that (RHS) does not depend on the choice of $\m{G} \in \c{G}(\m{A})$, as (LHS) does not depend on the choice of $\m{G} \in \c{G}(\m{A})$:

\begin{itemize}[leftmargin = *]
\item [$(\supseteq)$] For any $\m{\Gamma}_{t}$ in (RHS), by direct calculation as in \cref{thm:kanto:geo}, we get $\c{W}_{2}(\m{A}, \m{\Gamma}_{t}) = t \c{W}_{2}(\m{A}, \m{B})$ and $\c{W}_{2}(\m{B}, \m{\Gamma}_{t}) = (1-t) \c{W}_{2}(\m{A}, \m{B})$. Then, invoking the proof of \cref{thm:kanto:geo} again, we get
\begin{align*}
    (1-t) \c{W}_{2}^{2}(\m{A}, \m{C}) + t \c{W}_{2}^{2}(\m{B}, \m{C}) \ge
    t (1-t) \c{W}_{2}^{2}(\m{A}, \m{B}) = (1-t) \c{W}_{2}^{2}(\m{A}, \m{\Gamma}_{t}) + t \c{W}_{2}^{2}(\m{B}, \m{\Gamma}_{t}),
\end{align*}
for any $\m{C} \in \c{B}_{1}^{+}(\c{H})$, i.e., $\m{\Gamma}_{t} \in \c{K}_{t}(\m{A}, \m{B})$.

\item [$(\subseteq)$] Given $\m{C} \in \c{K}_{t}(\m{A}, \m{B})$, let $\m{M} \in \c{G}(\m{B})$, and $\m{R} \in \c{G}(\m{C})$. Then,
\begin{align*}
    (1-t) \c{W}_{2}^{2}(\m{A}, \m{C}) + t \c{W}_{2}^{2}(\m{B}, \m{C}) &\stackrel{(a)}{\ge} (1-t) \vertii{\m{G} - \m{R}}_{2}^{2} + t \vertii{\m{M} - \m{R}}_{2}^{2} \\
    &= \vertii{\m{R}}_{2}^{2} - 2 \tr [\m{R}^{*} ((1-t) \m{G} + t \m{M})] + (1-t) \vertii{\m{G}}_{2}^{2} + t \vertii{\m{M}}_{2}^{2} \\
    &\stackrel{(b)}{\ge} - \vertii{(1-t) \m{G} + t \m{M}}_{2}^{2} + (1-t) \vertii{\m{G}}_{2}^{2} + t \vertii{\m{M}}_{2}^{2} \\
    &= t(1-t) \vertii{\m{G} - \m{M}}_{2}^{2} \\
    &\stackrel{(c)}{\ge} t(1-t) \tr[\m{A} + \m{B} - 2(\m{A}_{1}^{1/2} \m{A}_{2} \m{A}_{1}^{1/2})^{1/2}] = t(1-t) \c{W}_{2}^{2}(\m{A}, \m{B}).
\end{align*}
The equality holds in (c) if and only if $\m{G}^{*} \m{M} \succeq \m{0}$ and the equality holds in (b) if and only if $\m{R} = (1-t) \m{G} + t \m{M}$. In this case, the equality automatically holds in (a). Hence $\m{C}$ belongs to (RHS).
\end{itemize}
\end{proof}

\begin{theorem}[Kantorovich Problem]\label{thm:Kanto:char:inf}
Let $\m{A}, \m{B} \in \c{B}_{1}^{+}(\c{H})$ be covariance operators with $\m{A} \leadsto \m{B}$. Then, the set of all operators $\m{M} \in \s{G}(\m{B})$ satisfying $\m{A}^{1/2} \m{M} \succeq \m{0}$ is given by
\begin{align*}
    \m{M} =
    \begin{pmatrix}
    \m{A}_{11}^{\dagger/2} (\m{A}_{11}^{1/2} \m{B}_{11} \m{A}_{11}^{1/2})^{1/2} & \m{0} \\
    [(\m{A}_{11}^{1/2} \m{B}_{11} \m{A}_{11}^{1/2})^{\dagger/2} \m{A}_{11}^{1/2} \m{B}_{12} + \m{N}_{12}]^{*} & \m{M}_{22}
    \end{pmatrix}
    \    :
    \begin{array}{c}
    \c{H}_{1} \\
    \oplus \\
    \c{H}_{2}
    \end{array}
    \rightarrow
    \begin{array}{c}
    \c{H}_{1} \\
    \oplus \\
    \c{H}_{2}
    \end{array} ,
\end{align*}
where: 
\begin{itemize}[leftmargin = *]
    \item $\m{N}_{12}  : \c{H}_{2} \to \c{H}_{1}$ satisfies $\c{R}(\m{N}_{12}) \subset \c{N}(\m{B}_{11}^{1/2} \m{A}_{11}^{1/2})$ and $\m{N}_{12}^{*} \m{N}_{12} \preceq \m{B}/\m{A}$.
    \item $\m{M}_{22}  : \c{H}_{2} \to \c{H}_{2}$ satisfies $\m{M}_{22} \m{M}_{22}^{*} = \m{B}/\m{A} - \m{N}_{12}^{*} \m{N}_{12}$.
\end{itemize}
A solution corresponds to a Monge map (i.e., $\m{M} = \m{T}\m{A}^{1/2}$ for some optimal pre-pushforward $\m{T}: \c{R}(\m{A}^{1/2}) \subset \c{H}_{1} \rightarrow \c{H}$) if and only if one chooses $\m{M}_{22} = \m{0}$ (i.e., $\m{N}_{12}^{*} \m{N}_{12} = \m{B}/\m{A}$).
\end{theorem}

\begin{proof}[Proof of \cref{thm:Kanto:char:inf}]
Decompose $\m{A}^{1/2}$ and $\m{M}$ as:
\begin{equation*}
    \m{A}^{1/2} = \begin{pmatrix}
    \m{A}_{11}^{1/2} & \m{0} \\
    \m{0} & \m{0}
    \end{pmatrix}, \quad \m{M} = \begin{pmatrix}
    \m{M}_{11} & \m{M}_{12} \\
    \m{M}_{21} & \m{M}_{22}
    \end{pmatrix}.
\end{equation*}
Repeating the proof of \cref{thm:Kanto:char}, we get $\m{M}_{11} = \m{A}_{11}^{\dagger/2} (\m{A}_{11}^{1/2} \m{B}_{11} \m{A}_{11}^{1/2})^{1/2}$ and $\m{M}_{12} = \m{0}$, and it remains to solve:
\begin{align}
    &\m{A}_{11}^{\dagger/2} (\m{A}_{11}^{1/2} \m{B}_{11} \m{A}_{11}^{1/2})^{1/2} \m{M}_{21}^{*} = \m{B}_{12}, \label{eq:reach:rk:1:inf} \\
    &\m{M}_{22} \m{M}_{22}^{*} + \m{M}_{21} \m{M}_{21}^{*} = \m{B}_{22}. \label{eq:reach:rk:2:inf}
\end{align}
As in the proof of \cref{thm:sol:sym:inf},  \eqref{eq:reach:rk:1} always have a solution characterized by
\begin{equation*}
    \m{M}_{21}^{*} = (\m{A}_{11}^{1/2} \m{B}_{11} \m{A}_{11}^{1/2})^{\dagger/2} (\m{A}_{11}^{1/2} \m{B}_{12}) + \m{N}_{12}, \quad \c{R}(\m{N}_{12}) \subset \c{N}((\m{A}_{11}^{1/2} \m{B}_{11} \m{A}_{11}^{1/2})^{1/2}) = \c{N}(\m{B}_{11}^{1/2} \m{A}_{11}^{1/2}).
\end{equation*}
The rest of the proof is then same to that of \cref{thm:Kanto:char}.
\end{proof}

Consequently, \cref{rmk:linear:kanto} holds if one replaces $(\m{G}_{11}^{*})^{-1}$ by $\m{A}_{11}^{\dagger/2}$. In particular, for any $\m{\Gamma}_{t} \in \c{K}_{t}(\m{A}, \m{B})$, $(\m{\Gamma}_{t})_{11}$ and $(\m{\Gamma}_{t})_{22}$ are fixed, and the off-diagonal blocks $(\m{\Gamma}_{t})_{21} = (\m{\Gamma}_{t})_{12}^{*}$ only depends on the choice of $\m{N}_{12} : \c{H}_{2} \to \c{H}_{1}$, subject to the conditions in \cref{thm:Kanto:char:inf}. To emphasize this dependence, we again write $\m{\Gamma}_{t}^{\m{N}_{12}}$. This leads to the counterpart of \cref{thm:spd:revise}.

\begin{corollary}[Uniqueness]\label{thm:spd:revise:inf}
Let $\m{A}, \m{B} \in \c{B}_{1}^{+}(\c{H})$ be covariance operators. For any $t \in (0, 1)$, the set of barycenters $\c{K}_{t}(\m{A}, \m{B})$ is a singleton if and only if the equivalent conditions in \cref{thm:spd:inf} hold.
Furthermore, this unique barycenter is the McCann interpolant at time $t$.
\end{corollary}

\begin{proof}[Proof of \cref{thm:spd:revise:inf}]
Trivial from \cref{prop:Kanto:geo:inf,thm:Kanto:char:inf}.
\end{proof}

Also, all the geodesics connecting $\m{A}$ and $\m{B}$ can be constructed in a same fashion to \cref{thm:kanto:geo}:

\begin{theorem}[Kantorovich Geodesics]\label{thm:kanto:geo:inf}
Let $\m{A}, \m{B} \in \c{B}_{1}^{+}(\c{H})$ be covariance operators with $\m{A} \leadsto \m{B}$.
\begin{enumerate}[leftmargin = *]
\item Any path $\m{\Gamma}^{\m{N}_{12}}: [0, 1] \to \c{B}_{1}^{+}(\c{H}), \, t \mapsto \m{\Gamma}_{t}^{\m{N}_{12}}$ constructed by fixing a choice of $\m{N}_{12}$ in \cref{thm:Kanto:char:inf} generates a constant-speed geodesic, that is, $\c{W}_{2}(\m{\Gamma}_{s}^{\m{N}_{12}}, \m{\Gamma}_{t}^{\m{N}_{12}}) = (t-s) \c{W}_{2}(\m{A}, \m{B})$ for any $0 \le t \le s \le 1$.
\item For any $t \in (0, 1)$ and $\m{\Gamma}_{t} \in \c{K}_{t}(\m{A}, \m{B})$, the \emph{reverse} Schur complements vanish, i.e., $\m{A}/\m{\Gamma}_{t} = \m{0}$ and $\m{B}/\m{\Gamma}_{t} = \m{0}$.
\item Conversely, any constant-speed geodesic connecting $\m{A}$ and $\m{B}$ is of the form $\m{\Gamma}^{\m{N}_{12}}$ for some valid choice of $\m{N}_{12}$ in \cref{thm:Kanto:char:inf}.
\end{enumerate}
\end{theorem}

\begin{proof}[Proof of \cref{thm:kanto:geo:inf}] \quad
\begin{enumerate}[leftmargin = *]
\item Same to the proof of \cref{thm:kanto:geo}.
\item Let $\m{\Gamma}_{t} = \m{G}_{t} \m{G}_{t}^{*} \in \c{K}_{t}(\m{A}, \m{B})$, where $\m{G}_{t} = (1-t) \m{A}^{1/2} + t \m{M}$ with $\m{M} \in \s{G}(\m{B}), \ \m{A}^{1/2} \m{M} \succeq \m{0}$. To show $\m{A}/\m{\Gamma}_{t} = \m{0}$, it suffices to show
\begin{align*}
    \c{R}(\m{A}^{1/2}) \cap \c{N}(\m{\Gamma}_{t}) = \c{R}(\m{A}^{1/2}) \cap \c{N}(\m{G}_{t}^{*}) = \{ \m{0} \},
\end{align*}
due to \cref{lem:Schur:op}. Let $\m{c} \in \c{R}(\m{A}^{1/2}) \cap \c{N}(\m{G}_{t}^{*})$. Then, there exists some $\m{d} \in \c{H}$ such that $\m{c} = \m{A}^{1/2} \m{d}$ and
\begin{align*}
    0 = \m{d}^{*} \m{G}_{t}^{*} \m{c} = \m{d}^{*} \m{G}_{t}^{*} \m{G}_{1} \m{d} = (1-t) \m{d}^{*} [\m{A}^{1/2} \m{M}] \m{d} + t \m{c}^{*} \m{c} \ge t \|\m{c}\|^{2} \quad \Rightarrow \quad \m{c} = \m{0}.
\end{align*}
Consequently, we get $\m{A}/\m{\Gamma}_{t} = \m{0}$. Repeating this with $\m{B}$, we get $\m{B}/\m{\Gamma}_{t} = \m{0}$.
\item Let $\tilde{\m{\Gamma}}: [0, 1] \to \c{B}_{1}^{+}(\c{H})$ be a constant-speed geodesic connecting $\m{A}$ and $\m{B}$. Fix $t \in (0, 1)$. By the Cauchy-Schwarz inequality, we get
\begin{align*}
    (1-t) \c{W}_{2}^{2}(\m{A}, \m{C}) + t \c{W}_{2}^{2}(\m{B}, \m{C}) \ge 
    \left(\frac{1}{1-t} + \frac{1}{t} \right)^{-1} (\c{W}_{2}(\m{A}, \m{C}) + \c{W}_{2}(\m{B}, \m{C}))^{2} \ge
    t (1-t) \c{W}_{2}^{2}(\m{A}, \m{B}).
\end{align*}
From the proof of \cref{prop:Kanto:geo:inf}, the equality holds if and only if $\m{C} \in \c{K}_{t}(\m{A}, \m{B})$. This shows that $\tilde{\m{\Gamma}}_{t} \in \c{K}_{t}(\m{A}, \m{B})$, i.e., $\tilde{\m{\Gamma}}_{t} = \m{\Gamma}_{t}^{\m{N}_{12}}$ for some $\m{N}_{12}$ given in \cref{thm:Kanto:char:inf}. Moving forward, from part 2 and 
\cref{thm:spd:revise:inf}, $\m{\Gamma}^{\m{N}_{12}}: [t, 0] \to \c{B}_{1}^{+}(\c{H})$ (reversed time) is the unique constant-speed geodesic connecting $\m{\Gamma}_{t}^{\m{N}_{12}}$ and $\m{A}$. Similarly, $\m{\Gamma}^{\m{N}_{12}}: [t, 1] \to \c{B}_{1}^{+}(\c{H})$ is the unique constant-speed geodesic connecting $\m{\Gamma}_{t}^{\m{N}_{12}}$ and $\m{B}$. 
\end{enumerate}
\end{proof}

Analogous to Villani's quote \cite{villani2008optimal} \emph{``a geodesic in the space of laws is the law of a geodesic,''} our findings for Gaussian measures can be summarized as: \emph{``A geodesic in the space of Green operators is the Green operator of a geodesic''}. Finally, we establish an infinite-dimensional counterpart to \cref{thm:ext:geo}. While we show that every McCann interpolant is an extreme point of the set of barycenters, the converse remains an open question, especially due to (Case 2) in the proof of \cref{thm:ext:geo}.

\begin{proposition}[Extreme Geodesics]\label{prop:ext:geo:inf}
Let $\m{A}, \m{B} \in \c{B}_{1}^{+}(\c{H})$ be covariance operators with $\m{A} \leadsto \m{B}$. The set of barycenters $\c{K}_{t}(\m{A}, \m{B})$, which is a closed convex set, has the following properties:
\begin{enumerate}[leftmargin = *]
\item If $\m{\Gamma}_{t}$ is a McCann interpolant at time $t$, then it is an extreme point of $\c{K}_{t}(\m{A}, \m{B})$.
\item For any barycenter $\m{\Gamma}_{t}^{\m{N}_{12}} \in \c{K}_{t}(\m{A}, \m{B})$, we have $\m{\Gamma}_{t}^{\m{N}_{12}}/\m{A} = t^{2} (\m{B}/\m{A} - \m{N}_{12}^{*} \m{N}_{12})$.
\item $\m{\Gamma}_{t}$ is a McCann interpolant at time $t$ if and only if $\m{\Gamma}_{t}/\m{A} = \m{0}$.
\end{enumerate}
\end{proposition}

\begin{proof}[Proof of \cref{prop:ext:geo:inf}] \quad
\begin{enumerate}[leftmargin = *]
\item Repeating the proof of \cref{thm:ext:geo},
\begin{align*}
    \c{A} := \left\{\m{N}_{12}  : \c{H}_{2} \to \c{H}_{1} \mid \c{R}(\m{N}_{12}) \subset \c{N}(\m{B}_{11}^{1/2} \m{A}_{11}^{1/2}), \ \m{N}_{12}^{*} \m{N}_{12} \preceq \m{B}/\m{A} \right\},
\end{align*}
is a closed convex set. Then, any $\m{N}_{12}  : \c{H}_{2} \to \c{H}_{1}$ with $\c{R}(\m{N}_{12}) \subset \c{N}(\m{B}_{11}^{1/2} \m{A}_{11}^{1/2}), \ \m{N}_{12}^{*} \m{N}_{12} = \m{B}/\m{A}$ is an extreme point of $\c{A}$ since the equality in \eqref{eq:displ:cvx:Schur} holds if and only if $\m{P}_{0} = \m{P}_{1} \in \c{E}$.
Then, from \cref{rmk:linear:kanto}, it follows that $\c{K}_{t}(\m{A}, \m{B})$ is the closed convex set. It also follows that $\m{\Gamma}_{t} \in \c{K}_{t}(\m{A}, \m{B})$ is an extreme point if and only if $\m{N}_{12} \in \c{A}$ is an extreme point, hence a McCann interpolant at time $t$ is an extreme point due to \cref{thm:Kanto:char:inf}.

\item Recall $\m{T}_{11} = \m{A}_{11}^{\dagger/2} (\m{A}_{11}^{1/2} \m{B}_{11} \m{A}_{11}^{1/2})^{1/2} \m{A}_{11}^{\dagger/2}: \c{R}(\m{A}_{11}^{1/2}) \subset \c{H}_{1} \rightarrow \c{H}_{1}$ from \cref{thm:sol:sym:inf}, which is an s.p.d. operator due to \cref{thm:sol:injec}. Using the block decomposition as in the proof of \cref{thm:Kanto:char:inf}, we obtain
\begin{align*}
    (\m{\Gamma}_{t})_{11} = \m{X}_{11} \m{X}_{11}^{*}, \quad
    (\m{\Gamma}_{t})_{12} &= \m{X}_{11} [t \m{M}_{21}]^{*}, \quad
    (\m{\Gamma}_{t})_{22} = t^{2} \m{B}_{22} = t^{2}[\m{M}_{22} \m{M}_{22}^{*} + \m{M}_{21} \m{M}_{21}^{*}],
\end{align*}
where
\begin{align*}
    \m{X}_{11} := [(1-t) \m{A}_{11}^{1/2} + t \m{M}_{11}] = [(1-t) \m{I}_{11} + t \m{T}_{11}] \m{A}_{11}^{1/2}. 
\end{align*}
Note that $\m{X}_{11}$ is injective because $\m{X}_{11} f = 0$ implies
\begin{align*}
    0 = \innpr{\m{A}_{11}^{1/2} f}{\m{X}_{11} f} \ge (1-t) \|\m{A}_{11}^{1/2} f \|^{2} \quad \Rightarrow \quad \m{A}_{11}^{1/2} f = 0 \quad \Rightarrow \quad f = 0.
\end{align*}

To ease the notation, let us denote 
\begin{align*}
    \m{X}_{11} := [(1-t) \m{G}_{11} + t \m{M}_{11}] = [(1-t) \m{G}_{11} + t (\m{G}_{11}^{*})^{\dagger} (\m{G}_{11}^{*} \m{B}_{11} \m{G}_{11})^{1/2} \m{G}_{11}^{-1}] \m{G}_{11}, 
\end{align*}
which is injective. Then, the $\m{A}$-Schur complement of $\m{\Gamma}_{t}$ is given by
\begin{align*}
    \m{\Gamma}_{t}/\m{A} &= (\m{\Gamma}_{t})_{22} - \overline{(\m{\Gamma}_{t})_{21} (\m{\Gamma}_{t})_{11}^{\dagger} (\m{\Gamma}_{t})_{12}} \\
    &= t^{2} \m{M}_{22} \m{M}_{22}^{*} + t^{2} \m{M}_{21} \overline{\left[ \m{I}_{11} - \m{X}_{11}^{*} (\m{X}_{11} \m{X}_{11}^{*})^{\dagger} \m{X}_{11}  \right]} \m{M}_{21}^{*} \\
    &= t^{2} \m{M}_{22} \m{M}_{22}^{*} + t^{2} \m{M}_{21} \m{\Pi}_{\c{N}(\m{X}_{11})} \m{M}_{21}^{*} = t^{2} \m{M}_{22} = t^{2} [\m{B}/\m{A} - \m{N}_{12}^{*} \m{N}_{12}].
\end{align*}

\item This is an immediate consequence of part 2 and \cref{thm:Kanto:char:inf}.
\end{enumerate}
\end{proof}

\subsection{Convex Analysis Perspective}
We revisit our results through the lens of convex analysis in
Kantorovich dual formulation. Let $\m{A}, \m{B} \in \c{B}_{1}^{+}(\c{H})$ and $\m{G} \in \c{G}(\m{A})$, and $\m{G} \in \c{G}(\m{A}), \m{M} \in \c{G}(\m{B})$ with $\m{G}^{*} \m{M} \succeq \m{0}$. This entails that $\m{G}^{*} \m{M} = \m{M}^{*} \m{G}$ so that $\innpr{\m{G} z}{\m{M} \tilde{z}} = \innpr{\m{G} \tilde{z}}{\m{M} z}$ for any $z, \tilde{z} \in \c{H}$. Also, the displacement convexity reads
\begin{align}\label{eq:disp:cvx}
    &(1-t) \innpr{\m{G} z}{\m{M} z} + t \innpr{\m{G} \tilde{z}}{\m{M} \tilde{z}} - \innpr{\m{G} ((1-t) z + t \tilde{z})}{\m{M} ((1-t) z + t \tilde{z})} \nonumber \\
    &= t(1-t) \innpr{\m{G} (z-\tilde{z})}{\m{M} (z-\tilde{z})} \ge 0.
\end{align}
Define a closed linear subspace in $\c{H} \times \c{H}$ by $\c{S} := \overline{\{(\m{G} z, \m{M} z) : z \in \c{H} \}}^{\c{H} \times \c{H}}$.
In other words, $(x, y) \in \c{S}$ if and only if there exists a sequence $\{z_{n} \in \c{H}\}$ such that $\m{G} z_{n} \to x$ and $\m{M} z_{n} \to y$. We remark that $\c{S}$ is a graph if and only if there exists a closable optimal pre-pushforward $\m{T}$ such that $\m{M} = \m{T} \m{G}$, and in this case, $\c{S} = \c{G}(\overline{\m{T}})$. Moving forward, we consider the sections of $\c{S}$, which are closed affine subspaces by the continuity of projection:
\begin{align*}
    \c{S}_{\m{M}}(x) := \{y \in \c{H} : (x, y) \in \c{S}\}, \quad \c{S}_{\m{G}}(y) := \{x \in \c{H} : (x, y) \in \c{S}\}.
\end{align*}

\begin{lemma}\label{lem:innpr:const}
For any $x \in \c{H}$ with $\c{S}_{\m{M}}(x) \ne \emptyset$, $y \in \c{S}_{\m{M}}(x) \mapsto \innpr{x}{y}/2 \in \b{R}$
is a constant function.
\end{lemma}
\begin{proof}[Proof of \cref{lem:innpr:const}]
Let $y, \tilde{y} \in \c{S}(x)$. Then, there exist sequences $\{z_{n} \in \c{H}\}$ and $\{\tilde{z}_{n} \in \c{H}\}$ such that 
\begin{align*}
    (\m{G} z_{n}, \m{M} z_{n}) \to (x, y), \quad (\m{G} \tilde{z}_{n}, \m{M} \tilde{z}_{n}) \to (x, \tilde{y}).
\end{align*}
Accordingly,
\begin{align*}
    \innpr{x}{y} - \innpr{x}{\tilde{y}} &= \lim_{n \to \infty} [\innpr{\m{G} z_{n}}{\m{M} z_{n}} - \innpr{\m{G} \tilde{z}_{n}}{\m{M} \tilde{z}_{n}}]
    \\
    &= \lim_{n \to \infty} [\innpr{\m{G} (z_{n} - \tilde{z}_{n})}{\m{M} z_{n}} + \innpr{\m{G} (z_{n} - \tilde{z}_{n})}{\m{M} \tilde{z}_{n}}] = \innpr{0}{y} + \innpr{0}{\tilde{y}} = 0.
\end{align*}
\end{proof}

From the lemma above, the following functions $\phi, \psi: \c{H} \to [0, +\infty]$ are well-defined:
\begin{align*}
    \phi(x) := 
    \begin{cases}
        \innpr{x}{y}/2 &, \quad y \in \c{S}_{\m{M}}(x) \ne \emptyset, \\
        +\infty &, \quad \c{S}_{\m{M}}(x) = \emptyset. 
    \end{cases}, \quad 
    \psi(y) := 
    \begin{cases}
        \innpr{x}{y}/2 &, \quad x \in \c{S}_{\m{G}}(y) \ne \emptyset, \\
        +\infty &, \quad \c{S}_{\m{G}}(y) = \emptyset. 
    \end{cases}.
\end{align*}

\begin{proposition}\label{prop:subdiff}
The function $\phi: \c{H} \to [0, +\infty]$ (resp. $\psi: \c{H} \to [0, +\infty]$) is convex. Their subdifferentials at $x \in \c{H}$ (resp. $y \in \c{H}$) satisfies $\partial \phi(x) \supset \c{S}_{\m{M}}(x)$ (resp. $\partial \psi(y) \supset \c{S}_{\m{G}}(y)$).
\end{proposition}
\begin{proof}[Proof of \cref{prop:subdiff}]
To show the convexity of $\phi$, we fix $x, \tilde{x} \in \c{H}$ and $t \in (0, 1)$, and demonstrate
\begin{align*}
    (1-t) \phi(x) + t \phi(\tilde{x}) \ge \phi((1-t)x + t \tilde{x}).
\end{align*}
We may assume $y \in \c{S}_{\m{M}}(x) \ne \emptyset$ and $\tilde{y} \in \c{S}_{\m{M}}(\tilde{x}) \ne \emptyset$, otherwise trivial. Then, there exist sequences $\{z_{n} \in \c{H}\}$ and $\{\tilde{z}_{n} \in \c{H}\}$ such that 
\begin{align*}
    (\m{G} z_{n}, \m{M} z_{n}) \to (x, y), \quad (\m{G} \tilde{z}_{n}, \m{M} \tilde{z}_{n}) \to (\tilde{x}, \tilde{y}),
\end{align*}
and thus $(1-t) y + t \tilde{y} \in \c{S}((1-t)x + t \tilde{x})$. Consequently, \cref{lem:innpr:const} and \eqref{eq:disp:cvx} lead to
\begin{align*}
    \phi((1-t)x + t \tilde{x}) &= \lim_{n \to \infty} \frac{\innpr{\m{G}((1-t)z_{n} + t \tilde{z}_{n})}{\m{M}((1-t)z_{n} + t \tilde{z}_{n})}}{2} \\
    &\le \limsup_{n \to \infty} \left[ (1-t) \frac{\innpr{\m{G} z_{n}}{\m{M} z_{n}}}{2} + t \frac{\innpr{\m{G} \tilde{z}_{n}}{\m{M} \tilde{z}_{n}}}{2} \right] = (1-t) \phi(x) + t \phi(\tilde{x}).
\end{align*}
We now show that $\c{S}_{\m{M}}(x) \subset \partial \phi(x)$. It suffices to show that for any $(x, y) \in \c{S}$, it holds that $\phi(x+\delta_{x}) \ge \phi(x) + \innpr{y}{\delta_{x}}$ for any $\delta_{x} \in \c{H}$. We may assume that $\c{S}_{\m{M}}(\delta_{x}) \ne \emptyset$, otherwise trivial. Choose $\delta_{y} \in \c{S}_{\m{M}}(\delta_{x})$. Then, there exist sequences $\{z_{n} \in \c{H}\}$ and $\{\tilde{z}_{n} \in \c{H}\}$ such that 
\begin{align*}
    (\m{G} z_{n}, \m{M} z_{n}) \to (x, y), \quad (\m{G} \tilde{z}_{n}, \m{M} \tilde{z}_{n}) \to (\delta_{x}, \delta_{y}).
\end{align*}
Consequently, $(x+\delta_{x}, y+\delta_{y}) \in \c{S}$, thus \cref{lem:innpr:const} and \eqref{eq:disp:cvx} lead to
\begin{align*}
    \phi(x+\delta_{x}) - \phi(x) - \innpr{y}{\delta_{x}} &= \lim_{n \to \infty} \frac{\innpr{\m{G} (z_{n} + \tilde{z}_{n})}{\m{M} (z_{n} + \tilde{z}_{n})}}{2} - \frac{\innpr{\m{G} z_{n}}{\m{M} z_{n}}}{2} - \innpr{\m{G} z_{n}}{\m{M} \tilde{z}_{n}} \\
    &= \lim_{n \to \infty} \frac{\innpr{\m{G} \tilde{z}_{n}}{\m{M} \tilde{z}_{n}}}{2} = \phi(\delta_{x}) \ge 0.
\end{align*}
\end{proof}

As a direct consequence, we get $\phi(x) + \psi(y) = \innpr{x}{y}$ whenever $(x, y) \in \c{S}$. If we consider the Gaussian coupling $\pi$ of $\mu = N(\m{0}, \m{A})$ and $\nu = N(\m{0}, \m{B})$ with its covariance
\begin{align*}
    \begin{pmatrix}
    \m{A} & \m{G} \m{M}^{*} \\
    \m{M} \m{G}^{*} & \m{B}
    \end{pmatrix} = 
    \begin{pmatrix}
    \m{G} \\
    \m{M}
    \end{pmatrix}
    \begin{pmatrix}
    \m{G}^{*} & \m{M}^{*}
    \end{pmatrix}
    \in \c{B}_{1}^{+}(\c{H} \times \c{H}),
\end{align*}
then $\text{supp}(\pi) = \c{S}$ due to \cref{lem:a.s.:dense}, and the classical results from Kantorovich dual formulation again confrims that $\pi$ is indeed optimal \cite{villani2008optimal, ambrosio2008gradient}. However, the Fenchel's strong duality \cite{rockafellar1970convex} between $\phi$ and $\psi$ does not hold in general.

\begin{theorem}\label{thm:legen:transf}
The convex conjugate of the function $\phi: \c{H} \to [0, +\infty]$ is given by
\begin{align*}
    \phi^{c}(y) = \sup_{x \in \c{H}} [\innpr{x}{y} - \phi(x)] = 
    \begin{cases}
        \|(\m{G}^{*} \m{M})^{\dagger/2} \m{G}^{*} y \|^{2}/2 &, \quad \m{G}^{*} y \in \c{R}[(\m{G}^{*} \m{M})^{1/2}], \\
        +\infty &, \quad \m{G}^{*} y \notin \c{R}[(\m{G}^{*} \m{M})^{1/2}]. 
    \end{cases}
\end{align*}
\end{theorem}
\begin{proof}[Proof of \cref{thm:legen:transf}]
We first show that $\phi^{c}(y) < +\infty$ only if $\m{G}^{*} y \in \c{R}[(\m{G}^{*} \m{M})^{1/2}]$ by contradiction. 
\begin{itemize}
\item [(Step 1)] We claim $\m{G}^{*} y \in \overline{\c{R}[(\m{G}^{*} \m{M})^{1/2}]}$. If not, let $z \ne 0$ be the projection of $\m{G}^{*} y$ onto $\overline{\c{R}[(\m{G}^{*} \m{M})^{1/2}]}$. Then, for any $t \in \b{R}$, $t(\m{G}z, \m{M}z) \in \c{S}$, thus \cref{lem:innpr:const} yields
\begin{align*}
    \phi^{c}(y) \ge \sup_{t \in \b{R}} [t \innpr{\m{G} z}{y} - \phi(\m{G} z)] = \sup_{t \in \b{R}} [t \innpr{z}{\m{G}^{*} y} - \|(\m{G}^{*} \m{M})^{1/2}z\|^{2}/2] = \sup_{t \in \b{R}} t \|z\|^{2} = +\infty,
\end{align*}
which is a contradiction.

\item [(Step 2)] From (Step 1), the following linear functional is well-defined:
\begin{align*}
    l_{y} : \c{R}[(\m{G}^{*} \m{M})^{1/2}] \to \b{R}, (\m{G}^{*} \m{M})^{1/2} z \mapsto \innpr{\m{G} z}{y}
\end{align*}
By our assumption, we also have for any $0 \ne \omega \in \c{R}[(\m{G}^{*} \m{M})^{1/2}]$ that
\begin{align*}
    + \infty > \phi^{c}(y) = \sup_{z \in \c{H}} [\innpr{\m{G} z}{y} - \phi(\m{G} z)] \ge \sup_{t \in \b{R}} \left[l_{y}(t \omega) - \frac{\|t \omega\|^{2}}{2} \right] = \frac{l_{y}(\omega)^{2}}{2 \|\omega\|^{2}},
\end{align*}
hence $|l_{y}(\omega)| \le \sqrt{2 \phi^{c}(y)} \|\omega\|$, i.e. $l_{y}$ is bounded. Due to the Riesz representation theorem, there exists a unique $z_{y} \in \overline{\c{R}[(\m{G}^{*} \m{M})^{1/2}]}$ such that $l_{y}(\omega) = \innpr{z_{y}}{\omega}$ for any $\omega \in \c{R}[(\m{G}^{*} \m{M})^{1/2}]$. Consequently, for any $z \in \c{H}$, we have
\begin{align*}
    \innpr{z}{\m{G}^{*} y} = l_{y}((\m{G}^{*} \m{M})^{1/2}z) = \innpr{z_{y}}{(\m{G}^{*} \m{M})^{1/2}z} = \innpr{z}{(\m{G}^{*} \m{M})^{1/2} z_{y}},
\end{align*}
i.e., $\m{G}^{*} y = (\m{G}^{*} \m{M})^{1/2} z_{y} \in \c{R}[(\m{G}^{*} \m{M})^{1/2}]$.
\end{itemize}

Conversely, assume $\m{G}^{*} y \in \c{R}[(\m{G}^{*} \m{M})^{1/2}]$, and define $z_{y} := (\m{G}^{*} \m{M})^{\dagger/2} \m{G}^{*} y \in \overline{\c{R}[(\m{G}^{*} \m{M})^{1/2}]}$. We claim that $\phi^{c}(y)= \|z_{y}\|^{2}/2$.
\begin{itemize}[leftmargin = *]
\item [($\le$)] Fix any $x \in \c{H}$ that has a non-trivial effective domain, i.e. $\c{S}_{\m{M}}(x) \ne \emptyset$. Then, there exist $y_{x} \in \c{S}_{\m{M}}(x)$ and a sequence $\{z_{n} \in \c{H}\}$ such that $(\m{G} z_{n}, \m{M} z_{n}) \to (x, y_{x})$. Hence,
\begin{align*}
    \innpr{x}{y} - \phi(x) = \lim_{n \to \infty} \innpr{\m{G} z_{n}}{y} - \frac{\innpr{\m{G} z_{n}}{\m{M} z_{n}}}{2} = \lim_{n \to \infty} \innpr{(\m{G}^{*} \m{M})^{1/2} z_{n}}{z_{y}} - \frac{\|(\m{G}^{*} \m{M})^{1/2} z_{n}\|^{2}}{2}
    \le \frac{\|z_{y}\|^{2}}{2}.
\end{align*}

\item [($\ge$)] Choose a sequence $\{z_{n} \in \c{H} \}$ such that $\lim_{n \to \infty} (\m{G}^{*} \m{M})^{1/2} z_{n} = z_{y}$. Since $(\m{G} z_{n}, \m{M} z_{n}) \in \c{S}$, we have
\begin{align*}
    \phi^{c}(y) &\ge \limsup_{n \to \infty} [\innpr{\m{G} z_{n}}{y} - \phi(\m{G} z_{n})] = \limsup_{n \to \infty} \innpr{(\m{G}^{*} \m{M})^{1/2} z_{n}}{z_{y}} - \frac{\|(\m{G}^{*} \m{M})^{1/2} z_{n}\|^{2}}{2} \\
    &= \frac{\|z_{y}\|^{2}}{2} - \liminf_{n \to \infty} \frac{\|(\m{G}^{*} \m{M})^{1/2} z_{n} - z_{y}\|^{2}}{2} = \frac{\|z_{y}\|^{2}}{2}.
\end{align*}
\end{itemize}
\end{proof}

\begin{example}[Particle Movement]
We consider $\m{A} = \diag(1, 0)$ and $\m{B} = \diag(0, 1)$ from \cref{ex:rk:goes:up}. From \cref{ex:midpoint:simple,thm:kanto:geo}, any geodesics connecting $\m{A}$ and $\m{B}$ can be parametrized by $s \in [-1, +1]$: 
\begin{align*}
    \m{\Gamma}^{s}: [0, 1] \to \b{R}^{2 \times 2}, \, t \mapsto \m{\Gamma}_{t}^{s} = \begin{pmatrix}
        (1-t)^{2} & t(1-t)s \\
        t(1-t)s & t^{2}
    \end{pmatrix},
\end{align*}
and the extremals $s = \pm 1$ correspond to Monge geodesics with the OT matrices $\m{T}_{\m{A} \to \m{B}}^{\pm}$ (or $\m{T}_{\m{B} \to \m{A}}^{\pm}$) in \cref{ex:rk:goes:up}. We remark that when $|s| < 1$, i.e., Kantorovich geodesics, there are infinite choices of corresponding optimal couplings: for instance, when $s=0$,
\begin{itemize}
\item Independent Coupling: $\pi = \mu \otimes \nu$ with
\begin{align*}
    \text{supp}(\pi) = \c{S} = \left\{ \begin{pmatrix}
    x_{1} \\ x_{2}
    \end{pmatrix} : x_{2} = 0 \right\} \times \left\{ \begin{pmatrix}
    y_{1} \\ y_{2}
    \end{pmatrix} : y_{1} =0 \right\} \subset \b{R}^{2} \times \b{R}^{2}.
\end{align*}
\item Gaussian Mixture: $\pi = \frac{1}{2} (\m{I}, \m{T}_{\m{A} \to \m{B}}^{+}) \sharp \mu + \frac{1}{2} (\m{I}, \m{T}_{\m{A} \to \m{B}}^{-}) \sharp \mu$
\item Convex combinations of above couplings.
\end{itemize}

However, regardless of the choice of $s$, the corresponding Green matrices always yield $\m{G}^{*} \m{M} = \m{0}$ due to \cref{cor:ortho:Green's}, hence they all have the same corresponding convex functions:
\begin{align*}
    \phi \begin{pmatrix}
    x_{1} \\ x_{2}
    \end{pmatrix} = 
    \begin{cases}
        0 &, \quad x_{2} = 0, \\
        +\infty &, \quad x_{2} \ne 0. 
    \end{cases}, \quad 
    \psi \begin{pmatrix}
    y_{1} \\ y_{2}
    \end{pmatrix} = 
    \begin{cases}
        0 &, \quad y_{1} = 0, \\
        +\infty &, \quad y_{1} \ne 0. 
    \end{cases}.
\end{align*}
In this case, they are convex conjugate to one another: $\phi(x) + \psi(y) = \innpr{x}{y}$ if and only if $(x, y) \in \c{S}$.
\end{example}

\subsection{Comments on Barycenters}
The nullity of an operator is not necessarily same to that of the adjoint when $\dim \c{H} = \infty$. Consequently, unlike in the finite-dimensional case, the set of all Green's operators $\c{G}(\m{A})$ cannot be generated by applying unitary transforms to a single Green's operator, while it can be by partially isometric transform (\cref{lem:Green's:part:iso}). This fact invalidates \cref{lem:opt:algn}, a cornerstone of the barycenter theory developed in \cref{ssec:rand:Green,ssec:multi:trans}. Its infinite-dimensional counterpart, presented below, highlights the crucial differences:

\begin{theorem}\label{lem:opt:algn:inf-dim}
Let $\m{A}_{1}, \m{A}_{2} \in \c{B}_{1}^{+}(\c{H})$ be covariance operators. For any $\m{G}_{1} \in \c{G}(\m{A}_{1})$ and $\m{G}_{2} \in \c{G}(\m{A}_{2})$, we have 
\begin{align*}
    \tr [(\m{A}_{1}^{1/2} \m{A}_{2} \m{A}_{1}^{1/2})^{1/2}] = \tr [(\m{G}_{1}^{*} \m{A}_{2} \m{G}_{1})^{1/2}] =
    \sup \{\tr (\m{G}_{1}^{*} \m{G}_{2} \m{U}) : \m{U}^{*} \m{U} = \m{U} \m{U}^{*} = \m{I} \}.
\end{align*}
The supremum is achieved as a maximum if and only if $\dim \c{N}(\m{G}_{1}^{*} \m{G}_{2}) = \dim \c{N}(\m{G}_{2}^{*} \m{G}_{1})$.
\end{theorem}

\begin{proof}[Proof of \cref{lem:opt:algn:inf-dim}] \quad
\begin{enumerate}[leftmargin=*]
\item [(Step 1)] We show that $\tr [(\m{G}_{1}^{*} \m{A}_{2} \m{G}_{1})^{1/2}]$ is independent of the choice of $\m{G}_{1} \in \c{G}(\m{A}_{1})$. Let $\m{G}_{1} = \m{A}_{1}^{1/2} \m{U}$ be the polar decomposition, where $\m{U}$ is a unique partial isometry with initial space $\c{N}(\m{G}_{1})^{\perp}$ and final space $\c{H}_{1} = \c{N}(\m{A}_{1}^{1/2})^{\perp}$. Then, $\m{U} \m{U}^{*} = \m{\Pi}_{1}$, hence
\begin{align*}
    [\m{U}^{*} (\m{A}_{1}^{1/2} \m{A}_{2} \m{A}_{1}^{1/2})^{1/2} \m{U}]^{2} &= \m{U}^{*} (\m{A}_{1}^{1/2} \m{A}_{2} \m{A}_{1}^{1/2})^{1/2} \m{\Pi}_{1} (\m{A}_{1}^{1/2} \m{A}_{2} \m{A}_{1}^{1/2})^{1/2} \m{U} = \m{U}^{*} (\m{A}_{1}^{1/2} \m{A}_{2} \m{A}_{1}^{1/2}) \m{U} \\
    &= \m{G}_{1}^{*} \m{A}_{2} \m{G}_{1},
\end{align*}
hence $(\m{G}_{1}^{*} \m{A}_{2} \m{G}_{1})^{1/2}= \m{U}^{*} (\m{A}_{1}^{1/2} \m{A}_{2} \m{A}_{1}^{1/2})^{1/2} \m{U}$. Consequently, 
\begin{align*}
    \tr [(\m{G}_{1}^{*} \m{A}_{2} \m{G}_{1})^{1/2}] = \tr [\m{U} \m{U}^{*} (\m{A}_{1}^{1/2} \m{A}_{2} \m{A}_{1}^{1/2})^{1/2}] = \tr [\m{\Pi}_{1} (\m{A}_{1}^{1/2} \m{A}_{2} \m{A}_{1}^{1/2})^{1/2}] =
    \tr [(\m{A}_{1}^{1/2} \m{A}_{2} \m{A}_{1}^{1/2})^{1/2}].
\end{align*}

\item [(Step 2)] We claim that $\sup \{\tr (\m{G}_{1}^{*} \m{G}_{2} \m{U}) : \m{U}^{*} \m{U} = \m{U} \m{U}^{*} = \m{I} \} \ge \tr [(\m{G}_{1}^{*} \m{A}_{2} \m{G}_{1})^{1/2}]$ (the other direction is trivial from \cref{lem:Neumann:eq}), using finite rank approximation. Note that $\m{G}_{1}^{*} \m{G}_{2} \in \c{B}_{1}(\c{H})$ since $\m{G}_{1}, \m{G}_{2} \in \c{B}_{2}(\c{H})$. Consider its SVD decomposition:
\begin{align*}
    \m{G}_{1}^{*} \m{G}_{2} = \sum_{k=1}^{\infty} \sigma_{k} e_{k} \otimes f_{k},
\end{align*}
where $\sigma_{1} \ge \sigma_{2} \ge \dots >0$ are positively-valued non-increasing sequence of singular values, $\{f_{k} : k \in \b{N}\}$ is an ONB of $\overline{\c{R}(\m{G}_{1}^{*} \m{G}_{2})}$, and $\{e_{k} : k \in \b{N}\}$ is an ONB of $\overline{\c{R}(\m{G}_{2}^{*} \m{G}_{1})}$. For each $n \in \b{N}$, consider a unitary operator $\m{U} \in \c{B}_{\infty}(\c{H})$ such that $\m{U}_{n} e_{k} = f_{k}$ for $k = 1, 2, \cdots, n$. Then, $\m{G}_{1}^{*} \m{G}_{2} \m{U}_{n} \in \c{B}_{1}(\c{H})$ so
\begin{align*}
    \sum_{k=n+1}^{\infty} |\innpr{\m{G}_{1}^{*} \m{G}_{2} \m{U}_{n} e_{k}}{e_{k}}| \le \sum_{k=n+1}^{\infty} \sigma_{k+1} |\innpr{\m{U}_{n} e_{k}}{e_{k}}| \le \sum_{k=n+1}^{\infty} \sigma_{k+1} \to 0, \quad \text{ as } n \to \infty,
\end{align*}
which results in
\begin{align*}
    \tr (\m{G}_{1}^{*} \m{G}_{2} \m{U}_{n}) &= \sum_{k=1}^{\infty} \innpr{\m{G}_{1}^{*} \m{G}_{2} \m{U}_{n} e_{k}}{e_{k}} = \sum_{k=1}^{n} \sigma_{k} + \sum_{k=n+1}^{\infty} \innpr{\m{G}_{1}^{*} \m{G}_{2} \m{U}_{n} e_{k}}{e_{k}} \\
    &\to \sum_{k=1}^{\infty} \sigma_{k} = \vertii{\m{G}_{1}^{*} \m{G}_{2}}_{1} = \tr [(\m{G}_{1}^{*} \m{A}_{2} \m{G}_{1})^{1/2}], \quad \text{ as } n \to \infty.
\end{align*}

\item [(Step 3)] Given a unitary operator $\m{U} \in \c{B}_{\infty}(\c{H})$, we have $\tr (\m{G}_{1}^{*} \m{G}_{2} \m{U}) = \tr [(\m{G}_{1}^{*} \m{A}_{2} \m{G}_{1})^{1/2}]$ if and only if $\m{G}_{1}^{*} \m{G}_{2} \m{U} = (\m{G}_{1}^{*} \m{A}_{2} \m{G}_{1})^{1/2}$, which is equivalent to $\m{U} e_{k} = f_{k}$ for all $k \in \b{N}$. This happens if and only if $\dim \c{N}(\m{G}_{1}^{*} \m{G}_{2}) = \dim \c{N}(\m{G}_{2}^{*} \m{G}_{1})$ \cite{conway2019course}.
\end{enumerate}
\end{proof}

This result has direct implications for the Procrustes formulation of the Wasserstein distance. As correctly characterized in \cite{pigoli2014distances, masarotto2019procrustes}, the Procrustes distance must be defined via an infimum:
\begin{align}\label{eq:inf:not:min}
    \c{W}_{2}^{2}(\m{A}_{1}, \m{A}_{2}) = \inf \{\vertii{\m{A}_{1}^{1/2} - \m{A}_{2}^{1/2} \m{U}}_{2}^{2} : \m{U}^{*} \m{U} = \m{U} \m{U}^{*} = \m{I} \},
\end{align}
However, contrary to a popular misconception, this infimum is not always achieved as a minimum when $\dim \c{H} = \infty$. This is not a rare pathology; the following proposition demonstrates that for any injective covariance operator, one can always construct a second operator for which the minimum is not attained, although the reachability is satisfied.

\begin{proposition}\label{prop:inf:not:min}
Let $\dim \c{H} = \infty$. For any injective covariance operator $\m{A}_{1} \in \c{B}_{1}^{+}(\c{H})$, there exists another covariance operator $\m{A}_{2} \in \c{B}_{1}^{+}(\c{H})$ such that $\m{A}_{2}^{1/2} \m{A}_{1}^{1/2}$ is injective while $\m{A}_{1}^{1/2} \m{A}_{2}^{1/2}$ is not. In this case, the infimum in \eqref{eq:inf:not:min} cannot be attained by any unitary operator $\m{U}$.
\end{proposition}

\begin{proof}[Proof of \cref{prop:inf:not:min}]
Since $\m{A}_{1}^{1/2} \in \c{B}_{2}^{+}(\c{H})$ is a compact operator, it cannot be surjective due to the open mapping theorem \cite{conway2019course}. Choose $v \notin \c{R}(\m{A}_{1}^{1/2})$ with $\|v\| = 1$, and define
\begin{align*}
    \m{A}_{2} := \m{\Pi} \m{A}_{1}^{1/2} \m{\Pi} \m{A}_{1}^{1/2} \m{\Pi} \in \c{B}_{1}^{+}(\c{H}), \quad \m{\Pi} = \m{I} - v \otimes v.
\end{align*}
Note that $\m{\Pi}$ is a projection with $\c{N}(\m{\Pi}) = \spann \{v\}$ and $\m{A}_{2}^{1/2} = \m{\Pi} \m{A}_{1}^{1/2} \m{\Pi} \in \c{B}_{2}^{+}(\c{H})$. It is straightforward that $v \in \c{N}(\m{A}_{1}^{1/2} \m{A}_{2}^{1/2})$, hence $\m{A}_{1}^{1/2} \m{A}_{2}^{1/2}$ is not injective.

We claim that $\m{A}_{2}^{1/2} \m{A}_{1}^{1/2}$ is injective. Let $f \in \c{N}(\m{A}_{2}^{1/2} \m{A}_{1}^{1/2})$. Then,
\begin{equation*}
    \m{A}_{1}^{1/2} f \in \c{R}(\m{A}_{1}^{1/2}) \cap \c{N}(\m{A}_{2}^{1/2}) = \c{R}(\m{A}_{1}^{1/2}) \cap \spann \{v\} = \{0\}.
\end{equation*}
This shows that $\m{A}_{1}^{1/2} f = 0$, thus $f = 0$, i.e., $\m{A}_{2}^{1/2} \m{A}_{1}^{1/2}$ is injective. Consequently, 
\begin{align*}
    \c{W}_{2}^{2}(\m{A}_{1}, \m{A}_{2}) < \vertii{\m{A}_{1}^{1/2} - \m{A}_{2}^{1/2} \m{U}}_{2}^{2}
\end{align*}
for any unitary operator $\m{U}$ due to 
part 1 in \cref{lem:opt:algn:inf-dim}.
\end{proof}

If $\m{A}_{1} \leadsto \m{A}_{2}$ with an optimal pre-pushforward $\m{T}$ from $\m{A}_{1}$ to $\m{A}_{2}$, then for any $\m{G}_{1} \in \c{G}(\m{A}_{1})$, $\m{G}_{2} = \m{T} \m{G}_{1} \in \c{G}(\m{A}_{2})$ is a Green's operator satisfying $\m{G}_{1}^{*} \m{G}_{2} \succeq \m{0}$.




\section{Gaussian Processes}\label{sec:Gproc}
This section demonstrates how to find optimal pre-pushforwards between Gaussian processes using the theory of integral operators and their kernels.
Let $(E, \s{B}, \mu)$ be a measure space and $\c{H} := L_{2} (E, \mu)$, where $E$ is a compact set and $\mu$ is a finite measure. For a bivariate function $G \in L_{2}(E \times E, \mu \times \mu)$, we define the corresponding integral operator by
\begin{equation*}
    \s{K}_{G}: \c{H} \rightarrow \c{H}, \quad \s{K}_{G} f(s) := \int_{E} G(s, t) f(t) \rd \mu (t), \quad f \in \c{H} := L_{2} (E, \mu).
\end{equation*}
The bivariate function $G^{*}(s, t) = G(t, s)$ is called the \emph{adjoint kernel}, as $(\s{K}_{G})^{*} = \s{K}_{G^{*}}$. The composition of two integral operators corresponds to the box product of their kernels, defined as
\begin{equation*}
    G_{1} \square G_{2}: E \times E \rightarrow \b{R}, \quad (G_{1} \square G_{2})(s, t) := \int_{E} G_{1}(s, u) G_{2}(u, t) \rd \mu(u),
\end{equation*}
which satisfies $\s{K}_{G_{1} \square G_{2}} = \s{K}_{G_{1}} \s{K}_{G_{2}}$. In particular, $K = G \square G^{*}$ is an s.p.d. kernel. Conversely, Mercer's theorem \cite{mercer1909xvi} states that for any such continuous kernel, there exists a bivariate function $G$, called the \emph{Green's function} such that $K = G \square G^{*}$, which allows us to define the RKHS \cite{paulsen2016introduction} as
\begin{equation*}
    \b{H}(K) := \c{R}(\s{K}_{G}), \quad \|f\|_{\b{H}(K)} := \| \s{K}_{G}^{\dagger} f \|_{\c{H}}.
\end{equation*}
When $G$ is also chosen to be a continuous s.p.d. kernel, it is denoted by $G = K^{1/2}$ as $\s{K}_{G} = (\s{K}_{K})^{1/2}$. Unless $\b{H}(K)$ is finite-dimensional, it is \emph{not} a closed subspace in $\c{H}$ with respect to the $L_{2}$-norm.

To treat a Gaussian process $X$ on $E$ as a random element in $\c{H}$, we assume it is jointly measurable. For Gaussian processes, this property can be easily verified from the regularity of its covariance kernel $K$ via Kolmogorov's continuity theorem \cite{da2006introduction,kallenberg1997foundations,hairer2009introduction}. For such a process, the integral operator $\s{K}_{K}: \c{H} \rightarrow \c{H}$ serves as its covariance \emph{operator} \cite{hsing2015theoretical}.

Numerous covariance kernels of interest are generated by a linear differential operator $\s{L}$ with boundary conditions. The Green's function $G$ for such an operator is defined as the distributional solution to $\s{L} G = \delta$ in the weak sense, where $\delta$ is the Dirac delta function \cite{hall2013quantum, conway2019course}. This relationship implies that $\s{L}$ acts as the pseudoinverse of the integral operator $\s{K}_{G}$:
\begin{equation*}
    \s{L} \s{K}_{G} f(s) = \s{L} \int_{E} G(s, t) f(t) \rd \mu(t) = \int_{E} \delta(s, t) f(t) \rd \mu(t) = f(s).
\end{equation*}
Thus, we get $\s{L}: \b{H}(K) \rightarrow \c{H}$ where $K = G \square G^{*}$, and can be identified with $\s{L} = \s{K}_{G}^{\dagger}$. This perspective suggests that Green's functions, which are often non-symmetric, are more fundamental objects than the covariance kernels they generate.

\begin{example}[Brownian Motion and Bridge]\label{ex:BM2BB}
Let $\c{H} = L_{2}[0,1]$ and consider the differential operators
\begin{align*}
    \s{L}_{1} f = f', \, f(0) = 0, \qquad
    \s{L}_{2} f = f', \, f(0) = f(1) = 0.
\end{align*}
The corresponding Green's functions are the non-symmetric:
\begin{equation*}
    G_{1}(s, t) = I_{[0, s]}(t), \quad G_{2}(s, t) = I_{[0, s]}(t) - s, \quad 0 \le s, t \le 1,
\end{equation*}
These generate the well-known covariance kernels for Brownian motion (BM) and a Brownian bridge (BB), respectively: $K_{1}(s, t) = \min (s, t)$ and $K_{2}(s, t) = \min (s, t) -st$ for $0 \le s, t \le 1$.
Note that they are jointly measurable Gaussian, and their associated integral operators are given by
\begin{equation*}
    (\s{K}_{G_{1}} f) (s) = \int_{0}^{s}  f(t) \rd s, \quad (\s{K}_{G_{2}} f) (s) = \int_{0}^{s}  f(t) \rd t - s \int_{0}^{1} f(t) \rd t, \quad 0 \le s \le 1,
\end{equation*}
which results in
\begin{align*}
    &\c{R}(\s{K}_{G_{1}}) =  \b{H}(K_{1}) = \left\{ f:[0, 1] \rightarrow \b{R} \left. \right\vert f \text{ is absolutely continuous}, \ f(0) = 0, \,  f' \in \c{H} \right\}, \\ 
    &\c{R}(\s{K}_{G_{2}}) =  \b{H}(K_{2}) = \left\{ f:[0, 1] \rightarrow \b{R} \left. \right\vert f \text{ is absolutely continuous}, \ f(0) = f(1) = 0, \,  f' \in \c{H} \right\},
\end{align*}
which are dense due to the Stone-Weierstrass theorem.
It is straightforward that $\s{L}_{1} = \s{K}_{G_{1}}^{\dagger}$ and $\s{L}_{1} = \s{K}_{G_{2}}^{\dagger}$ are just usual differentiations, hence
\begin{align}\label{eq:BM2BB}
    &\s{L}_{2} = \s{L}_{1} \circ \imath_{2 \to 1} : \b{H}(K_{2}) \hookrightarrow \b{H}(K_{1}) \rightarrow \c{H}, \quad \imath_{2 \to 1} : \b{H}(K_{2}) \hookrightarrow \b{H}(K_{1}), f \mapsto f \nonumber \\ 
    &\s{L}_{1} = \s{L}_{2} \circ \imath_{1 \leadsto 2} : \b{H}(K_{1}) \rightarrow \b{H}(K_{2}) \rightarrow \c{H}, \quad \imath_{1 \leadsto 2} : \b{H}(K_{1}) \rightarrow \b{H}(K_{2}), f \mapsto f(\cdot) - \cdot f(1),
\end{align}
to match the boundary conditions. 
Importantly, $\imath_{1 \leadsto 2}$ is not bounded. 
Using integration by parts,
\begin{align*}
    \c{D}(\s{L}_{1}^{*}) = \b{H}(K_{1}), \quad 
    \c{D}(\s{L}_{2}^{*}) = \left\{ f:[0, 1] \rightarrow \b{R} \left. \right\vert f \text{ is absolutely continuous}, \,  f' \in \c{H} \right\}.
\end{align*}
In particular, $\c{D}(\s{L}_{2}^{*})$ has no boundary condition as we have put too many boundary conditions on $\s{L}_{2}$.
\end{example}

\begin{proposition}\label{prop:box:prod:opt}
Let $X_{1}$ and $X_{2}$ be two jointly measurable Gaussian processes with covariance kernels $K_{1}$ and $K_{2}$ and corresponding Green's functions $G_{1}$ and $G_{2}$. Assume that $\c{N}(\s{K}_{K_{1}}) = \{0\}$.  Then the unique optimal pre-pushforward from $X_{1}$ to $X_{2}$ is the s.p.d. operator
\begin{equation*}
    \m{S}_{X_{1} \leadsto X_{2}} := \s{K}_{G_{1}^{*}}^{\dagger} \s{K}_{(G_{1}^{*} \square K_{2} \square G_{1})^{1/2}} \s{K}_{G_{1}}^{\dagger}: \b{H}(K_{1}) \subset \c{H} \rightarrow \c{H}.
\end{equation*}
\end{proposition}

\begin{proof}[Proof of \cref{prop:box:prod:opt}]
Since $(\s{K}_{G})^{*} = \s{K}_{G^{*}}$ and $\s{K}_{G_{1} \square G_{2}} = \s{K}_{G_{1}} \s{K}_{G_{2}}$, the result follows from \cref{thm:sol:injec}.
\end{proof}

Applying this proposition directly can be cumbersome, as shown by the following corollary.
\begin{corollary}\label{cor:comp:Mercer}
For a BM ($X_{1}$) and a BB ($X_{2}$), the optimal pre-pushforwards to one another are
\begin{align*}
    \m{S}_{X_{1} \leadsto X_{2}}: \b{H}(K_{1}) \subset \c{H} \rightarrow \c{H}, \ f \mapsto \left( \s{K}_{K_{121}^{1/2}} f' \right)', \\
    \m{S}_{X_{2} \leadsto X_{1}}: \b{H}(K_{2}) \subset \c{H} \rightarrow \c{H}, \ f \mapsto \left( \s{K}_{K_{212}^{1/2}} f' \right)',
\end{align*}
where the s.p.d. kernels $K_{121}, K_{212} \succeq 0$ are complicated integral kernels
\begin{align*}
    &K_{121}:[0, 1] \times [0, 1] \rightarrow \b{R}, \quad K_{121}(s, t) = \int_{s}^{1} \int_{t}^{1} (\min(u, v) - uv) \rd u \rd v, \\
    &K_{212}:[0, 1] \times [0, 1] \rightarrow \b{R}, \quad K_{212}(s, t) = \int_{0}^{1} \int_{0}^{1} (I_{[s, 1]}(u) -u) \min (u, v) (I_{[t, 1]}(v) - v) \rd u \rd v.
\end{align*}
\end{corollary}

\begin{proof}[Proof of \cref{cor:comp:Mercer}]
Let $G_{1}$ and $G_{2}$ be their associated Green's functions, respectively, as in \cref{ex:BM2BB}. Note that the differntial operators $\s{L}_{i}^{*} = - \s{L}_{i}$ for $i=1, 2$, so it suffices to show that $K_{121} = G_{1}^{*} \square K_{2} \square G_{1}$ and $K_{212} = G_{2}^{*} \square K_{1} \square G_{2}$:
\begin{align*}
    (G_{1}^{*} \square K_{2} \square G_{1})(s, t) &= \int_{0}^{1} \int_{0}^{1} G_{1}(u, s) K_{2}(u, v) G_{1} (v, t) \rd u \rd v \\
    &= \int_{0}^{1} \int_{0}^{1} I_{[s, 1]}(u) K_{2}(u, v) G_{1} I_{[t, 1]}(v) \rd u \rd v 
    = \int_{s}^{1} \int_{t}^{1} (\min(u, v) - uv) \rd u \rd v,
\end{align*}
and similarly,
\begin{align*}
    (G_{2}^{*} \square K_{1} \square G_{2})(s, t) 
    &= \int_{0}^{1} \int_{0}^{1} (I_{[s, 1]}(u) -u) \min (u, v) (I_{[t, 1]}(v) - v) \rd u \rd v.
\end{align*}
\end{proof}

The square roots of the kernels $K_{121}$ and $K_{212}$ are difficult to compute analytically. Fortunately, a simpler method often exists. 

\begin{theorem}\label{thm:joint:Gauss}
Let $X_{1}$ and $X_{2}$ be two jointly measurable Gaussian processes with the Green's functions $G_{1}$ and $G_{2}$.
Assume that $G_{1}^{*} \square G_{2} \succeq 0$ and $\c{N}(\s{K}_{G_{1}}) \subset \c{N}(\s{K}_{G_{2}})$. Then, $X_{1} \leadsto X_{2}$, and an optimal pre-pushforward is given by:
\begin{equation*}
    \m{T}: \b{H}(K_{1}) \rightarrow \b{H}(K_{2}), \, \s{K}_{G_{1}} f \mapsto \s{K}_{G_{2}} f, \quad f \in \c{H}.
\end{equation*}
\end{theorem}

\begin{proof}[Proof of \cref{thm:joint:Gauss}]
Denote by $\m{A} = \s{K}_{K_{1}}$, $\m{G} = \s{K}_{G_{1}}$, and $\m{B} = \s{K}_{K_{2}}$. 
Then $\m{T} \m{G} = \s{K}_{G_{2}}$ satisfies \eqref{eq:cond:ricatti} and \eqref{eq:cond2:opt}:
\begin{align*}
    (\m{T} \m{G}) (\m{T} \m{G})^{*} = \s{K}_{G_{2}} \s{K}_{G_{2}}^{*} = \s{K}_{K_{2}} = \m{B}, \quad
    \m{G}^{*} (\m{T} \m{G}) = \s{K}_{G_{1}^{*}} \s{K}_{G_{2}} = \s{K}_{G_{1}^{*} \square G_{2}} \succeq \m{0},
\end{align*}
thus $\m{T}$ is an optimal pre-pushforward.
\end{proof}
Provided that $\c{N}(\s{K}_{G_{1}}) = \c{N}(\s{K}_{G_{2}})$, this procedure automatically generates an optimal pre-pushforward in the other direction, as $G_{1}^{*} \square G_{2} \succeq 0$ implies $G_{2}^{*} \square G_{1} \succeq 0$:

\begin{example}[Integrated Brownian Motion]
Let $\c{H} = L_{2}[0,1]$ and consider the differential operator $\s{L}_{2} f = f ''$ with boundary conditions $f(0) = f'(0) = 0$. The corresponding Gaussian process is called the Integrated Brownian Motion (IBM), denoted $X_{2}$, which is jointly measurable. Its Green's function $G_{2}$ 
generates the RKHS:
\begin{align*}
    &\c{R}(\s{K}_{G_{2}}) =  \b{H}(K_{2}) = \left\{ f:[0, 1] \rightarrow \b{R} \left. \right\vert f \text{ is absolutely continuous}, \ f(0) = f'(0) = 0, \,  f', f'' \in \c{H} \right\},
\end{align*}
which are dense. 

Let $X_{1}$ be a BB and $X_{2}$ be an IBM. The optimal pre-pushforward from $X_{1}$ to $X_{2}$ and vice versa are unique since $\c{N}(\s{K}_{G_{1}}) = \c{N}(\s{K}_{G_{2}}) = \{0\}$ due to \cref{thm:bdd:reach}. Also, \cref{thm:bdd:reach} shows that $X_{1} \to X_{2}$ as $\c{R}(\s{K}_{G_{2}}) \subset \c{R}(\s{K}_{G_{1}})$, while $X_{2} \leadsto X_{1}$ and $X_{2} \nrightarrow X_{1}$. In this case, denoting its definite integral by $F(t) = \int_{0}^{t} f(s) \rd s$, the condition $G_{1}^{*} \square G_{2} \succeq 0$ can be verified easily:
\begin{align*}
    \innpr{h}{\s{K}_{G_{1}^{*}} \s{K}_{G_{2}} h} 
    &= \innpr{\s{K}_{G_{1}} h}{\s{K}_{G_{2}} h} = \innpr{F}{\int_{0}^{\cdot} F(s) \rd s} \\
    &= \int_{0}^{1} F(t) \left( \int_{0}^{t} F(s) \rd s  \right) \rd t = \frac{1}{2} \left( \int_{0}^{1} F(s) \rd s  \right)^{2} \ge 0, \quad f \in \c{H}.
\end{align*}
Therefore, the unique optimal pre-pushforwards are given by \cref{thm:joint:Gauss}:
\begin{align*}
    &\m{T}_{X_{1} \leadsto X_{2}}: \b{H}(K_{1}) \rightarrow \c{H}, \, \s{K}_{G_{1}} f \mapsto \s{K}_{G_{2}} f, \quad f \in \c{H}, \\
    &\m{T}_{X_{2} \leadsto X_{1}}: \b{H}(K_{2}) \rightarrow \c{H}, \, \s{K}_{G_{1}} f \mapsto \s{K}_{G_{2}} f, \quad f \in \c{H},
\end{align*}
which are are simply integration and differentiation upon extension:
\begin{align*}
    \tilde{\m{T}}_{X_{1} \to X_{2}} = \s{K}_{G_{1}}: \c{H} \rightarrow \c{H}, \, f \mapsto \int_{0}^{\cdot} f(s) \rd s, \quad
    \tilde{\m{T}}_{X_{2} \leadsto X_{1}} = \s{L}_{1}: \b{H}(K_{1}) \rightarrow \c{H}, \, f \mapsto f'.
\end{align*} 
\end{example}

This reasoning extends naturally to higher-order Integrated Brownian Motions (IBMs). We identify that optimal pre-pushforwards between IBMs correspond to multiple integration and differentiation operators. To the best of our knowledge, this constitutes the first result providing an explicit value for the 2-Wasserstein distance between infinite-dimensional Gaussian measures.

\begin{corollary}\label{cor:IBM:prepush}
Let $\c{H} = L_{2}[0,1]$, and for each $n \in \b{N}$, consider a differential opertaor $\s{L}_{n} f = f ^{(n)}$ with boundary conditions $f(0) = f'(0) = \cdots = f^{(n-1)}(0) = 0$, and $X_{n}$ be the associated Gaussian process, called the $n$-th IBM, which is jointly measurable. 
Then for any $n \le m$, the optimal pre-pushforwards between $n$-th and $m$-th IBMs are given by $(m-n)$-fold integration and differentiation:
\begin{align*}
    \m{T}_{X_{n} \to X_{m}} = \s{K}_{G_{m-n}} \in \c{B}_{\infty}(\c{H}), \quad
    \m{T}_{X_{m} \leadsto X_{n}} = \s{L}_{m-n}: \b{H}(K_{m-n}) \subset \c{H} \rightarrow \c{H},
\end{align*}
Hence, $X_{n} \to X_{m}$, $X_{m} \leadsto X_{n}$, and $X_{m} \nrightarrow X_{n}$. Moreover, the squared 2-Wasserstein distance between them are given by
\begin{align*}
    \c{W}_{2}^{2}(X_{n}, X_{m}) = 
    \begin{cases}
        C_{n} + C_{m} - 2 C_{(n+m)/2} &, \quad n+m \text{ is even} \\
        C_{n} + C_{m} - C_{(n+m+1)/2} &, \quad n+m \text{ is odd}
    \end{cases}, \quad C_{n} = \frac{1}{2n(2n-1)[(n-1)!]^{2}}. 
\end{align*}
\end{corollary}

\begin{proof}[Proof of \cref{cor:IBM:prepush}]
By the Stone-Weierstrasse theorem, $\c{N}(\s{K}_{G_{n}}) = \{0\}$ for any $n \in \b{N}$. Hence, it suffices to show that $\s{K}_{G_{n}^{*}} \s{K}_{G_{m}} \succeq 0$. For any $h \in \c{H}$, if $n+m$ is even, $\innpr{h}{\s{K}_{G_{n}^{*}} \s{K}_{G_{m}} h} = \innpr{\s{K}_{G_{n}} h}{\s{K}_{G_{m}} h} = \left\| \s{K}_{G_{(n+m)/2}} h \right\|^{2} \ge 0$, and if $n+m$ is odd,
\begin{align*}
    \innpr{h}{\s{K}_{G_{n}^{*}} \s{K}_{G_{m}} h} 
    &= \innpr{\s{K}_{G_{n}} h}{\s{K}_{G_{m}} h} = \innpr{\s{K}_{G_{(n+m-1)/2}} h}{\int_{0}^{\cdot} (\s{K}_{G_{(n+m-1)/2}} h)(s) \rd s} \\
    &= \frac{1}{2} (\s{K}_{G_{(n+m+1)/2}} h)(1)^{2} \ge 0,
\end{align*}
which proves the result. Moving forward, $\s{K}_{G_{n}}$ is the $n$-th order Volterra operator \cite{gohberg1978introduction, paulsen2016introduction}:
\begin{align*}
    \s{K}_{G_{n}} f (s) = \int_{0}^{s} \frac{(s-t)^{n-1}}{(n-1)!} f(t) \rd t,
\end{align*}
hence
\begin{align*}
    C_{n} := \vertii{\s{K}_{G_{n}}}_{2}^{2} = \int_{0}^{1} \int_{0}^{s} \left| \frac{(s-t)^{n-1}}{(n-1)!} \right|^{2} \rd t \rd s = \frac{1}{2n(2n-1)[(n-1)!]^{2}}.
\end{align*}
Consequently, we obtain
\begin{align*}
    \c{W}_{2}^{2}(X_{n}, X_{m}) = \vertii{\s{K}_{G_{n}}}_{2}^{2} + \vertii{\s{K}_{G_{m}}}_{2}^{2} - 2 \vertii{\s{K}_{G_{n}^{*}} \s{K}_{G_{m}}}_{1} = 
    \begin{cases}
        C_{n} + C_{m} - 2 C_{(n+m)/2} &, \quad n+m \text{ is even}, \\
        C_{n} + C_{m} - C_{(n+m+1)/2} &, \quad n+m \text{ is odd}.
    \end{cases} 
\end{align*}
\end{proof}

Indeed, $\m{T}_{X_{n} \to X_{m}} = \s{K}_{G_{m-n}}$ is bounded, thus a valid pushforward. In contrast, the differential operator $\s{L}_{m-n}$ is only closable, hence one may have to close it to make it as a valid pushforward in light of \cref{thm:prepush}.
However, this simple approach has a pitfall: it fails when the boundary conditions of the processes are incompatible. For the BM and BB in \cref{ex:BM2BB}, $(G_{1}^{*} \square G_{2})$ is not symmetric, let alone s.p.d., \cref{thm:joint:Gauss} cannot be used:
\begin{align*}
    (G_{1}^{*} \square G_{2}) (s, t) &= \int_{0}^{1} G_{1}(r, s) G_{2} (r, t) \rd r = \int_{0}^{1} I_{[s, 1]}(r) (I_{[t, 1]}(r) - r) \rd r \\
    &= \int_{0}^{1} [I_{[s \vee t, 1]}(r) - r I_{[s, 1]}(r)] \rd r
    = 1- \max(s, t) - \frac{1-s^{2}}{2} \neq (G_{1}^{*} \square G_{2}) (t, s).
\end{align*}

\begin{acks}[Acknowledgments]
We are grateful to Victor Panaretos for his valuable feedback. His insightful discussions and pointers to key literature were instrumental in enhancing the quality of this work.
\end{acks}

\bibliographystyle{imsart-number}
\bibliography{bibliography}  

\newpage

\begin{appendix}

\section{Additional Results}
\begin{lemma}\label{lem:Neumann:eq}
Let $\c{H}$ be a separable Hilbert space, and $\m{R} \in \c{B}_{1}(\c{H})$ be a trace-class operator. Then 
\begin{equation*}
    \tr(\m{R}) \le \tr((\m{R}\m{R}^{*})^{1/2}) = \tr((\m{R}^{*}\m{R})^{1/2}),
\end{equation*}
and equality holds if and only $\m{R} \succeq \m{0}$, i.e.\ $\m R=(\m R\m R^*)^{1/2}$.
\end{lemma}
\begin{proof}
Consider the singular value decompostion of $\m{R}$:
\begin{equation*}
    \m{R} = \sum_{k=1}^{\infty} \lambda_{j} f_{j} \otimes h_{j},
\end{equation*}
where $\lambda_{1} \ge \lambda_{2} \ge \dots > 0$ are non-zero singular values, and $(f_{j})$ are orthonormal eigenvectors of $\m{R}\m{R}^{*}$, and $(h_{j})$ are orthonormal eigenvectors of $\m{R}^{*} \m{R}$. Then,
\begin{align*}
    \tr(\m{R}) = \sum_{k=1}^{\infty} \innpr{h_{j}}{\m{R} h_{j}} = \sum_{j=1}^{\infty} \lambda_{j} \innpr{f_{j}}{h_{j}} \le \sum_{j=1}^{\infty} \lambda_{j} = \tr((\m{R}\m{R}^{*})^{1/2}),
\end{align*}
and the equality holds if and only if $\lambda_{j} \ge 0$ and $f_{j} \equiv h_{j}$ for all $j \in \b{N}$, i.e. $\m{R} \succeq \m{0}$
\cite{hsing2015theoretical}.
\end{proof}

\begin{lemma}\label{lem:Schur:root}
Let $\m{C} \in \b{R}^{n \times n}$ be a s.p.d. block matrix:
\begin{equation*}
    \m{C} = \begin{pmatrix}
    \m{C}_{11} & \m{C}_{12} \\
    \m{C}_{21} & \m{C}_{22}
    \end{pmatrix},
\end{equation*}
where $\m{C}_{11} \in \b{R}^{n_{1} \times n_{1}}$ with $n_{1} \in \b{N}$.
If $\rk (\m{C}) \le n_{1}$, there exists a $\m{N}_{12} \in \b{R}^{n_{1} \times n_{2}}$ such that $\m{N}_{12} \in \c{N}(\m{C}_{11})$ (i.e. $\m{C}_{11} \m{N}_{12} = \m{0}$), and $\m{N}_{12}^{*} \m{N}_{12} = \m{C}/\m{I}_{11}$. Additionally, the solution $\m{N}_{12}$ is unique if and only if $\rk (\m{C}) = \rk (\m{C}_{11})$.
\end{lemma}
\begin{proof}
Let $d := \rk(\m{C}/\m{I}_{11}) = \rk(\m{C}) - \rk(\m{C}_{11})$. We may assume $d \ge 1$, otherwise trivial. Note that
\begin{equation*}
    \dim \c{N}(\m{C}_{11}) = n_{1} - \rk(\m{C}_{11}) \ge \rk(\m{C}) - \rk(\m{C}_{11}) = d,
\end{equation*}
thus there are some unit vectors $\m{v}_{1}, \dots, \m{v}_{d} \in \c{N}(\m{C}_{11}) \subset \b{R}^{n_{1}}$ that are mutually orthogonal. Consider the spectral decomposition of the Schur complement:
\begin{equation*}
    \m{C}/\m{I}_{11} = \sum_{k=1}^{d} \lambda_{k} \m{u}_{k} \m{u}_{k}^{\top} \succeq \m{0},
\end{equation*}
where $\lambda_{1} \ge \dots \ge \lambda_{d} > 0$, and $\m{u}_{1}, \dots, \m{u}_{d} \in \b{R}^{n_{2}}$ are corresponding unit eigenvectors. Then
\begin{equation*}
    \m{N}_{12} := \sum_{k=1}^{d} \sqrt{\lambda_{k}} \m{v}_{k} \m{u}_{k}^{\top} \neq \m{0},
\end{equation*}
is a desired solution, and $- \m{N}_{12}$ is also another solution.
\end{proof}

More generally, the solution set in \cref{lem:Schur:root} is given by
\begin{equation*}
    \displaystyle
    \m{N}_{12} = \left\{\m{U}_{12} (\m{C}/\m{I}_{11})^{1/2} \mid \m{U}_{12} \text{ is a partial isometry from $\c{R}(\m{C}/\m{I}_{11})$ to $\mathcal{V}$} \right\},
\end{equation*}
where $\mathcal{V}$ is a subspace of $\c{N}(\m{C}_{11})$ of $\dim \mathcal{V} =  \rk(\m{C}/\m{I}_{11})$, i.e. $\m{U}_{12}^{*} \m{U}_{12} = \m{\Pi}_{\c{R}(\m{C}/\m{I}_{11})}$ and $\m{U}_{12} \m{U}_{12}^{*} = \m{\Pi}_{\mathcal{V}}$.

\begin{lemma}\label{lem:spd:complete}
Consider a $n \times n$ block matrix
\begin{equation*}
    \m{S} = \begin{pmatrix}
    \m{S}_{11} & \m{S}_{12} \\
    \m{S}_{21} & \ast
    \end{pmatrix},
\end{equation*}
where $\ast$ is an unspecified entry. Then, $\m{S}$ admits a s.p.d. completion if and only if $\m{S}_{11} \succeq \m{0}$ and $\m{S}_{12} = \m{S}_{21}^{*} \in \c{R}(\m{S}_{11})$. 
\end{lemma}
\begin{proof}
The necessity follows directly from the Douglas lemma \cite{douglas1966majorization}. For sufficiency, choosing $\ast = \m{S}_{21} \m{S}_{11}^{\dagger} \m{S}_{12}$ ensures that $\m{S}$ is s.p.d. due to \eqref{eq:LDU2}.
\end{proof}

\begin{lemma}\label{lem:closable:half}
Let $\c{H}$ be a separable Hilbert space, $\m{A} \in \c{B}_{0}^{+}(\c{H})$ be a s.p.d. compact operator, and $\m{T}$ be an unbounded operator. Then, the following holds:
\begin{enumerate}
\item $\m{A}^{\dagger/2} \m{T}^{*} \subset (\m{T} \m{A}^{\dagger/2})^{*}$.
\item If $\m{T} \in \c{B}_{\infty}(\c{H})$ is bounded, then $\m{A}^{\dagger/2} \m{T}^{*} = (\m{T} \m{A}^{\dagger/2})^{*}$ is a closed operator.
\end{enumerate}
\end{lemma}
\begin{proof}
\begin{enumerate}[leftmargin=*]
\item Fix any $f \in \c{D}[(\m{A}^{\dagger/2} \m{T}^{*})]$. Then $\m{T}^{*} f \in \c{D}(\m{A}^{\dagger/2}) = \c{R}(\m{A}^{1/2}) \oplus \c{N}(\m{A}^{1/2})$, hence there exist unique $k = \m{A}^{\dagger/2} \m{T}^{*} f \in \overline{\c{R}(\m{A}^{1/2})}$ and $n \in \c{N}(\m{A}^{1/2})$ such that 
\begin{equation*}
    \m{T}^{*} f = \m{A}^{1/2} k + n.
\end{equation*}
It remains to show that for any $g \in \c{D}[(\m{T} \m{A}^{\dagger/2})] \subset \c{D}(\m{A}^{\dagger/2})$, it holds that $\innpr{f}{(\m{T} \m{A}^{\dagger/2}) g} = \innpr{k}{g}$.
Let $h = \m{A}^{\dagger/2} g \in \overline{\c{R}(\m{A}^{1/2})} \cap \c{D}(\m{T})$. Then,
\begin{equation*}
    \innpr{f}{\m{T} \m{A}^{\dagger/2} g} = \innpr{f}{\m{T} h} = \innpr{\m{T}^{*} f}{h} = \innpr{\m{A}^{1/2} k + n}{h} = \innpr{\m{A}^{1/2} k}{h} = \innpr{k}{\m{A}^{1/2} h} = \innpr{k}{g},
\end{equation*}
since $k \in \overline{\c{R}(\m{A}^{1/2})}$ and $g - \m{A}^{1/2} h \in \c{N}(\m{A}^{1/2})$ are orthogonal.

\item 
Due to the previous argument, it suffices to show that $\c{D}[(\m{T} \m{A}^{\dagger/2})^{*}] \subset \c{D}(\m{A}^{\dagger/2} \m{T}^{*})$, which is equivalent to $\m{T}^{*} h \in \c{D}(\m{A}^{\dagger/2})$ for any fixed $h \in \c{D}((\m{T} \m{A}^{\dagger/2})^{*})$. 
Let $g = (\m{T} \m{A}^{\dagger/2})^{*} h  \in \c{H}$. Now, for any $k = \m{A}^{\dagger/2} \m{A}^{1/2} k \in \overline{\c{R}(\m{A}^{1/2})}$, it holds that
\begin{align*}
    \innpr{k}{\m{A}^{1/2} g} = \innpr{\m{A}^{1/2} k}{g} &= \innpr{\m{A}^{1/2} k}{(\m{T} \m{A}^{\dagger/2})^{*} h} \\
    &= \innpr{ (\m{T} \m{A}^{\dagger/2}) \m{A}^{1/2} k}{h} = \innpr{\m{T} k}{h} = \innpr{k}{\m{T}^{*} h},
\end{align*}
since $\c{R}(\m{A}^{1/2}) \subset \c{D}(\m{A}^{\dagger/2}) = \c{D}(\m{T} \m{A}^{\dagger/2})$. This yields that 
\begin{equation*}
    \m{T}^{*} h \in \m{A}^{1/2} g \oplus \overline{\c{R}(\m{A}^{1/2})}^{\perp} \subset \c{R}(\m{A}^{1/2}) \oplus \c{N}(\m{A}^{1/2}) = \c{D}(\m{A}^{\dagger/2}).
\end{equation*}
Finally, note that any adjoint operator is closed, which completes the proof.
\end{enumerate}
\end{proof}

We remark that, even in the case where $\m{T} \in \c{B}_{\infty}(\c{H})$, the operator $\m{A}^{\dagger/2} \m{T}^{*}$ may not be even densely defined in $\c{H}$, which is equivalent to the fact that $(\m{T} \m{A}^{\dagger/2})$ is not closable.

\begin{corollary}[Hierarchical Barycenter]\label{cor:hier:bary}
Let $p_{1}, \cdots, p_{m} > 0$ be the weights with $p_{1} + \cdots + p_{m} = 1$. For each $i = 1, 2, \cdots, m$, let
$\m{A}_{i 1}, \cdots, \m{A}_{i r_{i}}  \in \b{R}^{n \times n}$ be s.p.d. matrices, and $q_{1}, \cdots, q_{r_{i}} > 0$ be the sub-weights with $q_{1} + \cdots + q_{r_{i}} = 1$. Denote the set of the associated optimal weighted Green's matrices by
\begin{align*}
    \c{G}_{i} = \left\{(\m{G}_{i 1}, \cdots, \m{G}_{i r_{i}}) \in \argmax_{\m{S}_{ij} \in \c{G}(\m{A}_{ij})} \vertiii{\sum_{j=1}^{r_{i}} q_{ij} \m{S}_{ij}}_{2}^{2} \right\}, \quad i = 1, \cdots, m. 
\end{align*}

Assume $\m{A}_{ij} \m{A}_{i'j'} = \m{0}$ whenever $i \ne i'$. Then, the set of doubly weighted Wasserstein barycenters are given by
\begin{small}
\begin{equation*}
    \argmin_{\m{B} \in \b{R}^{n \times n}, \ \m{B} \succeq \m{0}}  \sum_{i=1}^{m} p_{i} \sum_{j=1}^{r_{i}} q_{ij} \c{W}_{2}^{2}(\m{A}_{i}, \m{B}) = \left\{ \hat{\m{G}} \hat{\m{G}}^{*} : \hat{\m{G}} = \sum_{i=1}^{m} p_{i} \sum_{j=1}^{r_{i}} q_{ij} \m{G}_{ij}, \, (\m{G}_{i 1}, \cdots, \m{G}_{i r_{i}}) \in \c{G}_{i}, \, i =1, \cdots, m \right\}.
\end{equation*}    
\end{small}
\end{corollary}
\begin{proof}[Proof of \cref{cor:hier:bary}]
From the proof of \cref{cor:ortho:Green's}, we have
\begin{align*}
    \vertiii{\sum_{i=1}^{m} p_{i} \sum_{j=1}^{r_{i}} q_{ij} \m{G}_{ij}}_{2}^{2} = \sum_{i=1}^{m} p_{i}^{2} \vertiii{\sum_{j=1}^{r_{i}} q_{ij} \m{S}_{ij}}_{2}^{2},
\end{align*}
so the result follows from \cref{thm:char:barycenter}.
\end{proof}








\end{appendix}
\end{document}